\documentclass[10pt]{article}

\usepackage{amsthm}
\usepackage[section]{algorithm}
\usepackage[title]{appendix}
\usepackage[affil-it]{authblk}
\usepackage{fancyhdr}
\usepackage{graphicx}
\usepackage{ifpdf}
\usepackage{xcolor}
\usepackage{hyperref}
\usepackage[differential]{boldtensors}

\usepackage{algorithm}
\usepackage[noend]{algpseudocode}

\usepackage{amssymb}
\usepackage{booktabs}
\usepackage{mathtools}
\usepackage{setspace}
\usepackage{ulem}
\usepackage{wrapfig}
\usepackage{xcolor}
\usepackage{multirow}
\usepackage{minted}

\setminted{
    style=pastie,
    fontsize=\small,
    linenos,
    frame=none
}

\newcommand{\lr}[1]{\left(#1\right)}
\newcommand{\nn}[1]{\left|#1\right|}
\newcommand{\dual}[2]{\left\langle #1\,|\,#2\right\rangle}

\newcommand{\ip}[2]{\left\langle #1,#2\right\rangle}

\DeclareMathOperator*{\argmin}{argmin}

\let\originalleft\left
\let\originalright\right
\def\left#1{\mathopen{}\originalleft#1}
\def\right#1{\originalright#1\mathclose{}}  
\usepackage[nameinlink]{cleveref}

\usepackage{minted}
\usepackage{siunitx}

\newcommand{\figOVP}{
\begin{figure}[tbp] 
\begin{center}
    \includegraphics[width=\textwidth]{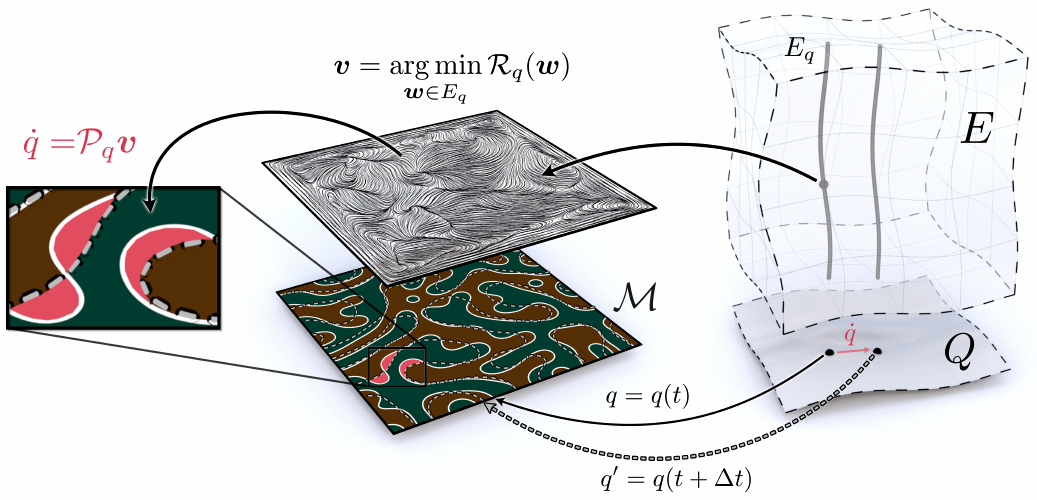}
    \caption{An illustration depicting Onsager's Variational Principle \Cref{def:OVP}. Here, configurations $q\in Q$, representing scalar fields over a manifold $\mathcal{M}$, are updated based on optimal elements $~w\in E_{q}$ in the fibers of a vector bundle $E\rightarrow Q$, which are vector fields on $\mathcal{M}$ depending on the configuration $q$. 
    }
    \end{center}
    \label{fig:OVP}
\end{figure}
}

\newcommand{\figFrames}{
\begin{figure}[htbp]
\centering
\includegraphics[width=\textwidth]{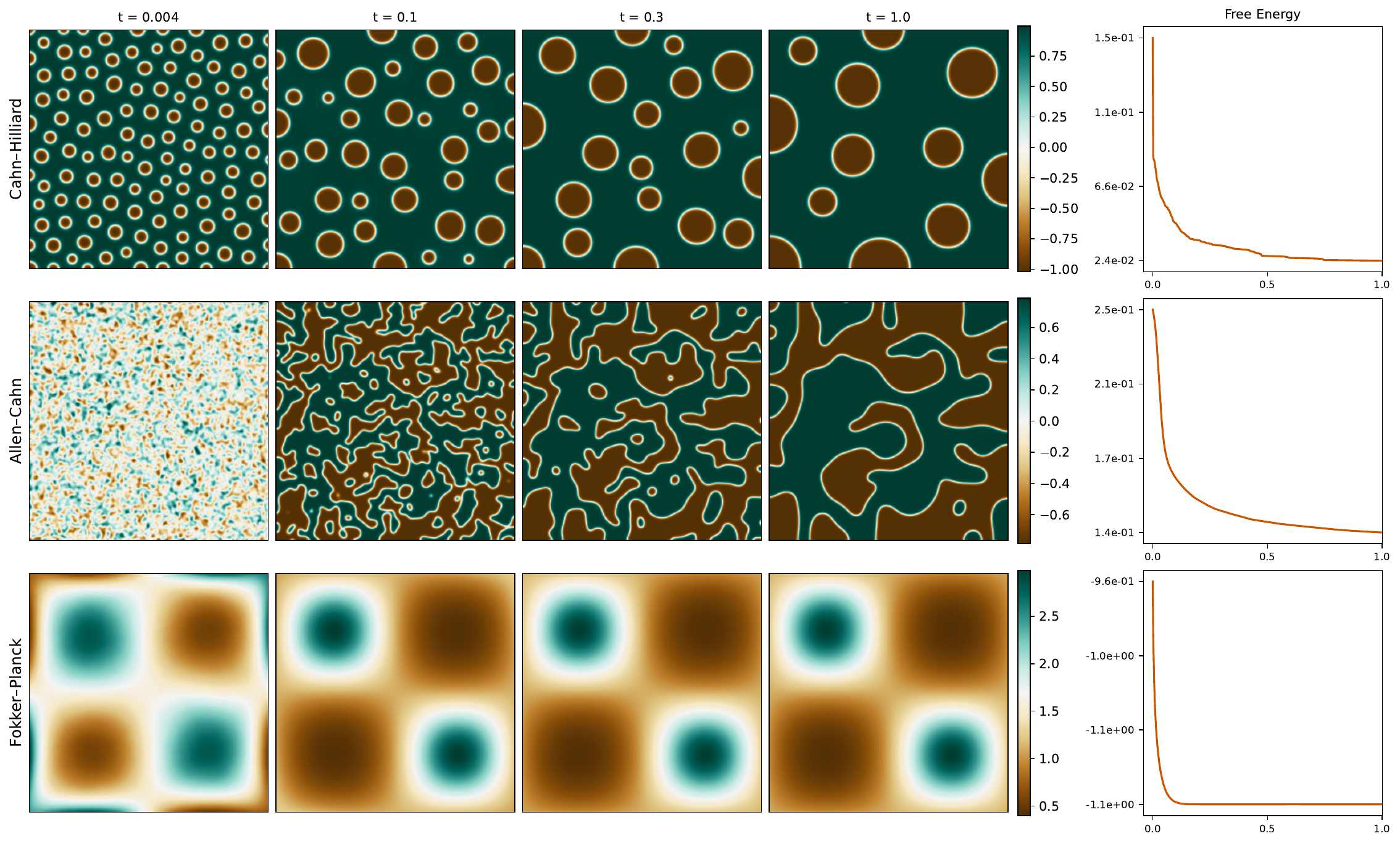}
    \caption{Snapshots of the Cahn-Hilliard (CH), Allen-Cahn (AC), and Fokker-Planck (FP) equations, computed using the OVP integrator \Cref{alg:OVP} with potential $U(x, y) = \sin 3x \sin 3y$ (for FP) and density $F(c) = (1/4)(1-c^2)^2$ (for CH). The initial condition for all simulations is the same Gaussian noise pattern (mean-shifted by 0.5 for CH and 1.25 for FP) and $\Delta t = 0.002$.  Concentration fields are row-normalized for visualization.}
    \label{fig:frames}
\end{figure}
}

\newcommand{\figDiffusion}{
\begin{figure}[htbp]
\centering
\includegraphics[width=\textwidth]{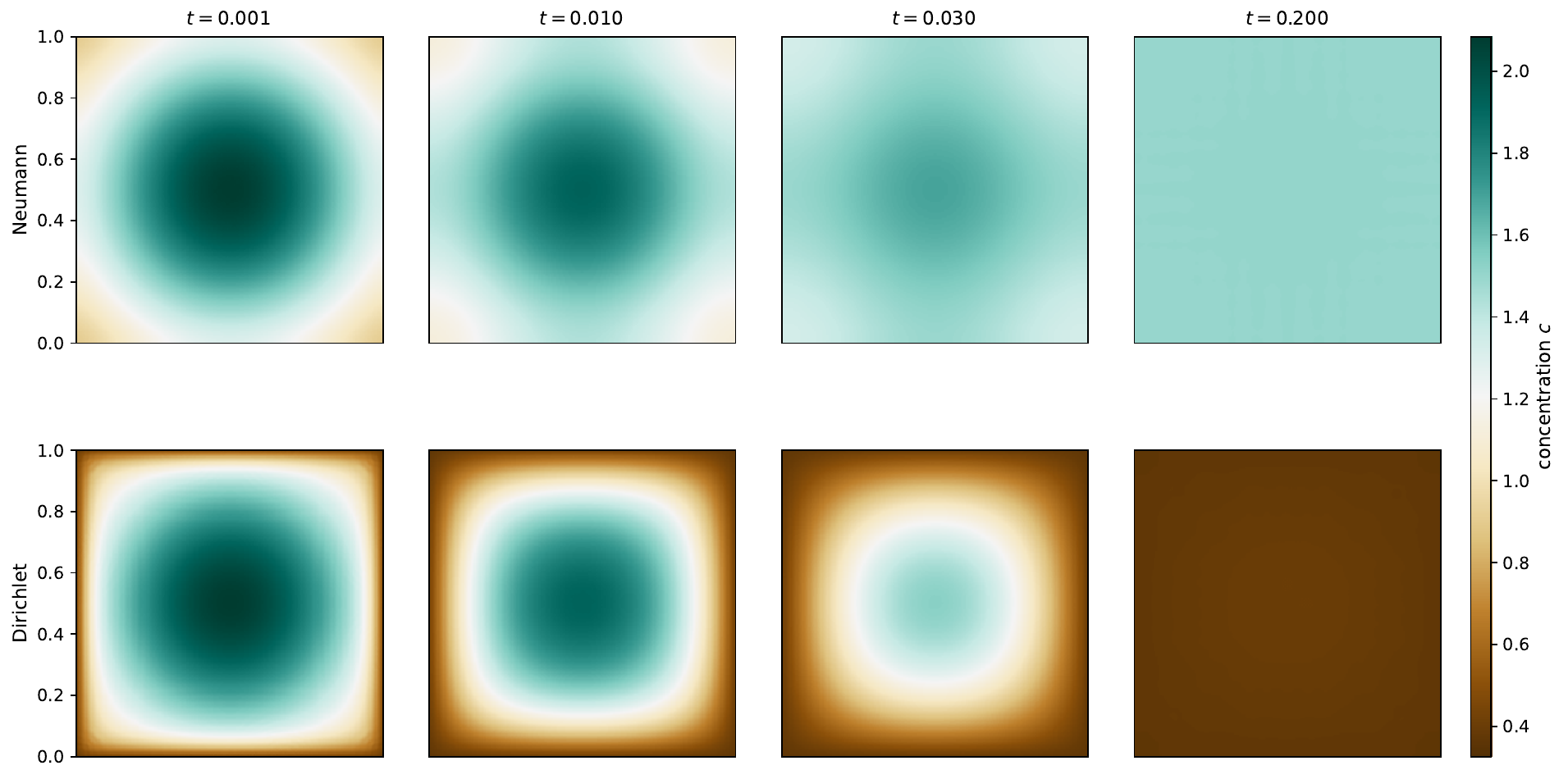}
    \caption{OVP-based diffusion with (top) and without (bottom) the essential condition $~w\cdot~n=0$ on $\partial\mathcal{M}$, implying natural Neumann and absorbing Dirichlet BCs, respectively.  Both plots use $\mu_0=1$, so that the Dirichlet BC is $e^{-1}$.}
    \label{fig:dirichlet-vs-neumann}
\end{figure}
}

\newcommand{\figCahnHilliard}{
\begin{figure}[htbp]
\centering
\includegraphics[width=\textwidth]{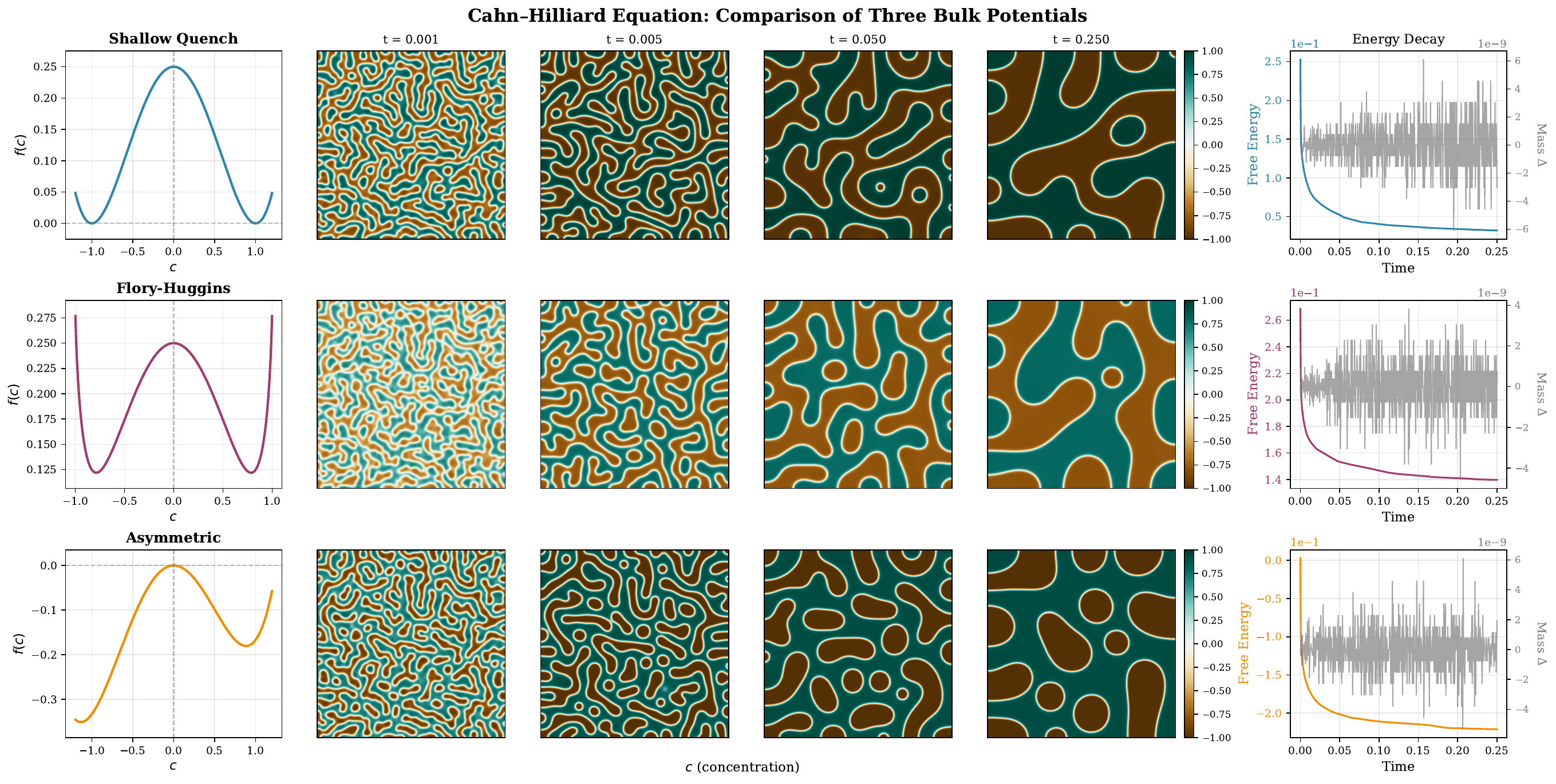}
    \caption{Three representative Cahn-Hilliard evolutions, differing in their bulk free energy density $f(c)$.  Note the monotonic free energy decay and conservation of mass up to the single precision used in these calculations.}
    \label{fig:threeCahnHilliard}
\end{figure}
}

\newcommand{\figFEloss}{
\begin{figure}[htbp]
\centering
\includegraphics[width=0.8\textwidth]{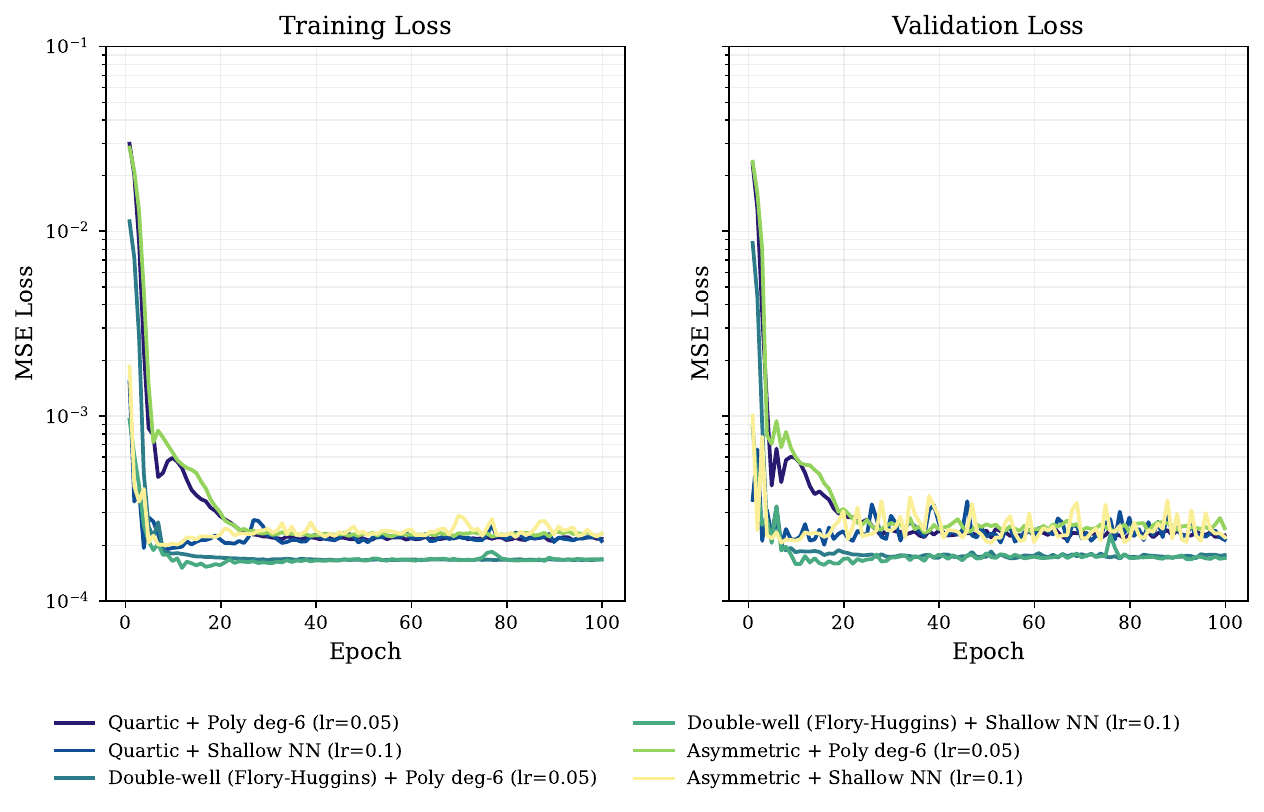}
    \caption{Training and validation loss (mean squared error) curves illustrating the training process in the case of the bulk free energy density recovery experiment.}
    \label{fig:61training}
\end{figure}
}

\newcommand{\figFEgrid}{
\begin{figure}[htbp]
\centering
\includegraphics[width=\textwidth]{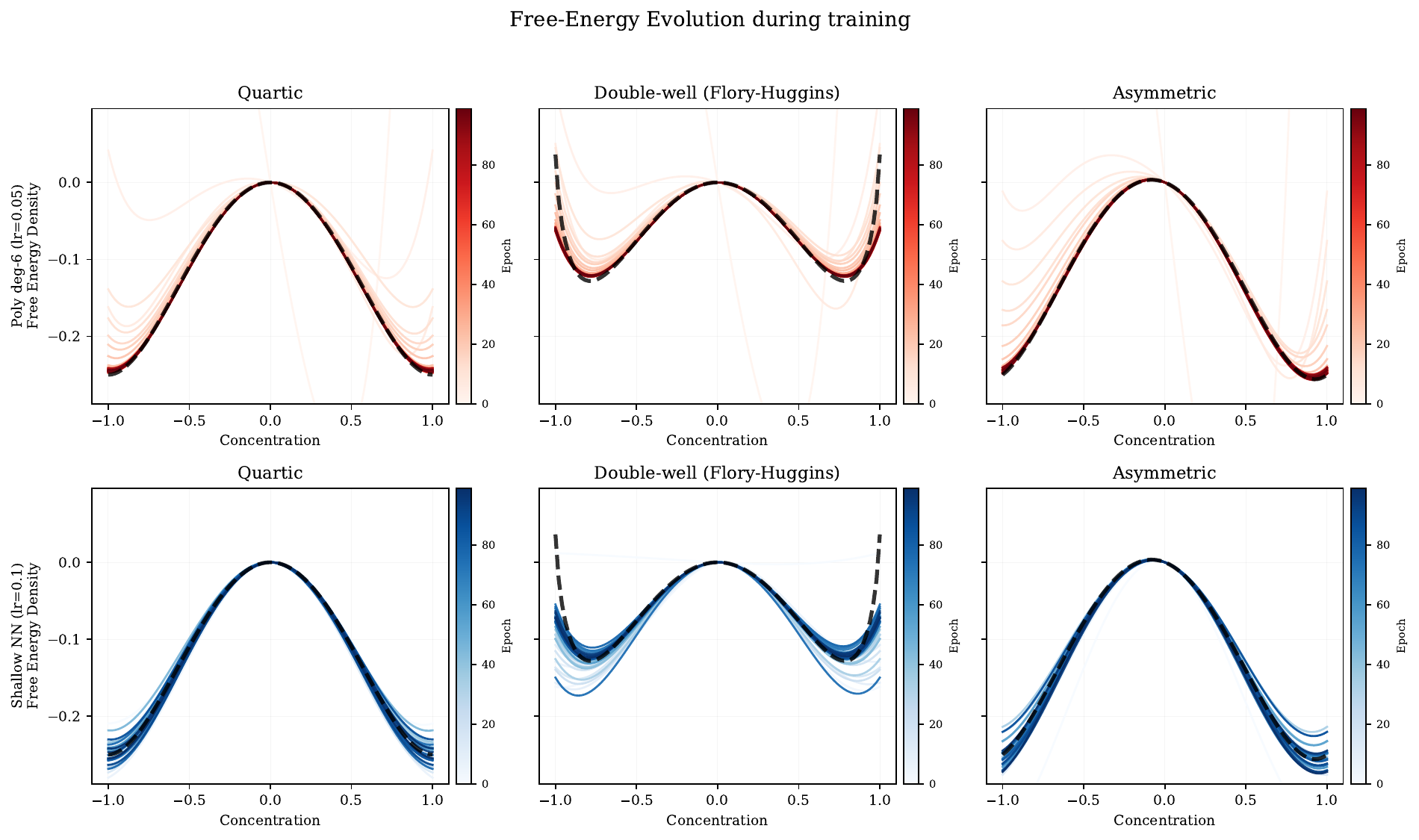}
    \caption{Evolution of the learned free energies $F_{\theta}$ over the duration of training, for the polynomial (red) and neural network (blue) parameterizations.  The passage of training epochs is illustrated with a gradient, with later epochs corresponding to darker curves. The ground truth is overlaid with a dashed line in each case.}
    \label{fig:61grid}
\end{figure}
}

\newcommand{\figLearnedEvolutions}{
\begin{figure}[htbp]
    \centering
    \includegraphics[width=0.85\textwidth]{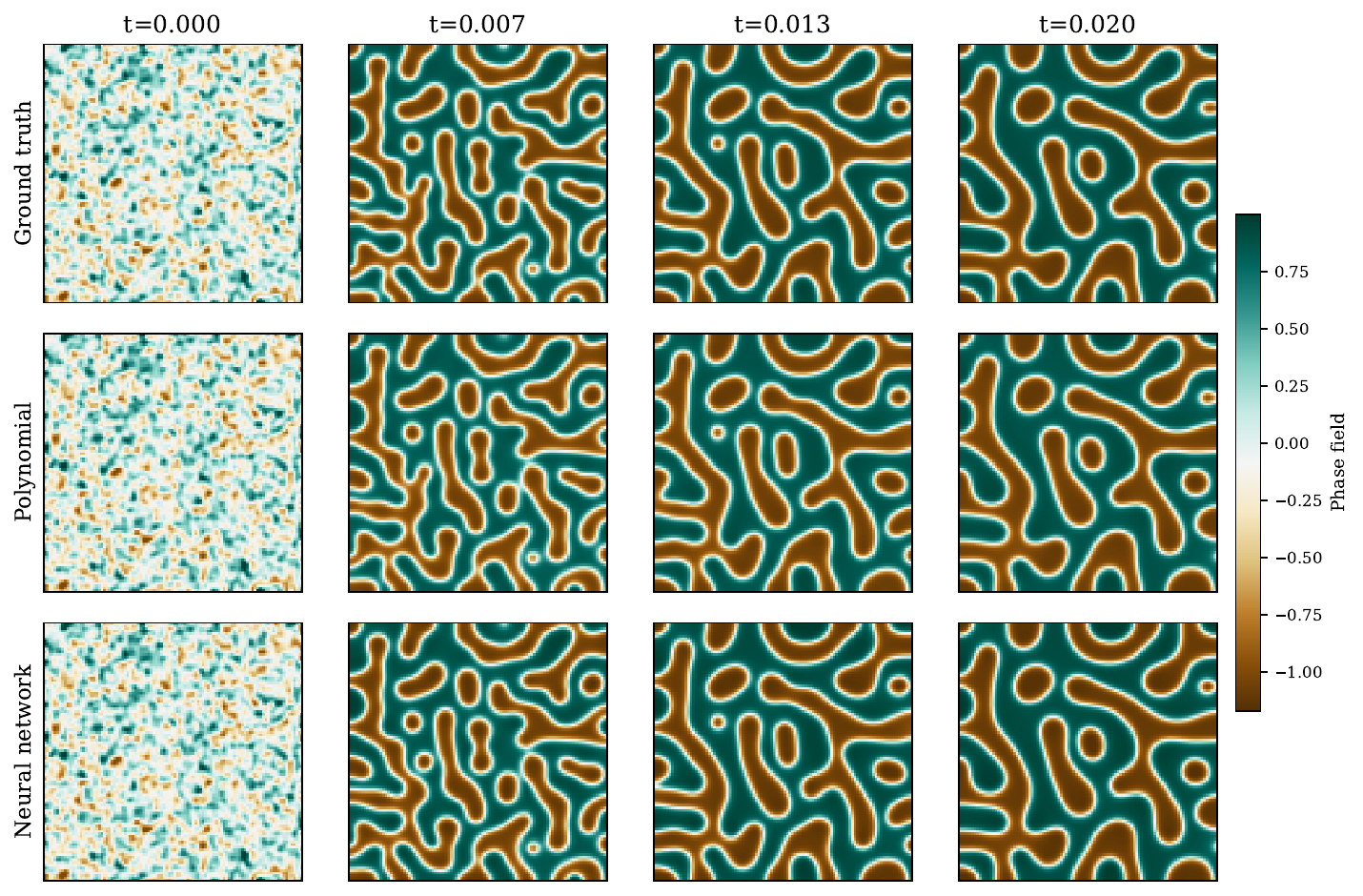}
    \caption{Evolution of the learned Cahn-Hilliard system in the asymmetric case, starting from an initial condition in the validation set. Pictured is the ground truth evolution (top) along with the evolution of the learned polynomial (middle) and neural network (bottom) surrogates.}
    \label{fig:learned_evolutions}
\end{figure}
}

\newcommand{\figBClossdens}{
\begin{figure}[htbp]
    \centering
    \begin{minipage}[t]{0.32\textwidth}
        \centering
        \includegraphics[width=\textwidth]{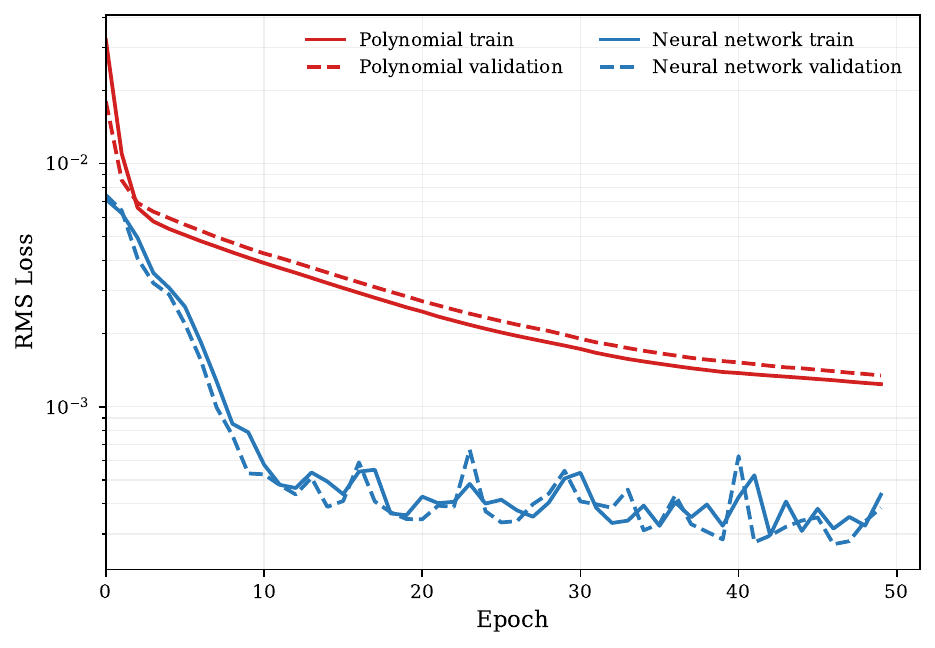}
    \end{minipage}
    \hfill
    \begin{minipage}[t]{0.64\textwidth}
        \centering
        \includegraphics[width=\textwidth]{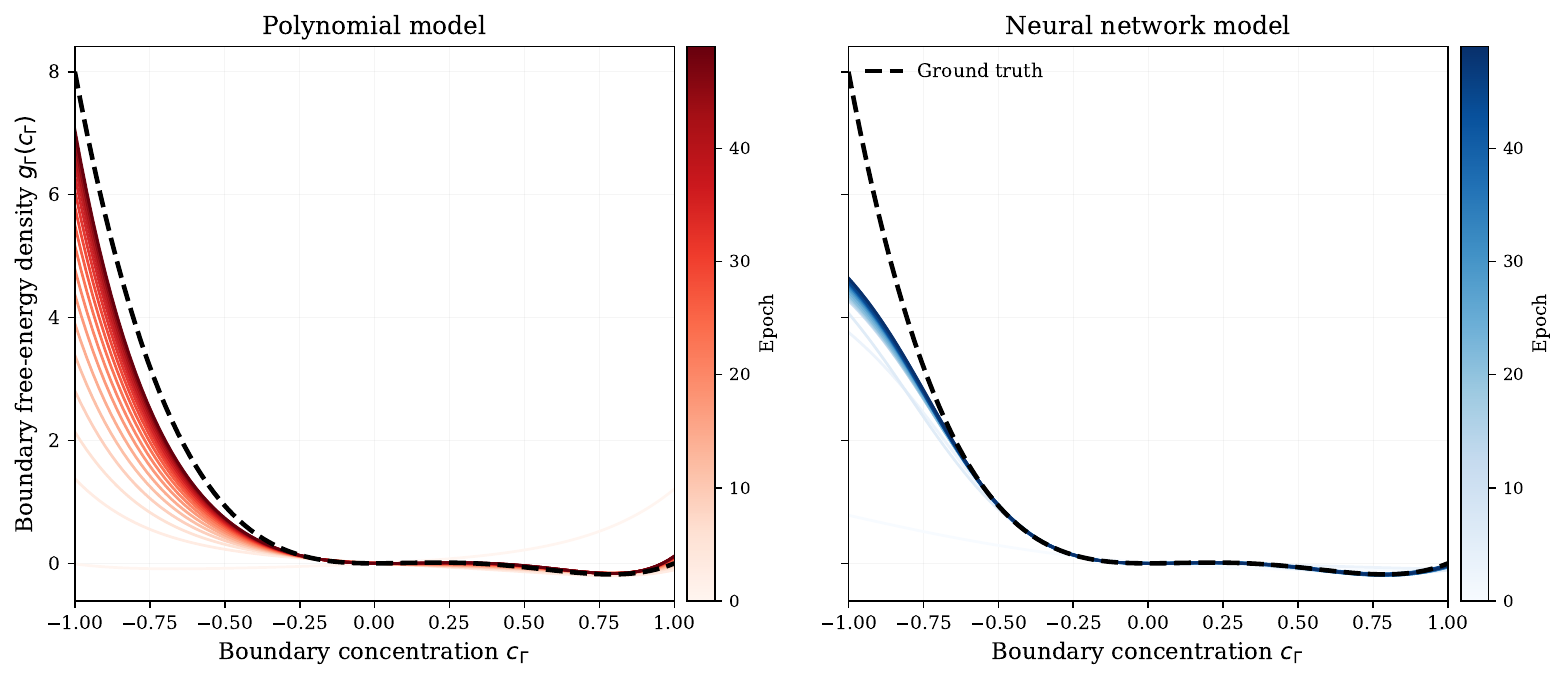}
    \end{minipage}
    \caption{Boundary potential recovery results. Left: training and validation loss curves as a function of training epochs. Right: free energy density evolving over training for both the polynomial (red) and neural network (blue) models. As before, later epochs correspond to darker curves, with ground truth overlaid as a dashed line in each case.}
    \label{fig:bc_lossdens}
\end{figure}
}

\newcommand{\figBCprofiles}{
\begin{figure}[htbp]
    \centering
    \includegraphics[width=0.85\textwidth]{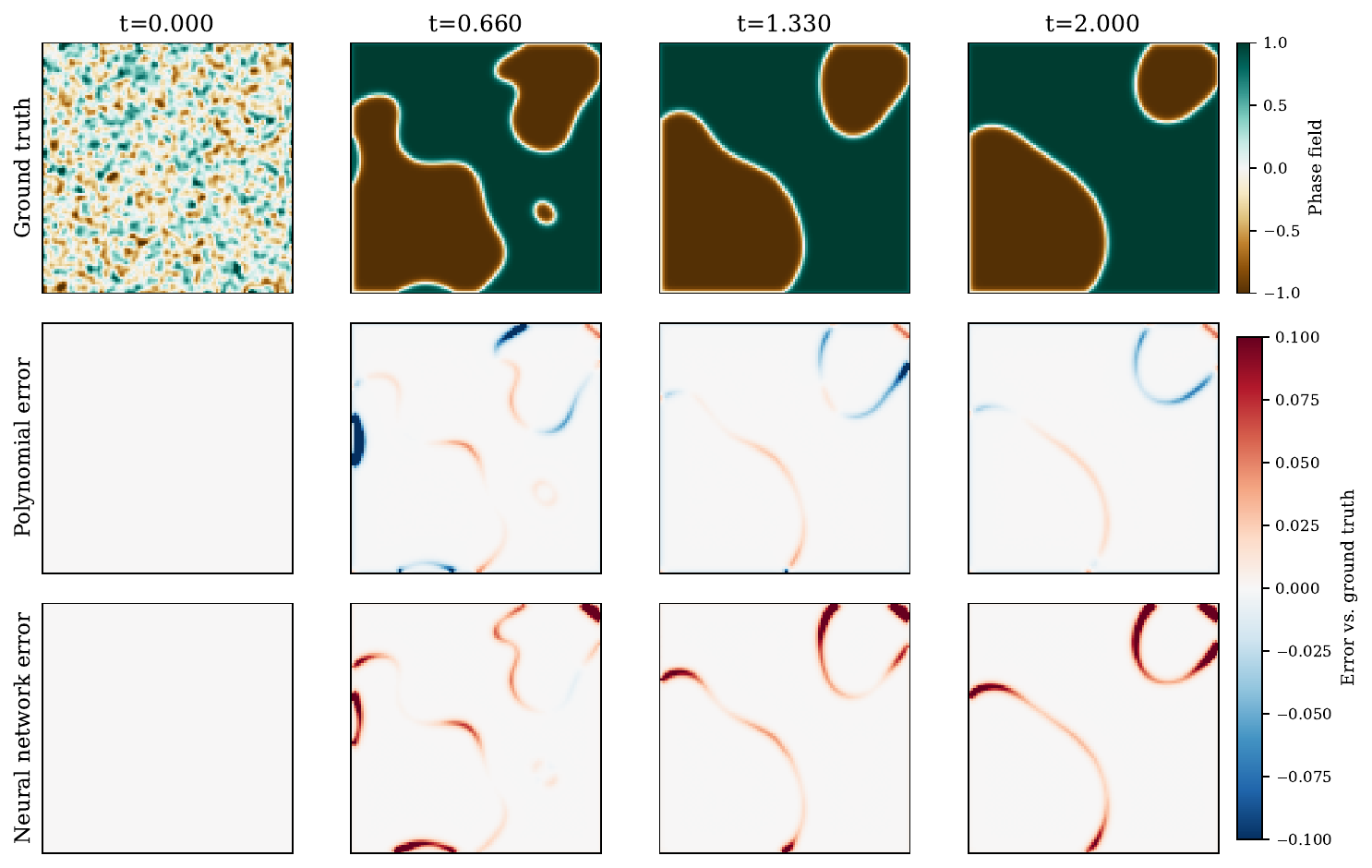}
    \caption{Evolution of the Allen-Cahn system with boundary potential starting from an initial condition in the validation set. Pictured is the ground truth evolution (top) along with the evolution of the signed errors in the case of the learned polynomial (middle) and neural network (bottom) surrogates. 
    }
    \label{fig:bc_profiles}
\end{figure}
}

\newcommand{\tabPoly}{
\begin{table}[htbp]
\caption{Coefficients of the learned polynomial free energies $F_{\theta}(C) = a_0 + a_1C + ... + a_6C^6$.}
\label{tab:poly}
\centering
\resizebox{\textwidth}{!}{
\begin{tabular}{l|l|lllllll}
\toprule
Density Type                & Model   & $a_0$ & $a_1$                & $a_2$   & $a_3$                & $a_4$   & $a_5$                & $a_6$                \\ \midrule
\multirow{2}{*}{Flory Huggins}               & Learned & 0     & $-8.28\times 10^{-4}$ & $-0.370$ & $1.23\times 10^{-3}$ & $0.236$ & -$3.97\times 10^{-4}$ & $7.42\times 10^{-2}$ \\ 
& Truth & N/A     & N/A & N/A & N/A & N/A & N/A & N/A \\ 
\midrule
\multirow{2}{*}{Quartic}    & Learned & 0     & $1.93\times 10^{-4}$ & $-0.515$ & $2.89\times 10^{-5}$ & $0.284$ & $-2.22\times 10^{-4}$ & $-1.28\times 10^{-2}$ \\
                            & Truth   & 0     & 0                    & $-0.5$  & 0                    & $0.25$  & 0                    & 0                    \\ \midrule
\multirow{2}{*}{Asymmetric} & Learned & 0     & $8.68\times 10^{-2}$ & $-0.516$ & $8.84\times 10^{-2}$ & $0.283$ & $-1.65\times 10^{-3}$ & $-1.07\times 10^{-2}$ \\
                            & Truth   & 0     & 0                    & $-0.5$  & $0.333$              & $0.25$  & 0                    & 0                    \\ \bottomrule
\end{tabular}
}
\end{table}
}

\newcommand{\tabThreeCases}{
\begin{table}[htbp]
\caption{Loss and error information corresponding to the bulk free energy density recovery experiment.}
\label{tab:threecases}
\resizebox{\textwidth}{!}{
\begin{tabular}{l|l|llll}
\toprule
Density Type                   & Model Class    & Training Loss & Validation Loss & Relative $L^2$ error ($C)$ & Relative $L^2$ error ($F_{\theta}$) \\
\midrule
\multirow{2}{*}{Flory Huggins} & Polynomial     & $1.68\times 10^{-4}$ & $1.77\times 10^{-4}$ & $4.66\times 10^{-2}$ & $2.07\times 10^{-1}$ \\
                               & Neural Network & $1.67\times 10^{-4}$ & $1.71\times 10^{-4}$ & $4.56\times 10^{-2}$ & $2.38\times 10^{-1}$ \\ \midrule

\multirow{2}{*}{Quartic}       & Polynomial     & $2.16\times 10^{-4}$ & $2.26\times 10^{-4}$ & $5.34\times 10^{-2}$ & $1.41\times 10^{-2}$ \\
                               & Neural Network & $2.11\times 10^{-4}$ & $2.14\times 10^{-4}$ & $5.89\times 10^{-2}$ & $2.43\times 10^{-2}$ \\ \midrule

\multirow{2}{*}{Asymmetric}    & Polynomial     & $2.31\times 10^{-4}$ & $2.47\times 10^{-4}$ & $5.57\times 10^{-2}$ & $1.45\times 10^{-2}$ \\
                               & Neural Network & $2.31\times 10^{-4}$ & $2.21\times 10^{-4}$ & $8.42\times 10^{-2}$ & $7.68\times 10^{-2}$ \\
\bottomrule
\end{tabular}
}
\end{table}
}

\newcommand{\figOneDevolutions}{
\begin{figure}[htbp]
    \centering
    \includegraphics[width=0.9\textwidth]{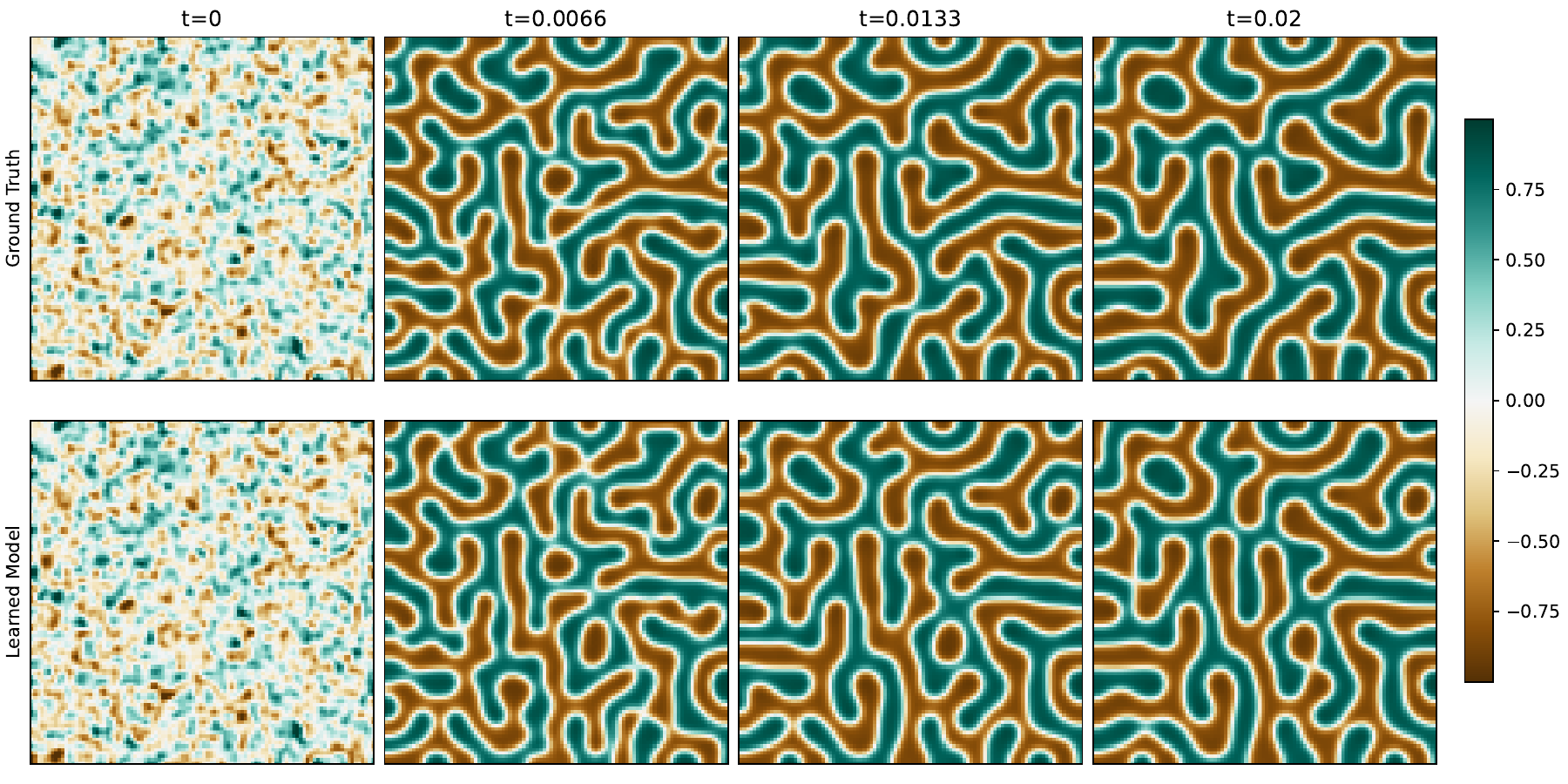}
    \caption{Evolution of the learned Ohta-Kawasaki system in Experiment 1, starting from an initial condition in the validation set and rolled out for $10\times$ the training window. Pictured is the ground truth evolution (top) along with the evolution of the learned surrogate (bottom).}
    \label{fig:OK_1D_evolutions}
\end{figure}
}

\newcommand{\figOneDeigenvalues}{
\begin{figure}[htbp]
    \centering
    \includegraphics[width=0.7\textwidth]{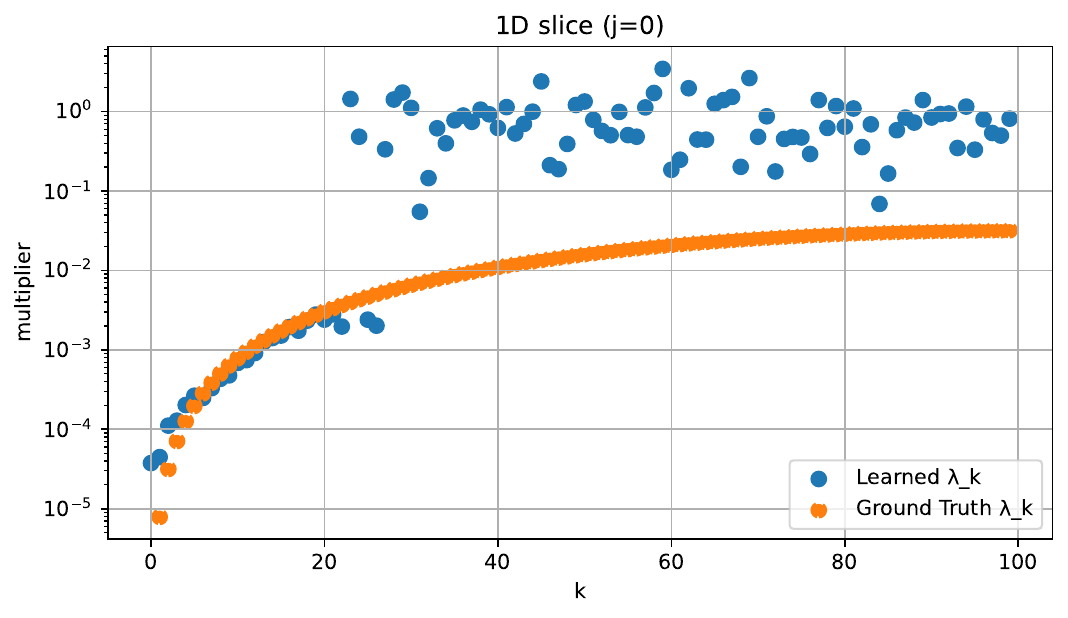}
    \caption{Eigenvalues of the Ohta-Kawasaki potential compared to the learned surrogate in Experiment 1.}
    \label{fig:OK_1D_eigenvalues}
\end{figure}
}

\newcommand{\figTwoDevolutions}{
\begin{figure}[htbp]
    \centering
    \includegraphics[width=0.9\textwidth]{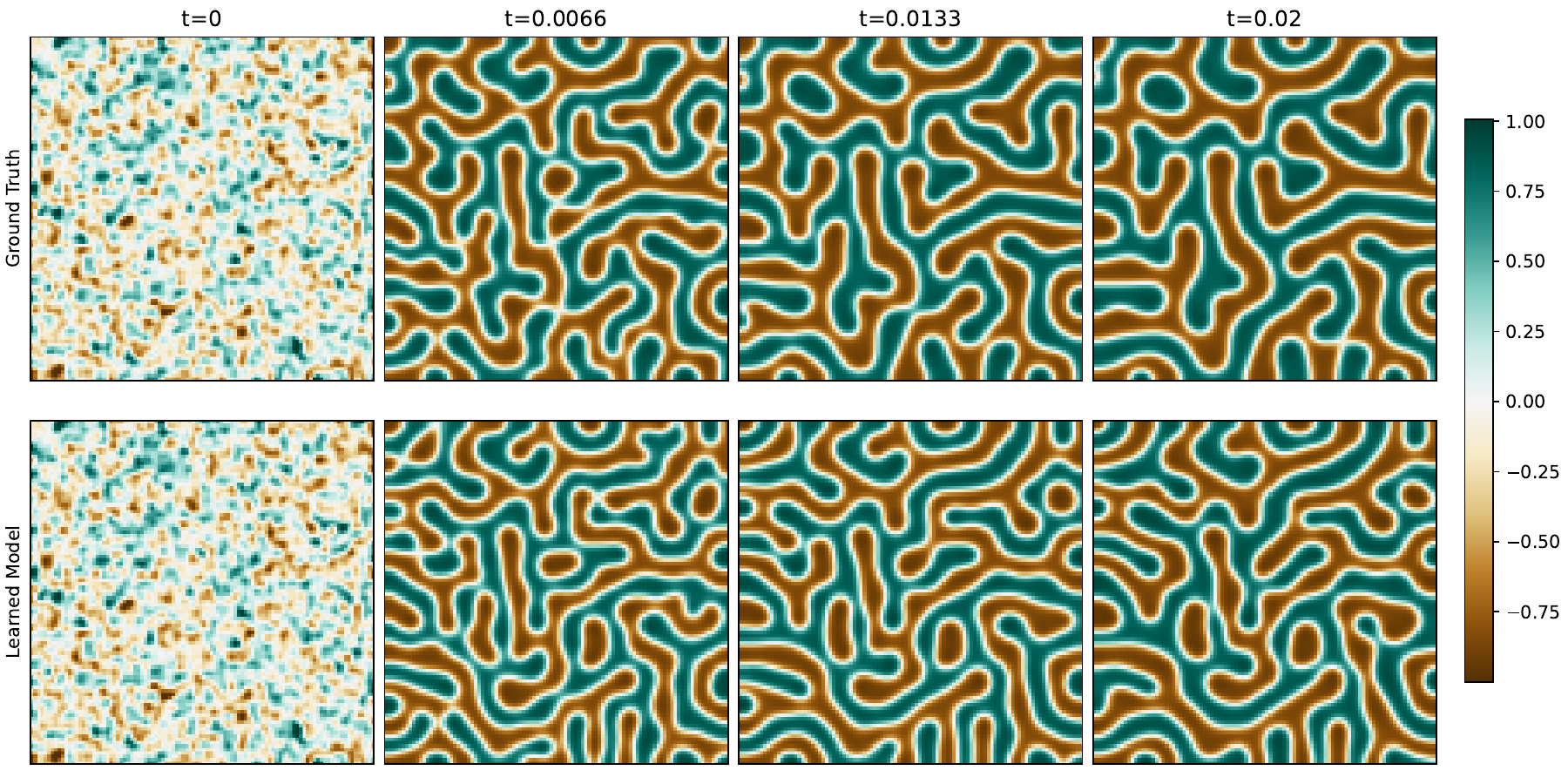}
    \caption{Evolution of the learned Ohta-Kawasaki system in Experiment 2, starting from an initial condition in the validation set and rolled out for $10\times$ the training window. Pictured is the ground truth evolution (top) along with the evolution of the learned surrogate (bottom).}
    \label{fig:OK_2D_evolutions}
\end{figure}
}

\newcommand{\figTwoDeigenvalues}{
\begin{figure}[htbp]
    \centering
    \includegraphics[width=\textwidth]{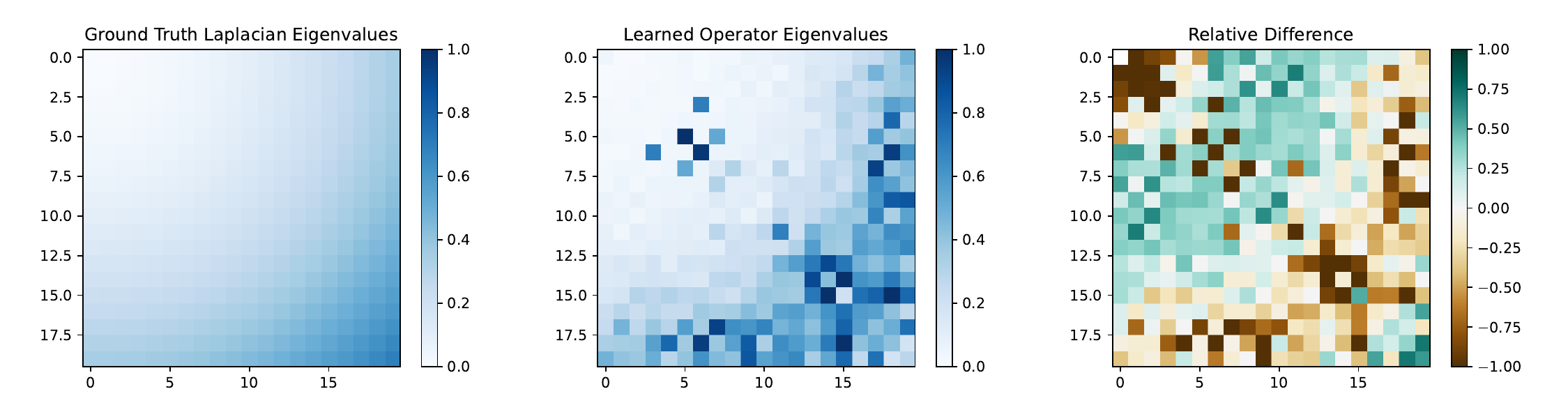}
    \caption{Spectrum of the Ohta-Kawasaki potential compared to that of the learned surrogate in Experiment 2. }
    \label{fig:OK_2D_eigenvalues}
\end{figure}
}

\definecolor{darkgreen}{rgb}{0.0, 0.5, 0.0}
\definecolor{darkblue}{rgb}{0.0, 0.0, 0.6}
\definecolor{darkred}{rgb}{0.5, 0.0, 0.0}
\hypersetup{
    colorlinks=true,
    linkcolor=darkgreen,    
    citecolor=darkblue,     
    urlcolor=darkred,       
}

\crefformat{equation}{#2(#1)#3}
\numberwithin{equation}{section}

\theoremstyle{plain}
\newtheorem{theorem}{Theorem}

\newtheorem{corollary}{Corollary}

\theoremstyle{remark}
\newtheorem{remark}{Remark}

\theoremstyle{definition}
\newtheorem{definition}{Definition}

\numberwithin{theorem}{section}
\numberwithin{proposition}{section}
\numberwithin{lemma}{section}
\numberwithin{corollary}{section}
\numberwithin{remark}{section}
\numberwithin{definition}{section}

\crefname{theorem}{Theorem}{Theorems}
\Crefname{theorem}{Theorem}{Theorems}
\crefname{proposition}{Proposition}{Propositions}
\Crefname{proposition}{Proposition}{Propositions}
\crefname{lemma}{Lemma}{Lemmata}
\Crefname{lemma}{Lemma}{Lemmata}
\crefname{corollary}{Corollary}{Corollaries}
\Crefname{corollary}{Corollary}{Corollaries}
\crefname{algorithm}{Algorithm}{Algorithms}
\Crefname{algorithm}{Algorithm}{Algorithms}
\crefname{appendix}{Appendix}{Appendices}
\Crefname{appendix}{Appendix}{Appendices}

\setminted{
    style=pastie,
    fontsize=\normalsize,
    linenos,
    frame=none
}

\begin{document} 

\pdfinfo{
   /Author (Anthony Gruber, Ritoban Roy-Chowdhury, Irina Tezaur, Nathan M. Urban)
   /Title (Flexible and Stable Dynamics Discovery with Onsager's Variational Principle)
   /Keywords (dissipative systems, dynamics discovery, gradient flow, variational integrators, structure preservation)
}

\title{\textbf{Flexible and Stable Dynamics Discovery with Onsager's Variational Principle}}

\author[1]{Anthony Gruber$^{*\dagger}$}
\author[1,2]{Ritoban Roy-Chowdhury$^\dagger$}
\author[3]{Irina Tezaur}
\author[4]{Nathan M. Urban}

\affil[1]{\normalsize Scientific Machine Learning, Center for Computing Research, Sandia National Laboratories}
\affil[2]{\normalsize Computer Science and Engineering, University of California at San Diego}
\affil[3]{\normalsize Quantitative Modeling and Software Engineering, Sandia National Laboratories}
\affil[4]{\normalsize Optimal Experimental Design \& Uncertainty Quantification, Brookhaven National Laboratory}
\date{}

\maketitle

\renewcommand{\thefootnote}{\fnsymbol{footnote}}
\maketitle
\footnotetext[1]{Corresponding author. E-mail: \href{mailto:adgrube@sandia.gov}{adgrube@sandia.gov}.}
\footnotetext[2]{These authors contributed equally to this work.}
\renewcommand{\thefootnote}{\arabic{footnote}}

\vspace{-.25in}

\begin{abstract}
\noindent
    Variational principles specify the dynamics of a physical system via the extremization of associated functional data. 
    Onsager's variational principle (OVP), which characterizes dissipation-dominated phenomena such as phase separation, admits an unconditionally energy-stable time discretization through the minimization of a Rayleighan functional combining free energy and dissipative effects.  The present work considers the case where one or more parts of this functional are empirically approximated or otherwise uncertain.  To address this, a novel variational discretization of OVP is introduced which recovers previous work as a special case, and a learning problem is formulated which identifies uncertain terms in the free energy and dissipation potential from observable data.  It is shown that the resulting OVP-based models connect directly to previous work in terms of proximal methods, Sobolev and Wasserstein gradient flows, while remaining provably energy-stable under arbitrarily long rollouts.  The approach is illustrated on examples including Allen-Cahn, Fokker-Planck, and Cahn-Hilliard system models, where the effects of bulk free-energy densities, nonlocal potentials, and nonstandard boundary conditions are effectively learned with model classes consisting of polynomials, shallow neural networks, and spectral convolution kernels. 
\end{abstract}




\section{Introduction}

Dissipative dynamical systems model phenomena ranging from nanoscale phase separation of polymers in solution to continent-scale motion of glacial ice sheets. However, in many practical settings of interest, sufficient quantitative descriptions of these phenomena are only partially known. As a result, it has become increasingly common to integrate data-driven modeling approaches with traditional methods for direct numerical simulation. Replacing unknown components of the governing equations with learnable surrogates can help bridge the gap between idealized theory and complex real-world behavior. Such hybrid models are particularly appealing in dynamics discovery problems, where the goal is to reconstruct the evolution of a system from a set of observational data.

On the other hand, physics is not arbitrary, and trajectories produced by a physical system are inherently constrained to configurations that obey its laws.  Na\"{i}vely including data-driven closure terms in a hybrid model of governing equations is unlikely to produce dynamics which generalize appropriately in every case, and it is well known that inserting unconstrained neural networks or other maximally flexible function approximators into governing equations often destroys fundamental structural properties of the underlying dynamics \cite{gruber2025efficiently,celledoni2021structure,loya2025structure,greydanus2019hamiltonian,hu2024energetic}.  In dissipative systems, the most important structural property is energy stability, which states that the free energy of the system decays monotonically in time.  This is equivalent (via a Legendre transform) to the second law of thermodynamics: the total entropy of the system cannot decrease as time progresses.  Violation of this basic principle leads to models that may overfit to training data but fail to exhibit physically reasonable behavior when queried out-of-distribution, rendering them useless in high-consequence scenarios.  Moreover, un-physical behavior such as spurious energy growth or dynamical instability during model training may compromise the optimization process, leading to poorly discovered dynamics in the first place.

Thankfully, this issue of poor adherence to physics can be substantially mitigated with the application of an appropriate variational principle.  When the physical phenomenon under consideration admits a variational representation in terms of the extremization of functional data, there is a well defined mapping from these functionals to the system's dynamics which comes from direct application of the variational principle.  Since learning functionals, which map system states to ordinary real numbers, is typically easier than learning state-to-state mappings produced by governing equations, an attractive option is to infer unknown or uncertain parts of these functionals from observational measurements.  Besides often yielding an easier learning problem, the variational approach has the advantage of automatic physical structure-preservation: given the trained functional data, the dynamics produced by the learned model are prescribed according to well-studied physical principles.

A particularly advantageous tool, which applies to dissipation-dominated systems, is Onsager's Variational Principle (OVP).  
In particular, OVP formulates the dynamics of non-equilibrium systems as those producing the \textit{least dissipative} path to a state of minimum free energy, which instantaneously balances the energetic driving forces and the dissipative drag effects at each step.  There are numerous advantages to this perspective.  First, it is inherently modular: system dynamics are specified through  functional minimization, meaning that coupled physical mechanisms are naturally incorporated through additional terms in the free energy or dissipation potential.  This is particularly useful for both phenomenological modeling as well as data-driven discovery, since unknown or uncertain effects can be incorporated parametrically through these functionals in a variationally consistent way.  This flexibility has inspired numerous descriptions of complex physical phenomena with OVP, including soft matter processes \cite{doi2011onsager,arroyo2017onsager}, bacterial suspensions \cite{wang2021onsager}, microstructural evolution processes \cite{du2024thermodynamic}, systems with odd elasticity \cite{lin2023onsager}, viscoelastic filaments \cite{zhou2018dynamics}, and multiphase flows \cite{xu2017hydrodynamic}, to name just a few.

On the computational side, OVP also provides a privileged starting point for the design of structure-preserving numerical schemes \cite{chen2025onsager,zhu2025stokes,liu2024variational,kou2020energy,xiao2025moving,xiong2025thermodynamically}.  When OVP is respected in the time-discrete setting, the resulting dynamics inherit unconditional, step-size-independent energy stability.  This means that simulations which start stable remain stable for all future times, since the free energy forms a Lyapunov function for the system.  Moreover, this holds regardless of the step-size employed as well as discrete free energy and dissipation potentials actually realized.  In the setting of dynamics discovery, this ensures that updates to these potentials always produce well behaved and physically reasonable dynamical evolutions, stabilizing training and improving efficiency in data-limited scenarios. 

Conversely, discretizations of dissipation-dominated phenomena which do not respect OVP are prone to un-physical dynamical instability or excessive numerical smoothing \cite{stuart1989nonlinear,furihata2001stable,tang2019energy}.  Besides being an issue for direct numerical simulation, this can cause significant difficulties in the setting of inverse problems, since ``outer loop'' updates to the governing equations (or the potentials producing them) are not guaranteed to yield well behaved solutions.  While dynamics learning problems certainly benefit from stable intermediates, it is even more critical that the trained model produces dynamics which generalize consistently out-of-distribution with guaranteed properties.  In particular, physical systems which equilibrate should correspond to models which also equilibrate. 
As a step toward fully modular, structure-preserving surrogate modeling for dissipative systems, this work considers the time discretization of OVP and its use in data-driven dynamics identification. The main novelty lies in the following contributions.
\begin{itemize}
    \item A geometry-first discussion of OVP and its relationship to gradient flows in general, and to Wasserstein and Sobolev gradient flows in particular.
    \item An unconditionally energy-stable time discretization of OVP that automatically includes process constraints and recovers previous work as a special case.
    \item An algorithm for dynamics discovery using OVP that is modular with respect to physics and boundary conditions and guaranteed to dissipate free energy independently of data or training quality.
    \item Demonstration of the proposed OVP-based strategy on a variety of test cases related to phase-field modeling and nonlinear reaction-diffusion equations.
\end{itemize}
The remainder of the work is organized as follows.  \Cref{sec:perspectives} discusses OVP in detail, drawing connections to force balance equations and the principle of maximum entropy, as well as presenting new results connecting OVP to various gradient flows.  Based on this discussion, \Cref{sec:timedisc} presents a structure-preserving time discretization of OVP which is unconditionally energy-stable and recovers previous  discretizations of OVP as a special case.  With this, \Cref{sec:inverse} formulates the key learning problem built on the proposed discrete OVP, which enables dissipative dynamics discovery in an optimization-based framework.  After a discussion on spatial discretization in \Cref{sec:spacedisc}, examples of this approach are presented in \Cref{sec:numerics}, where it is shown that OVP-based dynamics discovery produces flexible and accurate surrogate models which generalize appropriately out-of-distribution.  Finally, \Cref{sec:conclusion} draws conclusions and discusses avenues for future work.



\section{Background and Perspectives}\label{sec:perspectives}

Onsager's Variational Principle (OVP) \cite{onsager1931reciprocal,onsager1931reciprocal2} is an extension of Rayleigh's least-energy dissipation principle \cite{rayleigh1871some} to general irreversible systems which take place ``near equilibrium''.  Such systems are driven by a notion of free energy while being dissipation-dominated; they lack substantive inertial effects, and their resulting velocity is determined by the strength of this dissipative force.  This provides flexibility sufficient for describing a myriad of interesting partial differential equations (PDEs), including phase field models like Allen-Cahn and Cahn-Hilliard, the Stokes' equations for highly viscous fluids, and phenomena such as heat diffusion and Fokker-Planck which admit statistical interpretations.  The following definition is adapted from \cite{arroyo2017onsager}.



\begin{definition}\label{def:OVP}
Let $Q$ be a finite or infinite dimensional configuration manifold containing state variables $~q\in Q$, and let $E\to Q$ be a vector bundle of process variables $~w\in E_{~q}$\footnote{By $F_{~q}$ we mean the bundle-valued object $F$ restricted to its fiber over $~q$.} which specify dissipation.  Suppose the following functional data are provided: 
\begin{itemize}
    \item A free energy functional $\mathcal{F}:Q \rightarrow \mathbb{R}$ that is differentiable\footnote{Note that differentiability is assumed for convenience and not required for OVP to hold.} and minimized by the system.
    \item A dissipation potential $\mathcal{D}:E \rightarrow \mathbb{R}^{+}$ that is fiber-wise convex and satisfies the normalization condition $\mathcal{D}_{~q}(~0) = 0$.
    \item A process mapping $\mathcal{P}:E \rightarrow TQ$ that is a surjective and fiber-wise linear bundle homomorphism.
\end{itemize}
Denoting the fiber of $E$ at the point $~q$ by $E_{~q}$ and the usual evaluation pairing on $TQ$ by $\dual{\cdot}{\cdot}:T^*Q\times TQ\to\mathbb{R}$, consider the Rayleighian  $\mathcal{R}:E\to\mathbb{R}$ at $~w \in E_{~q}$: 
\begin{equation}\label{eq:rayleighan}
    \mathcal{R}_{~q}(~w) = \dual{d\mathcal{F}_{~q}}{\mathcal{P}_{~q}~w} + \mathcal{D}_{~q}(~w),
\end{equation}

Observing that $~w \mapsto \mathcal{R}_{~q}(~w)$ is a convex function\footnote{To guarantee a unique minimum, we will assume the usual coercivity condition $\mathcal{R}_{~q}(~w)\to\infty$ as $\nn{~w}\to\infty$.}, Onsager's variational principle is the statement that a curve $t\mapsto ~q(t)\in Q$ in the configuration manifold satisfies the condition
\begin{equation}\label{eq:Onsager's}
    \dot{~q}(t) = \mathcal{P}_{~q(t)}~v, \qquad ~v = \argmin_{~w\in E_{~q(t)}}\,\mathcal{R}_{~q(t)}(~w).
\end{equation}
\end{definition}

\figOVP

\begin{remark}\label{rem:ext-pow}
    The Rayleighan $\mathcal{R}$ in \eqref{eq:rayleighan} is occasionally modified to include a notion of external power, i.e., $\mathcal{R}_{~q}(~w) = \dual{d\mathcal{F}_{~q}}{\mathcal{P}_{~q}~w} + \mathcal{D}_{~q}(~w) + \dual{P_{~q}}{~w}$ for some $P:E\to\mathbb{R}$ which is linear on each fiber.  Of course, minimizers of this Rayleighan do not generally lead to curves $~q\in Q$ which monotonically decrease the free energy $\mathcal{F}$.
\end{remark}

Unlike more common variational principles, such as Hamilton’s least action principle in mechanics (see e.g. \cite{marsden1994mathematical}), \Cref{def:OVP} shows that OVP does not realize system dynamics as extreme or stationary points of a particular energy functional.  Instead, the velocity $\dot{~q}(t)$ in \cref{eq:Onsager's} is defined in terms of a minimal vector $~v\in E_{~q(t)}$ in the process space, which effectively selects an updated configuration $~q(t+\delta t)$ not too far from the equilibrium state.  This choice relies on the free energy $\mathcal{F}$ and dissipation potential $\mathcal{D}$ described in \Cref{def:OVP}, as well as the process mapping $\mathcal{P}$, which expresses the change in configuration $\dot{~q}=\mathcal{P}_{~q}~w$ as a linear combination of process variables $~w$ with nonlinear coefficients $\mathcal{P}_{~q}$.  Intuitively, the process bundle $E$ contains ``microscopic'' quantities like fluxes or process rates, which are converted to changes in ``macroscopic'' quantities like concentrations and temperature fields via the process mapping $\mathcal{P}$.

In many cases of interest, the process bundle $E = TQ$ is simply the tangent bundle to the configuration manifold, and therefore the process function $\mathcal{P}=I$ is just the identity mapping.  However, it is important to note that Onsager's description of dissipation generally contains more information than necessary for specifying changes in configuration $\dot{~q}$, in which case the mapping $\mathcal{P}$ is necessarily nontrivial  \cite{arroyo2017onsager,otto2001geometry}.  
For example, the Cahn-Hilliard model (discussed in \Cref{subsec:sobolev}) satisfies a continuity equation relating a scalar concentration field $c \in Q$ with a vector flux field $c~w \in E_c$. Infinitesimal changes $\dot{c}\in T_cQ$ to the concentration field then arise from an optimization over the larger space of fluxes via the process mapping $\mathcal{P}_c(~w) = -\nabla \cdot c~w$, which encodes the conservation of mass.  More generally, dissipation in chemical kinetics may be described by a process bundle of reaction rate fields, while dissipation in viscous fluids may be described a bundle of strain-rate tensors.  In any case, OVP \Cref{def:OVP} provides a powerful algorithm for producing physically consistent models from relatively high-level descriptions of configuration and dissipation.

In contrast to the corresponding definition (47)-(48) in \cite{arroyo2017onsager}, \Cref{def:OVP} does not discuss the treatment of explicit constraints on the configuration or process variables.  However, these restrictions are straightforward to incorporate into OVP and do not complicate its application.  First, observe that nonlinear holonomic constraints $N(~q)= ~0$ on the state variables $~q\in Q$ can be linearized and expressed as constraints on the process variables, i.e., $N'(~q)~w = ~0$.  Moreover, any linear function $~C_{~q}: E_{~q} \to G_{~q}$ mapping $E$ to some bundle $G$ and defining linear constraints $~C_{~q}~w = ~0$ on the process variables $~w\in E_{~q}$, including those constraints seen above, are readily incorporated into \Cref{def:OVP} via minimization on the sub-bundle $\tilde{E}\to Q$ with fibers $\tilde{E}_{~q} = \{~w \in E_{~q}\,|\,~C_{~q}~w = ~0\}$. In practice, this often involves the use of Lagrange multipliers $\mu\in G_{~q}^*$ on the dual fibers to the image space of $~C$ and a Lagrangian $\tilde{\mathcal{R}}:E\oplus G^*\to\mathbb{R}$ defined on the Whitney (i.e., fiber-wise direct) sum of $E$ and $G^*$, 
\begin{equation}
    \tilde{\mathcal{R}}_{~q}(~w,\mu) = \mathcal{R}_{~q}(~w) - \dual{\mu}{~C_{~q}~w}.
\end{equation}
In this case, it is straightforward to show (formally) that the section of $\tilde{E}$ specified by OVP solves the saddle point problem
\begin{equation}
    ~v = \argmin_{~w\in \tilde{E}_{~q}}\,\mathcal{R}_{~q}(~w) = \argmin_{~w\in E_{~q}} \max_{\vphantom{{A^2}'}\mu\in G^*_{~q}}\,\tilde{\mathcal{R}}_{~q}(~w,\mu).
\end{equation}

Given only functional data, OVP prescribes an appropriate path in $Q$, where the first term of the Rayleighan $\mathcal{R}$ becomes the rate-of-change in free energy: for $~v$ satisfying \eqref{eq:Onsager's},
\begin{equation*}
    \dual{d\mathcal{F}_{~q(t)}}{\mathcal{P}_{~q(t)}~v} = \dual{d\mathcal{F}_{~q(t)}}{\dot{~q}(t)} = \dot{\mathcal{F}}\lr{~q(t)}.
\end{equation*} 
In fact, OVP guarantees more: the dynamics on $Q$ move states in the direction of minimum free energy.
\begin{theorem}\label{thm:continuous_ES}
    The free energy $\mathcal{F}$ decreases monotonically along solutions to OVP \eqref{eq:Onsager's}.
\end{theorem}
\begin{proof}
    We will assume differentiability in the dissipation potential $\mathcal{D}$ for convenience. Overloading the notation $\dual{\cdot}{\cdot}$ to describe the usual evaluation pairing on $E\times E^*$, denote by $\partial_{~w}\mathcal{A}:E\to E^*$ the fiber derivative of the functional $\mathcal{A}:E\to\mathbb{R}$, defined at the point $~q\in Q$ by $\dual{\partial_{~w}\mathcal{A}_{~q}(~w)}{\mathring{~w}} = \frac{d}{ds}\mathcal{A}_{~q}(~w+s\mathring{~w})|_{s=0}$ for any variation vector $\mathring{~w} \in E_{~q}$.  It follows that the variation of the Rayleighan $\mathcal{R}$ can be computed as
    \begin{equation*}
        \dual{\partial_{~w}\mathcal{R}_{~q}(~w)}{\mathring{~w}} = \dual{d\mathcal{F}_{~q}}{\mathcal{P}_{~q}\mathring{~w}} + \dual{\partial_{~w}\mathcal{D}_{~q}(~w)}{\mathring{~w}} = \big\langle \mathcal{P}_{~q}^\top d\mathcal{F}_{~q} + \partial_{~w}\mathcal{D}_{~q}(~w)\,\big|\,\mathring{~w}\big\rangle,
    \end{equation*}
    where $\mathcal{P}^\top:T^*Q\to E^*$ is the transpose of the homomorphism $\mathcal{P}$. 
    At the minimal vector $~v = \argmin_{~w}\mathcal{R} \in E_{~q}$ in the fiber $E_{~q}$, this implies that the equation
    \begin{equation*}
        \mathcal{P}_{~q}^\top d\mathcal{F}_{~q} = -\partial_{~w}\mathcal{D}_{~q}(~v),
    \end{equation*}
    is satisfied in the dual fiber $E_{~q}^*$.  Pairing this equation with $~v$, the derivative of free energy satisfies 
    \begin{equation*}
        \dot{\mathcal{F}}_{~q}(~v) = \dual{d\mathcal{F}_{~q}}{\mathcal{P}_{~q}~v} = - \dual{\partial_{~w}\mathcal{D}_{~q}(~v)}{~v} \leq -\mathcal{D}_{~q}(~v) \leq 0,
    \end{equation*}
    where the final inequality used that $\mathcal{D}$ is nonnegative while the first inequality follows by convexity and the fact that $\mathcal{D}_{~q}$ vanishes at zero: 
    \begin{equation*}
        0 = \mathcal{D}_{~q}(~0) \geq \mathcal{D}_{~q}(~v) + \dual{\partial_{~w}\mathcal{D}_{~q}(~v)}{~0-~v} \quad \Longrightarrow \quad \dual{\partial_{~w}\mathcal{D}_{~q}(~v)}{~v} \geq \mathcal{D}_{~q}(~v).
    \end{equation*}
    This establishes that the rate-of-change $\dot{\mathcal{F}}\leq 0$ is non-positive, as desired.
\end{proof}

\begin{remark}
    It is often the case that forming the transpose $\mathcal{P}^\top$ of the process function requires an integration-by-parts on some Riemannian manifold $\mathcal{M}$.  Therefore, the definition of the process bundle $E$ should include appropriate boundary conditions when necessary (c.f., \Cref{subsec:wass_gfs}).  
\end{remark}

\Cref{thm:continuous_ES} guarantees that the trajectories produced by OVP are asymptotically stable, since the free energy $\mathcal{F}$ serves as a Lyapunov function for the dynamics.  This reflects the fact that states equilibrate under OVP, albeit in the ``least dissipative'' fashion.  This can be better understood through exploration of the optimality conditions seen in the above proof.  To that end, the remainder of this section presents two related perspectives on OVP which connect it to metric gradient flows and motivate the discretization in \Cref{sec:timedisc}.  While aspects of this discussion are known to experts in this area, the exposition, presentation, and mathematical formulation detailed here is novel.

\subsection{OVP as a Balance of Forces}
As illustrated in the proof of \Cref{thm:continuous_ES}, when $\mathcal{R}$ is differentiable, the optimality conditions for OVP \eqref{eq:Onsager's} take place in the dual $E^*$ to the process bundle\footnote{Technically, they take place in the pullback bundle $~q^*(E^*) \to \mathbb{R}_+$ over $\mathbb{R}_+$ whose fiber at $t\in\mathbb{R}_+$ is $E^*_{~q(t)}$.}:
\begin{equation}\label{eq:onsager_optimality}
    \mathcal{P}_{~q(t)}^{\top}\mathrm{d}\mathcal{F}_{~q(t)} + \partial_{~w}\mathcal{D}_{~q(t)}(~v) = 0.
\end{equation}
This can be interpreted as a balance law between the force associated to the free energy $\mathcal{F}$ and the corresponding dissipative force generated by $\mathcal{D}$. For example, consider a mass connected to a spring, with state $x\in Q \coloneqq \mathbb{R}$ given by its position in 1-dimension and process variable $w\in E_x\coloneqq \mathbb{R}$.  The spring force emerges from the potential energy $\mathcal{F}(x) = \frac{1}{2}kx^{2}$, defined in terms of an elastic constant $k>0$, while a dissipative force emerges from the potential
$\mathcal{D}_x(w) = \frac{1}{2}w^{2}$ generating the viscous effects.  Given a process function $\mathcal{P}=I$ which identifies $w\in E_x$ with the velocity $\dot{x}\in T_xQ$ tangent to the configuration $x$, OVP \Cref{def:OVP} yields the usual balance equation between the elastic and viscous forces,
\[\dot{x} + kx = 0.\] 
In this way, the potentials $\mathcal{F},\mathcal{D}$ used in OVP correspond directly to force-balanced dynamics on the state space $Q$.  However, this correspondence is not necessarily one-to-one: for instance, it is easy to see that constant shifts in these potentials do not affect the optimality condition \eqref{eq:onsager_optimality}.  

A second, infinite dimensional example described in \cite{arroyo2017onsager} is also helpful.  Consider the Stokes' equations for an incompressible Newtonian fluid in the low-Reynolds limit.  For a fluid velocity field denoted by $~w\in\Gamma(T\mathcal{M})$ where $\mathcal{M}$ is a fixed volume with boundary $\partial\mathcal{M} = \Gamma_D \cup \Gamma_N$ decomposed into fixed Dirichlet and Neumann components, these equations are:
\begin{equation}\label{eq:Stokes}
\begin{cases}
    \nabla\cdot~\sigma = ~0 & \mathrm{in}\,\,\mathcal{M}, \\
    \nabla\cdot~w = 0 & \mathrm{in}\,\,\mathcal{M}, \\
    ~w = ~w_D & \mathrm{on}\,\,\Gamma_D, \\
    ~\sigma\cdot~n = ~t & \mathrm{on}\,\,\Gamma_N.
\end{cases}
\end{equation}
Here, $~\sigma \in \Gamma(T\mathcal{M}\otimes_{\mathrm{sym}}T\mathcal{M})$ is the stress tensor of the fluid, $~n\in \Gamma(T\mathcal{M}^\perp)$ is the outward-directed unit normal to the boundary, and $~w_D, ~t\in \Gamma(T\partial\mathcal{M})$ are given Dirichlet and traction boundary conditions.  As dissipation dominates over inertial effects, the system \eqref{eq:Stokes} should be describable by OVP.  To that end, denote the fluid velocity field by $~w\in\Gamma(T\mathcal{M})$ where $\mathcal{M}$ is a fixed volume with boundary $\partial\mathcal{M}$.  Dissipation in this system is independent of the system state and generated by the potential 
\[ \mathcal{D}(~w) = \eta\int_{\mathcal{M}}\nn{\mathrm{sym}(\nabla~w)}^2\,dV, \]
where $\eta$ is the shear viscosity of the fluid and $\mathrm{sym}(\nabla~w) = (1/2)\big(\nabla~w + (\nabla~w)^\top\big)$ is the strain rate tensor.  While there is no energetic component to this system (hence, trivial configuration space $Q=\{0\}$), external power (c.f. \Cref{rem:ext-pow}) is supplied through the traction boundary condition: $\langle P\,|\,~w\rangle = -\int_{\Gamma_N}~t\cdot~w\,dS$.  It follows that the process mapping $\mathcal{P}=I$ is identity and the system Rayleighan is given by $\mathcal{R}(~w) = \mathcal{D}(~w) + \langle P\,|\,~w\rangle$, where the (trivial) process bundle $E \cong Q\times \bar{E}$ has fibers 
\[ \bar{E} = \{~w\in \Gamma(T\mathcal{M})\,|\,\nabla\cdot~w=0\,\,\mathrm{in}\,\,\mathcal{M},\,~w=~w_D\,\,\mathrm{on}\,\,\Gamma_D\}. \]
Introducing the Lagrange multiplier $p\in L^2(\mathcal{M})$ and defining the constrained Rayleighan $\tilde{\mathcal{R}}(~w,p) = \mathcal{R}(~w) - \langle p\,|\nabla\cdot~w\rangle$, OVP \eqref{eq:Onsager's} becomes the expected statement 
\[ ~v = \argmin_{~w\in\bar{E}}\max_{\vphantom{{A^2}'}p\in L^2(\mathcal{M})}\, \tilde{\mathcal{R}}(~w,p).\]
A routine calculation using the balance law \eqref{eq:onsager_optimality} immediately recovers the Stokes equations \eqref{eq:Stokes} for $~\sigma = 2\eta\,\mathrm{sym}(\nabla~v) - p~I$.

\subsection{Maximum Entropy and Gradient Flow}
Another powerful way to view OVP is as a non-equilibrium generalization of the principle of maximum entropy.  In equilibrium thermodynamics, state variables are functions of a system's internal energy $U$, volume $V$, and number of molecules $N$, or their conjugate quantities. 
For example, 
the thermodynamic free energy 
\begin{equation*}
    \mathcal{F}(T,V,N) = \inf_U\,\{U - T\,S(U,V,N) \},
\end{equation*}
is the Legendre transform (up to sign conventions) of the entropy $S$ in the internal energy variable. The fundamental postulate governing equilibrium phenomena is that systems will settle in configurations that maximize entropy, or, equivalently, minimize the free energy
\cite{friedli2018statistical,callen1998thermodynamics}.  OVP represents a non-equilibrium generalization of this principle which is genuinely dynamical: 
following the flow on the configuration space $Q$ specified by \Cref{def:OVP} prescribes the minimally dissipative path that a system state $~q$ should take to reach equilibrium in a stable and physically meaningful way \cite{onsager1931reciprocal,onsager1931reciprocal2}.

The idea of variationally prescribed dynamics inherent in OVP is intimately connected to gradient flow.  Recall that the gradient flow of a functional $\mathcal{F}:Q\to\mathbb{R}$ represents steepest descent with respect to a given metric $g\in \Gamma\lr{T^*Q\otimes_{\mathrm{sym}}T^*Q}$, which is a symmetric section of the covariant 2-tensor bundle over $Q$.  Writing $\dual{d\mathcal{F}}{\mathring{~q}} \eqqcolon g\lr{\mathrm{grad}_g\mathcal{F},\mathring{~q}}$ for any $\mathring{~q}\in\Gamma(TQ)$, this flow can be expressed in terms of the Riesz representer $\mathrm{grad}_g\mathcal{F}\in \Gamma(TQ)$ as 
\begin{equation*}
    \dot{~q}(t) = -\mathrm{grad}_g\mathcal{F}_{~q(t)},
\end{equation*}
from which is easy to see that the functional $\mathcal{F}$ decreases along solutions:
\begin{equation*}
    \dot{\mathcal{F}}(~q) = \dual{d\mathcal{F}_{~q}}{\dot{~q}} = -\nn{\mathrm{grad}_g\mathcal{F}_{~q}}_g^2 \leq 0,
\end{equation*}
where $\nn{\cdot}_g = g(\cdot,\cdot)^{1/2}$ is the norm induced by the metric $g$.  When the dissipation potential $\mathcal{D}$ is quadratic and the process mapping $\mathcal{P}$ is surjective, it turns out that OVP is naturally posed as a gradient flow with metric $h\in\Gamma\lr{T^*Q\otimes_{\mathrm{sym}}T^*Q}$ inherited from the dissipative effects.\footnote{Even when $\mathcal{D}$ is not quadratic, OVP can be interpreted as a type of generalized gradient flow \cite{mielke2016generalization}.}

To see this, observe that the derivative $\partial_{~w}\mathcal{D}:E\to E^*$ is fiber-wise linear and symmetric due to quadraticity of the potential, as well as positive definite by convexity.  Therefore,
$\flat \coloneqq \partial_{\mathbf{w}}\mathcal{D}$ is an isomorphism between the process bundle and its dual which induces a symmetric, fiber-wise inner product (i.e., metric) $g \in \Gamma\lr{E^*\otimes_{\mathrm{sym}}E^*}$: for $~w_1,~w_2\in E_{~q}$,
\begin{equation*}
    g\lr{~w_1,~w_2} \coloneqq \dual{\flat~w_1}{~w_2} = \dual{\flat~w_2}{~w_1}.
\end{equation*}
Denoting by $\sharp:E\to E^*$ the corresponding inverse mapping $\sharp=\flat^{-1}$, the optimality condition \eqref{eq:onsager_optimality} can be stated (for each $t$) as $~v = -\sharp_{~q}\mathcal{P}_{~q}^\top d\mathcal{F}_{~q}$.  It follows that \eqref{eq:Onsager's} becomes a gradient flow in terms of the transpose of the process mapping $\mathcal{P}^\top$ and the inverse metric $g^{-1}\in\Gamma\lr{E\otimes_{\mathrm{sym}}E}$ on $E^*$ satisfying $g^{-1}\lr{\flat~w_1,\flat~w_2}=g(~w_1,~w_2)$.  More precisely, OVP becomes the statement 
\begin{equation}\label{eq:OVP-GF}
    \dot{~{q}}(t) = - \mathcal{P}\sharp\mathcal{P}^{\top}\mathrm{d}\mathcal{F}\big|_{~{q}(t)}.
\end{equation}
Since $\mathcal{P}\sharp\mathcal{P}^\top:T^*Q\to TQ$ is symmetric, this is nearly a gradient flow of the free energy $\mathcal{F}$ for any process mapping $\mathcal{P}$.  It is certainly dissipative (as required by OVP), since 
\begin{equation*}
    \dot{\mathcal{F}}(~q) = \dual{d\mathcal{F}_{~q}}{\dot{~q}} = -\big\langle d\mathcal{F}_{~q}\,\big|\,\mathcal{P}\sharp\mathcal{P}^\top d\mathcal{F}_{~q}\big\rangle = -\big\langle\mathcal{P}^\top d\mathcal{F}_{~q}\,\big|\,\sharp\mathcal{P}^\top d\mathcal{F}_{~q}\big\rangle = -\big|\mathcal{P}^\top d\mathcal{F}_{~q}\big|^2_{g^{-1}} \leq 0. 
\end{equation*}
On the other hand, there is no guarantee that the mapping $\mathcal{P}\sharp\mathcal{P}^\top$ is invertible, in which case the pullback (pseudo-)metric $\tilde{h} \coloneqq  \lr{\mathcal{P}^\top}^*g^{-1}$ on $T^*Q$, defined by $\tilde{h}\lr{\omega_1,\omega_2} = g^{-1}\lr{\mathcal{P}^\top\omega_1,\mathcal{P}^\top\omega_2}$, is similarly non-invertible.  This issue is circumvented when $\mathcal{P}$ is surjective on the bundle of state velocities $TQ$, since this implies that its transpose $\mathcal{P}^\top$ is injective and therefore $\mathcal{P}\sharp\mathcal{P}^\top$ is a bundle isomorphism.  In this case $\tilde{h} = h^{-1}$ is a genuine metric tensor on $T^*Q$, and its inverse metric $h$ yields OVP as the genuine gradient flow $\dot{~q} = -\mathcal{P}\sharp\mathcal{P}^\top d\mathcal{F}|_{~q} = -\mathrm{grad}_h\mathcal{F}|_{~q}$.  Diagrammatically, this is summarized by the sequence of mappings
{\centering 
\[
T_{~q}^*Q
\;\xrightarrow{\;\;\mathcal{P}_{~q}^\top\;\;}\;
E_{~q}^*
\;\xrightarrow{\;\;\sharp_{~q}\;\;}\;
E_{~q}
\;\xrightarrow{\;\;\mathcal{P}_{~q}\;\;}\;
T_{~q}Q,
\]
\par }
that maps the free-energy covector \(d\mathcal{F}|_{~q} \in T_q^*Q\) to the process dual space by \(\mathcal{P}_{~q}^\top\), converts it into a process variable using the inverse dissipation metric \(\sharp_{~q}\), and pushes the result forward to a state velocity by \(\mathcal{P}_{~q}\) to yield the effective mobility operator \(\mathcal{P}_{~q} \sharp_{~q} \mathcal{P}_{~q}^\top : T_{~q}^*Q \to T_{~q}Q\) governing the gradient flow.

\begin{remark}
    When the dissipation potential $\mathcal{D}_{~q}$ is quadratic,  symmetry of the phenomenological kinetic coefficients $P\sharp\mathcal{P}^\top|_{~q}$ is equivalent to the celebrated ``reciprocal relations'' of Onsager  \cite{doi2011onsager, onsager1931reciprocal2}, sometimes phrased as $L_{ij}=L_{ji}$ where $\dot{q}_i = \sum_j L_{ij}\frac{\partial S}{\partial q^j}$ for entropy $S$.  
\end{remark}



This discussion shows that systems satisfying OVP are gradient flows in many cases of interest.  Conversely, it is also true that any gradient flow on an inner
product space or Riemannian manifold satisfies OVP.  This can be seen by considering a generic metric $h\in\Gamma\lr{T^*Q\otimes_{\mathrm{sym}} T^*Q}$ and its associated musical $\flat,\sharp$ isomorphisms, mapping between the bundles $TQ\leftrightarrow T^*Q$ and defined the same way as before.  It follows from the definitions of sharp and gradient that 
\[\dual{d\mathcal{F}}{\mathring{~q}} = h\lr{\mathrm{grad}_h\mathcal{F},\mathring{~q}} = \dual{\flat\,\mathrm{grad}_h\mathcal{F}}{\mathring{~q}} \qquad \forall \,\mathring{~q}\in\Gamma(TQ), \]
establishing that the $h$-gradient has the alternative expression $\mathrm{grad}_h\mathcal{F} = \sharp \, d\mathcal{F}$.  Therefore, the gradient flow equation $\dot{~q} = -\mathrm{grad}_h\mathcal{F}|_{~q} = -\sharp\, d\mathcal{F}|_{~q}$ matches \eqref{eq:OVP-GF} when the process bundle $E=TQ$ is the tangent bundle, the process mapping $\mathcal{P}=I$ is the identity, and the dissipation potential $\mathcal{D}_{~q}(~w) = \frac12\nn{~w}^2_h$ is the squared $h$-norm.  Collecting results, the discussion in this subsection has established the following formal correspondence.
\begin{theorem}\label{thm:gradflow}
    Suppose the dissipation potential $\mathcal{D}:E\to\mathbb{R}$ is quadratic and the process mapping $\mathcal{P}:E\to TQ$ is surjective.  Then, solutions to OVP are gradient flows \eqref{eq:OVP-GF} on the configuration manifold $Q$.  Conversely, any sufficiently regular gradient flow $\dot{~q} = -\mathrm{grad}_h\mathcal{F}|_{~q}$ on $Q$ defined in terms of a metric $h\in\Gamma\lr{T^*Q\otimes_{\mathrm{sym}} T^*Q}$ satisfies OVP with the choices $E=TQ$, $\mathcal{D}_{~q}(~w)= \frac12\nn{~w}_h^2$, and $\mathcal{P}=I$.
\end{theorem}

\Cref{thm:gradflow} shows that solutions to OVP are closely related to gradient flows on the configuration space $Q$.  For example, given $Q = \{c\in C^{\infty}(\mathcal{M}): \partial_{~n}c=0\}$, it is well known that the diffusion equation \(\dot{c} = \Delta c\) is the \(L^{2}\)-gradient flow of the Dirichlet energy $\mathcal{F}(c) = \frac12\int_{\mathcal{M}}|\nabla c|^{2}\,dV$, and therefore satisfies OVP with the choices $\mathcal{D}_c(w) = \frac12\int_\mathcal{M} |w|^2\,dV$ and $\mathcal{P}=I$.  On the other hand, even when OVP is equivalent to specifying a gradient flow, it offers a more granular view of the involved physics.  Phenomenologically, this is often useful: the free energy $\mathcal{F}$, dissipation potential $\mathcal{D}$, and process mapping $\mathcal{P}$ can be designed based on physical principles or learned from data, and OVP can be applied to produce stable and meaningful dynamics.  This degree of control over the constituent potentials can have benefit- even in the case that the resulting equations of motion are the same.  In the case of the diffusion equation, it has been argued (see \cite{otto2001geometry,arroyo2017onsager}) that the Dirichlet energy is not the correct physical mechanism by which free energy is dissipated: the free energy density of a solute with concentration $c$ should take the form $c\log c$ up to solubility constants.  This leads naturally to the idea of Wasserstein gradient flows, which will now be discussed in the context of several examples.

\subsubsection{Wasserstein Gradient Flows}\label{subsec:wass_gfs}
A more enlightening formulation of certain systems satisfying OVP, including simple diffusion, can be derived by considering Wasserstein gradient flows (WGFs) \cite{otto2001geometry}. To that end, consider $Q = \{c\in C^{\infty}_{+}\left( \mathcal{M} \right): \int_{\mathcal{M}} c\,dV = 1\}$ the configuration space of smooth and non-negative concentration fields $c\geq 0$ integrating to one on a manifold $\mathcal{M}$ with boundary $\partial\mathcal{M}$, as well as the process bundle $E$ with fibers $E_{c} =\{~w \in \Gamma(T\mathcal{M}): ~w\cdot ~n = 0\text{  on  }\partial\mathcal{M}\}$ consisting of (smooth) vector fields tangent to the boundary.  These fields affect the concentration by the continuity equation, encoded as $\dot{c} = \mathcal{P}_c ~w \coloneqq -\nabla\cdot c~w$ in terms of the process mapping $\mathcal{P}:E\to TQ$, which describes the conservation of mass for a substance of concentration $c$.  Observe that $\mathcal{P}$ is surjective since the tangent space to $Q$ at $c$ is $T_cQ = \{\mathring{c}\in C^{\infty}(\mathcal{M}): \int_{\mathcal{M}}\mathring{c}\,dV=0$\} (solve the Neumann problem $-\nabla\cdot c\nabla\phi = \mathring{c}$ and set $~w=\nabla\phi$).  Given the dissipation potential $\mathcal{D}_c\left(~w \right) = \frac{1}{2}\int_{\mathcal{M}}c\left| ~w \right|^{2}\,\mathrm{d}V$, the following result describes gradient flow generated by OVP. 


\begin{theorem}\label{thm:wasserstein-grad-flow}
    For any free energy $\mathcal{F}:Q \rightarrow \mathbb{R}$, the choices of $Q,E,\mathcal{D},\mathcal{P}$ above yield a gradient flow via OVP in terms of the chemical potential $\mu(c) = \mathrm{grad}_{L^2}\,\mathcal{F}|_c$: 
    \begin{equation*}
        \dot{c} = \nabla \cdot c\nabla\mu(c).
    \end{equation*}
\end{theorem}
\begin{proof}
    The result is a direct calculation using the gradient flow equation \eqref{eq:OVP-GF} along with the usual Riesz isomorphisms on $L^2$.  For completeness, a proof which does not rely on this additional structure is included in \Cref{app:omitted}. 
    
    The following observation will simplify things considerably.  Since the fibers $T_cQ$ are Hilbert spaces under the $L^2(\mathcal{M})$ inner product, it is convenient to identify the dual fibers $T_c^*Q\cong T_cQ$ with spaces of scalar fields.  Similarly, the dual fibers $E^*_c \cong E_c$ may be identified with spaces of vector fields under the $L^2(\mathcal{M},T\mathcal{M})$ inner product.  With this, the fiber derivative $\partial_{~w}\mathcal{D}:E\to E^*\cong E$ of $\mathcal{D}$ can be computed at the point $c\in Q$ by considering 
    \begin{equation*}
        \ip{\partial_{~w}\mathcal{D}_c(~w)}{\mathring{~w}} = \int_{\mathcal{M}} c~w\cdot\mathring{~w}\,\,dV = \ip{c~w}{\mathring{~w}},
    \end{equation*}
    establishing that the fiber-wise flat operator $\flat_c = \partial_{~w}\mathcal{D}_c$ is just multiplication by $c$.  Therefore, its inverse is given (at least formally) by $\sharp_c~w = c^{-1}~w$, and it remains to calculate the transpose of the process mapping $\mathcal{P}:E\to TQ$.  To that end, observe that for any $\mathring{c}\in T_cQ$,
    \begin{equation*}
        \ip{\mathring{c}}{\mathcal{P}_c{~w}} = -\int_\mathcal{M} \mathring{c}\nabla\cdot c{~w}\, dV = \int_\mathcal{M} c\nabla \mathring{c}\cdot{~w}\, dV,
    \end{equation*}
    since $~w\cdot~n=0$ on $\partial\mathcal{M}$ by construction.
    It follows that $\mathcal{P}^\top_c = c\nabla$, and combining this with the identification of $d\mathcal{F}$ with its $L^2$ Riesz representative $\mu = \mathrm{grad}_{L^2}\,\mathcal{F}$ finally yields
    \begin{equation*}
        \mathcal{P}\sharp\mathcal{P}^\top d\mathcal{F}\big|_c = \nabla\cdot cc^{-1}c\nabla\,\mathrm{grad}_{L^2}\,\mathcal{F}(c) = \nabla \cdot c\nabla\mu(c),
    \end{equation*}
    as claimed.  
\end{proof}

    

\begin{remark}
    Observe that points where $c=0$ can be handled without recourse to the nonexistent sharp operator $\sharp$ by observing that $c~v = -c\nabla\mu(c)$ and directly applying $\dot{c} = -\nabla\cdot c~v$.  More formally, $E_c$ should be thought of as the space of ``tangential fluxes'' $~j\in\Gamma(T\mathcal{M})$ with $~j\cdot~n =0$ on $\partial\mathcal{M}$. Then, $\mathcal{P}_c~w = -\nabla\cdot~j$ and the only division by $c$ occurs in the dissipation $\mathcal{D}_c(~j) = \frac{1}{2}\int_\mathcal{M}\frac{|~j|^2}{c}$ which can be defined in an extended sense.
\end{remark}


Interestingly, it can be shown that the term $\nabla\cdot c\nabla\mu(c)$ on the right-hand side of the gradient flow in \Cref{thm:wasserstein-grad-flow} is exactly the Wasserstein gradient $\mathrm{grad}_W\,\mathcal{F}$ of the free energy.  This provides a link between WGFs and OVP which is formalized in the following corollary.

\begin{corollary}
    Consider $Q$ as the space of probability densities with the Wasserstein inner product $W\in\Gamma\lr{T^*Q\otimes_{\mathrm{sym}}T^*Q}$ (see e.g. \cite[Section 8.2]{villani2021topics}), defined for tangent vectors $\mathring{c}_1,\mathring{c}_2\in T_c Q$ as
    \begin{equation*}
        W\lr{\mathring{c}_1,\mathring{c}_2} = \int_{\mathcal{M}} c\nabla C_1\cdot\nabla C_2\,dV \qquad \mathrm{where}\quad -\nabla\cdot c\nabla C_k=\mathring{c}_k, \quad \partial_{~n}C_k = 0, \quad k=1,2.
    \end{equation*}
    For any free energy $\mathcal{F}:Q\to\mathbb{R}$, OVP with $E,\mathcal{D},\mathcal{P}$ as in \Cref{thm:wasserstein-grad-flow} recovers the corresponding Wasserstein gradient flow:
    \begin{equation*}
        \dot{c} = -\mathrm{grad}_{W}\mathcal{F}\big|_{c} = \nabla\cdot c\nabla\,\mathrm{grad}_{L^2}\mathcal{F}\big|_{c}.
    \end{equation*}
\end{corollary}
\begin{proof}
    The fact that OVP recovers the stated gradient flow is the content of \Cref{thm:wasserstein-grad-flow}.  It remains to show that the right-hand side is also the Wasserstein gradient, i.e,  $\mathrm{grad}_W\mathcal{F}|_{c} = -\nabla\cdot c\nabla\,\mathrm{grad}_{L^2}\mathcal{F}|_{c}$.  To that end, denote $\mathring{a} = \mathrm{grad}_W\mathcal{F}|_c$ with $-\nabla\cdot c\nabla A = \mathring{a}$ and observe that for any $\mathring{c}\in T_cQ$,
    \begin{align*}
        \ip{\mathrm{grad}_{L^2}\mathcal{F}|_c}{\mathring{c}} = \int_\mathcal{M} c\nabla A \cdot \nabla C\,dV = -\int_{\mathcal{M}} A\nabla\cdot{c\nabla C}\,dV = \ip{A}{\mathring{c}}.
    \end{align*}
    Therefore, $A = \mathrm{grad}_{L^2}\mathcal{F}|_c$, so $\mathrm{grad}_W\mathcal{F}|_c = \mathring{a} = -\nabla\cdot c\nabla\,\mathrm{grad}_{L^2}\mathcal{F}|_c$ as desired.
\end{proof}

This result justifies the claim that solutions to OVP are WGFs for particular choices of functional data. Another immediate corollary of \Cref{thm:wasserstein-grad-flow} gives the Fokker-Planck, diffusion, and porous medium equations as WGF solutions to OVP.

\begin{corollary}\label{cor:wgf-examples}
    Given a potential $U:\mathcal{M} \rightarrow \mathbb{R}$, the free energy $\mathcal{F}(c) = \int_{\mathcal{M}}c\log c + cU\text{   }\mathrm{d}V$ along with the choices of $Q,E,\mathcal{D},\mathcal{P}$ in \Cref{thm:wasserstein-grad-flow} yield the Fokker-Planck equation $\dot{c} = \Delta c + \nabla \cdot (c\nabla U)$ through OVP, which reduces to the diffusion equation when $U\equiv 0$.  Using the free energy $\mathcal{F}(c) = \frac{1}{m-1}\int_{\mathcal{M}}c^{m}\,dV$ instead yields the porous medium equation $\dot{c} = \Delta c^{m}$.
\end{corollary}

It is worth noting that the tangentiality constraint $~w \cdot ~n = 0$ on $\partial \mathcal{M}$ inherent in the fibers $E_c$ of the process bundle can be lifted when $\mathcal{F}$ has a well defined $L^2$ gradient, at the expense of a different boundary condition imposed naturally.  To see this, let \(E_c \subset \Gamma(T\mathcal{M})\) be an unconstrained space of vector fields and recall that
\[
\mathcal{P}_c~w = - \nabla \cdot c ~w.
\]
Then, for any chemical potential \(\mu = \mathrm{grad}_{L^2} \mathcal{F}|_c\), it follows that
\[
\langle d\mathcal{F}_c\,|\,\mathcal{P}_c~w\rangle
=
\int_M \mu \, (-\nabla \cdot c ~w)\, dV
=
\int_M c \nabla \mu \cdot ~w\, dV
-
\int_{\partial M} c \mu\, ~w \cdot ~n\, dS.
\]
Accordingly, the first variation of the Rayleighian in the direction \(\mathring{~w}\) is given by
\[
\langle\partial_{~w} \mathcal{R}_c(~w)\,|\,\mathring{~w}\rangle
=
\int_M \bigl(c \nabla \mu + \partial_{~w}\mathcal{D}_c(~w)^\sharp \bigr)\cdot \mathring{~w}\, dV
-
\int_{\partial M} c \mu\, \mathring{~w} \cdot ~n\, dS,
\]
where \(\partial_{~w} \mathcal{D}_c(~w)^\sharp \in E_c\) denotes the \(L^2\)-Riesz representative of the fiber derivative. Since $\mathring{~w}$ is now unconstrained, this reduces to the force balance equation  
\[
c \nabla \mu + \partial_{~w} \mathcal{D}_c(~w)^\sharp = 0, \qquad \text{on } \mathcal{M}
\]
together with the natural boundary condition
\[
c \mu = 0
\qquad \text{on } \partial \mathcal{M}.
\]
Thus, there are (at least) two choices of $Q,E$ which lead to well posed formulations of OVP in this case: one where $Q$ is unconstrained and the fibers of $E$ are tangential to $\partial\mathcal{M}$, and another where both spaces are essentially unconstrained and $Q$ obtains a natural boundary condition.  Moreover, if $E_c$ is defined to be a space of ``fluxes'' $~j=c~w$ and the process mapping is $\mathcal{P}_c~j = -\nabla\cdot~j$, the second choice naturally enforces $\mu = 0$ on $\partial\mathcal{M}$.

\figDiffusion

An example of this flexibility is shown in \Cref{fig:dirichlet-vs-neumann}, where the free energy is given by $\mathcal{F}(c) = \int_\mathcal{M} c (\log c - 1 + \mu_0) \, dV$ in terms of the standard chemical potential $\mu_0$ (c.f. \cite{arroyo2017onsager}) and the process mapping is chosen as $\mathcal{P}_c~j = -\nabla\cdot~j$.  It follows that $\mu = \mathrm{grad}_{L^2} \mathcal{F} = \log c + \mu_0$, so that OVP implies the natural boundary condition $c = e^{-\mu_0}$ on $\partial\mathcal{M}$ when $Q$ is as before and $E_c$ is an unconstrained space of vector fields.   While the dynamics discovery experiments in \Cref{sec:numerics} will employ the tangentiality constraint on $E_c$ instead, this alternative approach may be useful in the case that natural boundary conditions on $Q$ are desired.

\subsubsection{Sobolev Gradient Flows}\label{subsec:sobolev}


Another class of equations with a nice connection to OVP are the Sobolev gradient flows (SGFs).  Of particular interest for this work are phase-field models \cite{chen2002phase,wu2020phase}, which model the concentration of a solute $c\in Q$ and capture phase separation behavior.  For example, the Cahn-Hilliard equation is useful for modeling di-block copolymers \cite{luo2023optimal,shi2013self,noshay2013block} which are polymer chains of blocks from two distinct species. The two polymer species generally repel each other, causing the solution to separate into distinct phases, while bonding between the species limits the extent of this separation and causes interesting patterns to form.  Several examples of such models will be considered in \Cref{sec:numerics}.

To explore the connection between SGFs and OVP, recall the Ginzburg-Landau (GL) free energy: 
\begin{equation}\label{eq:ginzburg-landau-fe}
    \mathcal{F}(c) = \int_{\mathcal{M}}\frac{\alpha}{2}\left| {\nabla c} \right|^{2} + F(c)\,\,dV
\end{equation}
where $F:Q \rightarrow Q$ represents the bulk free energy density which is generally empirical.  It is common to approximate this density with a double well potential such as the so-called Flory-Huggins logarithmic potential \cite{giorgini2020weak,padhan2025cahn},
\begin{equation}\label{eq:log-potential}
    F(c) = \frac{\theta}{2}\left\lbrack (1 - c)\log(1 - c) + (1 + c)\log(1 + c) \right\rbrack - \frac{\theta_{c}}{2}c^{2},
\end{equation}
where $\theta\in\mathbb{R}_{+}$ is the absolute temperature of the system, and $\theta_{c}\in\mathbb{R}_{+}$ is the critical temperature of phase separation.  Importantly, the parameters $\theta,\theta_c$ are material-dependent and often uncertain, making their determination a good candidate for replacement with a data-driven learning procedure. A further approximation to $F$ is sometimes called the ``shallow quenching'' limit \cite{wu2021review},
\begin{equation*}
    F(c) \approx \lr{\frac{\theta}{2}-\theta_c}c^2 + \frac{\theta}{12}c^4  \approx \frac{\left( 1 - c^{2} \right)^{2}}{4}.
\end{equation*}
In either case, note that the bulk free energy density has an interval on which its second derivative $F''(c)\leq 0$ is nonpositive.  This region is called the ``spinodal regime'' where concentrations tend towards phase separation.  The developments in \Cref{sec:inverse} will show how OVP can be used to determine $F$ from data in a way which preserves unconditional energy stability of the resulting discretization.


Assuming $c\in L^2(\mathcal{M})$ is square-integrable on a bounded domain $\mathcal{M}$, the $L^2$-gradient flow of the GL free energy yields the Allen-Cahn equation,
\begin{equation}\label{eq:Allen-Cahn}
\begin{cases}
    \dot{c} = \alpha\Delta c - F'(c) & \mathrm{in}\,\,\mathcal{M}, \\
    \partial_{~n}c = 0 & \mathrm{on}\,\,\partial\mathcal{M},
\end{cases}
\end{equation}
where $\Delta c$ is understood in the usual weak sense and the boundary condition in \eqref{eq:Allen-Cahn} ensures that the $L^2$-gradient $d\mathcal{F}(c)^\sharp = -\alpha \Delta c + F'(c)$ is well defined. Allen-Cahn is an effective model for phase transition in multi-component systems but does not conserve the total species concentration $\int_{\mathcal{M}} c$.  In cases where this is important, the Cahn-Hilliard model is more appropriate,
\begin{equation}\label{eq:Cahn-Hilliard}
\begin{cases}
    \dot{c} = \nabla\cdot ~M\nabla\mu(c) & \mathrm{in}\,\,\mathcal{M}, \\
    \mu = -\alpha\Delta c + F'(c) & \mathrm{in}\,\,\mathcal{M}, \\
    \partial_{~n}c = ~M\nabla\mu\cdot~n = 0 & \mathrm{on}\,\,\partial\mathcal{M}.
\end{cases}
\end{equation}
Here, $~M:T^*\mathcal{M}\to T\mathcal{M}$ is an invertible linear mobility field mapping forces to velocities, and the boundary conditions are necessary to ensure well posedness.  When $~M=I$ and so this object is trivial, it is well known that Cahn-Hilliard is the $H^{-1}$-gradient flow of the same GL free energy $\mathcal{F}$; A computation can be found in \Cref{thm:sobolev} in the Appendix.  As it turns out, these models are equally well expressed via OVP with certain choices of functional data.


\begin{theorem}\label{thm:AC-CH}
    Let $Q=C^\infty(\mathcal{M})$ and consider the Ginzburg-Landau free energy \eqref{eq:ginzburg-landau-fe}.  The Allen-Cahn model \eqref{eq:Allen-Cahn} follows from OVP with the choices $E=TQ$, $\mathcal{P}=I$, and $\mathcal{D}_c(\dot{c})=\frac{1}{2}\int_{\mathcal{M}}\dot{c}^2\,dV$.  Similarly, the Cahn-Hilliard model \eqref{eq:Cahn-Hilliard} follows from OVP with $E_{c} =\{~w \in \Gamma(T\mathcal{M}): ~w\cdot ~n = 0\text{  on  }\partial\mathcal{M}\}$, $\mathcal{P}_c~w = -\nabla \cdot ~w$, and $\mathcal{D}_c(~w) = \frac{1}{2}\int_{\mathcal{M}}\nn{~w}_{M^{-1}}^2\,dV$.
\end{theorem}
\begin{proof}
    Consider the variational derivative of $\mathcal{F}$ in the (unconstrained) direction $\mathring{c}\in T_cQ$,  
    \[ \dual{d\mathcal{F}_c}{\mathring{c}} = \int_{\mathcal{M}}\alpha \nabla c\cdot\nabla\mathring{c} + F'(c)\mathring{c}\,\,dV = \int_{\partial\mathcal{M}}\alpha \mathring{c}\,\partial_{~n}c\,\,dS + \int_\mathcal{M}\lr{F'(c) -\alpha\Delta c}\mathring{c}\,\,dV. \]
    Applying OVP with $\mathcal{P}=I$ and $\mathcal{D}_c(w) = \frac{1}{2}\int_{\mathcal{M}}w^2\,dV$ then yields the system minimizing 
    \[ \dual{d\mathcal{F}_c + \partial_w\mathcal{D}_c(w)}{\mathring{w}} = \int_{\mathcal{M}} \big(w - (\alpha\Delta c - F'(c))\big)\mathring{w}\,dV + \int_{\partial\mathcal{M}}\alpha\mathring{w}\,\partial_{~n}c\,dS = 0, \]
    among all variations $\mathring{w} \in E_c$.  This is equivalent to the Allen-Cahn system \eqref{eq:Allen-Cahn} since $w = \dot{c}$.
    To establish Cahn-Hilliard, consider the choices made in the statement and observe that the $L^2$-transpose of the process mapping $\mathcal{P}_c~w = -\nabla\cdot ~w$ (considered as a map $\mathcal{P}_c^\top: T_cQ \to E_c$ via the respective Riesz isomorphisms) requires that, for any $\mathring{c}\in T_cQ$,
    \[ \ip{\mathcal{P}_c~w}{\mathring{c}} = \int_{\mathcal{M}}-\lr{\nabla\cdot~w}\mathring{c}\,\,dV = \int_{\mathcal{M}}~w\cdot\nabla\mathring{c}\,\,dV - \int_{\partial\mathcal{M}} \mathring{c}\,~w\cdot~n\,\,dS = \big\langle~w,\mathcal{P}_c^\top\mathring{c}\big\rangle. \]
    In particular, $\mathcal{P}_c^\top\mathring{c} = \nabla\mathring{c}$ and  $~w\cdot~n$ must be $L^2(\partial\mathcal{M})$-orthogonal to the tangent space $T_cQ$ at the boundary $\partial\mathcal{M}$, which is guaranteed by the boundary condition $~w\cdot~n = 0$ on $\partial\mathcal{M}$ inherent in the fibers $E_c$.  Moreover, defining the chemical potential $\mu(c) = -\alpha\Delta c + F'(c) = d\mathcal{F}_c^{\sharp}$ via the $L^2$-gradient of $\mathcal{F}$ already implies the boundary condition $\partial_{~n}c = 0$.  It follows that OVP implies the system $\dot{c} = \mathcal{P}_c~v$ where $~v \in E_c$ satisfies 
    \begin{align*}
        \big\langle \mathcal{P}_c^\top d\mathcal{F}_c + \partial_{~w}\mathcal{D}_c(~v)\mid\mathring{~w}\big\rangle = \int_\mathcal{M} (\nabla \mu(c) + ~M^{-1}~v)\cdot \mathring{~w}\,dV = 0,
    \end{align*}
    for any admissible variation $\mathring{~w}$.
    This directly implies that $~v = -~M\nabla\mu(c)$ and therefore OVP is the statement $\dot{c} = \nabla\cdot~M\nabla\mu(c).$ This is precisely the Cahn-Hilliard system $\eqref{eq:Cahn-Hilliard}$, where the second boundary condition comes from $~v\cdot~n = 0$ in the definition of $E_c$. 
\end{proof}

\begin{remark}
    It is similarly straightforward to express Cahn-Hilliard in terms of the process mapping $\mathcal{P}_c~w = -\nabla\cdot c~w$ on the same fibers $E_c$ (now thought of as velocities), with the usual caveats about division by $c$.  The discretization in \Cref{sec:spacedisc} will adopt this perspective. 
\end{remark}

\section{Variational Time Discretization}\label{sec:timedisc}

OVP inherits nice behavior, such as energy-stability (c.f. \Cref{thm:continuous_ES}), due to its interpretation as a minimization principle.  With this in mind, it is now appropriate to discuss the discretization of OVP in a way which retains this character, hence provides a discrete analogue of the continuous variational structure.  For simplicity, discussion will be restricted to configuration spaces $Q$ which are vector spaces, although an extension to general manifolds is likely possible with an appropriate discretization of the exponential map.  

Despite its significant impact on physical understanding, there have been relatively few variational discretizations of OVP previously proposed in the literature \cite{chen2025onsager,zhu2025stokes,liu2024variational}.  Given discrete time instances $t_0< t_1 < ... < t_K$, the goal is to construct approximations $~q^k \approx ~q(t_k) \in Q$ of the time-continuous state $~q\in C^1([0,T],Q)$ which satisfy an appropriate discrete analogue of \Cref{def:OVP}.  In every case, this involves the design and subsequent minimization of a discrete Rayleighan $\mathcal{R}^d:E\to\mathbb{R}$, although the choice of process bundle $E$ and process mapping $\mathcal{P}$ may depend on the scheme employed.  In all previous works, the configuration space $Q\cong \mathbb{R}^n$ is identified with (perhaps some subset of) Euclidean space and the tangent bundle $E=TQ\cong Q\times Q$ is chosen as the process space.  Making first-order approximations to the time derivatives of the state and free energy, respectively,
\[D_{\Delta t}\big(~q^k,~q\big) =  \frac{~q-~q^k}{\Delta t}, \qquad D_{\Delta t}\,\mathcal{F}\big(~q^k, ~q\big) = \frac{\mathcal{F}\lr{~q}-\mathcal{F}\lr{~q^k}}{\Delta t},\]
then leads to the discrete Rayleighan, 
\begin{equation}\label{eq:chen-rayleighan}
    \mathcal{R}^d\big(~q^{k},~q\big) = D_{\Delta t}\,\mathcal{F}\big(~q^k, ~q\big) + \mathcal{D}\big(~q^k, D_{\Delta t}\big(~q^k,~q\big)\big), 
\end{equation}
which can be minimized directly for the time-updated state.  In fact, OVP in this setting is the update rule $~q^{k+1} = \argmin_{~q}\mathcal{R}^d\lr{~q^k,~q}$.

\begin{remark}\label{rem:first-order}
    The choice of a first-order derivative approximation $D_{\Delta t}$ is somewhat artificial and can likely be extended to higher-order approximation with standard techniques in variational integration \cite{marsden2001discrete}.  However, this will necessarily increase the cost, as multiple minimizations will be required per time step.
\end{remark}

~Despite its elegant formulation, the discrete OVP just discussed suffers a notable drawback: it cannot natively accommodate process bundles $E$ which differ from the tangent bundle $TQ$.  To handle systems with conservation constraints, such as the Wasserstein model discussed in \cref{subsec:wass_gfs}, one must resort to Lagrange multipliers in this framework, which increase the complexity of the optimization problem and bring the discrete interpretation further from the simple minimization in \Cref{def:OVP}.  To remedy this, we propose the following alternative discretization which recovers previous work as a special case.


\begin{definition}\label{def:discrete_OVP}
Given an Onsager system $\left( Q,E,\mathcal{F},\mathcal{D},\mathcal{P} \right)$ as in \Cref{def:OVP}, for which $Q$ is a vector space, consider the time-discrete Rayleighian $\mathcal{R}^{d}:E \rightarrow \mathbb{R}$,
\begin{equation}\label{eq:disc-rayleighian}
    \mathcal{R}^{d}\left(~{q},~{w}\right) = \mathcal{F}\lr{~{q} + \Delta t\,\mathcal{P}\lr{~q}~w} + \Delta t\,\mathcal{D}\left(~{q},~{w}\right),
\end{equation}
depending implicitly on the step-size $\Delta t > 0$.\footnote{Observe that the dimensional character of $\mathcal{R}^d$ here differs from the previous case by a factor of $\Delta t$, merely for convenience.} The time-discrete trajectory $~{q}^{\bullet}:[K] \rightarrow Q$, where $[K] = \{0,1,...,K\}$ and $~q^k \approx ~q(t_0 + k\Delta t)$, is a solution to the discrete OVP provided it satisfies the following  update equation for all $0\leq k\leq K-1$:
\begin{equation}\label{eq:disc-argmin}
~{q}^{k + 1} = ~{q}^{k} + \Delta t\,\mathcal{P}
\big(~{q}^{k}\big)~v, \qquad ~v = \argmin_{~{w} \in E_{~{q}^{k}}}
\mathcal{R}^{d}\big( ~{q}^{k},~{w} \big).
\end{equation}
\end{definition} 


Note that the discrete OVP in \Cref{def:discrete_OVP} is based on the same Taylor series expansion for the derivative of free energy used before,
\begin{equation*}
    \frac{\mathcal{F}\lr{~q^{k+1}} - \mathcal{F}\lr{~q^k}}{\Delta t} = \big\langle d\mathcal{F}\big(~q^k\big)\,\big|\,\mathcal{P}\big(~q^k\big)~v\big\rangle + \mathcal{O}\lr{\Delta t}.
\end{equation*}
However, it now involves a minimization which takes place in the fibers $E_{~q^k}$ of the process bundle, which are related to changes in configuration via the process mapping $\mathcal{P}:E\to TQ$.  This exactly recovers previous work based on the minimization of \cref{eq:chen-rayleighan} when the process mapping $\mathcal{P}=\mathcal{I}$ is trivial, while retaining the interpretation of OVP as a minimization problem in the presence of nontrivial $\mathcal{P}$, yielding additional flexibility in practical use.  For example, the space $E$ can be tailored to the problem at hand as in the continuous formulation (see \cite{arroyo2017onsager} for examples), and additional constraints can be incorporated into the fibers $E_{~q^k}$ as discussed in \Cref{sec:perspectives}.  Most importantly, it is easy to show that the integrator produced by \Cref{def:discrete_OVP} is unconditionally energy-stable.

\figFrames

\begin{theorem}\label{thm:discrete_ES}
Solutions to the discrete OVP \eqref{eq:disc-argmin} dissipate the free energy $\mathcal{F}$ independent of the step-size $\Delta t>0$.
\end{theorem}

\begin{proof}
Recall that the dissipation potential $\mathcal{D}$ is
fiber-wise non-negative and satisfies $\mathcal{D}( \cdot ,~0) = 0$.  Since $~v = \argmin_{~w}\mathcal{R}^d$ minimizes the discrete Rayleighan \eqref{eq:disc-rayleighian}, the free energy satisfies
\[ \mathcal{F}\big(~q^{k+1}\big) = \mathcal{F}\big(~q^k + \Delta t\,\mathcal{P}\big(~q^k\big)~v\big) \leq \mathcal{R}^{d}\big(~{q}^{k},~v\big) \leq \mathcal{R}^{d}\big(~{q}^{k},~0\big) = \mathcal{F}\big(~{q}^{k}\big), \]
for any time index $0\leq k\leq K-1$.
\end{proof}

Pseudocode corresponding to the discrete OVP \eqref{eq:disc-argmin} is given in \Cref{alg:OVP}, and example solutions to some relevant phase-field equations (\Cref{cor:wgf-examples}) are displayed in \Cref{fig:frames}. 
The spatial discretization is based on a staggered grid finite volume scheme detailed in \Cref{sec:spacedisc}.  The last column of \Cref{fig:frames} also confirms that these solutions are energy stable, since the free energy decreases monotonically over time.

\begin{algorithm}[!htb]
\caption{OVP Time Integration}\label{alg:OVP}
\begin{algorithmic}
\Require Discrete Onsager data $(Q,E,\mathcal{F},\mathcal{D},\mathcal{P})$ corresponding to \Cref{def:OVP}.  Step-size $\Delta t >0$.  Initial state $~q^0 \in Q$.
\Function{Step}{$\Delta t, \mathcal{F},\mathcal{D},\mathcal{P}, ~q^0$}
\State Set $~q \gets ~q^0$.
\State Form  $\mathcal{R}^{d}(~{q},~{w}) = \mathcal{F}(~{q} + \Delta t\,\mathcal{P}(~q)~w) + \Delta t\,\mathcal{D}(~{q},~{w})$ as in \cref{eq:disc-rayleighian}.
\State Compute $~v = \argmin_{~w} \mathcal{R}^{d}(~{q},~{w})$.   \Comment{We use L-BFGS.}
\State \Return $~q + \Delta t\,\mathcal{P}(~q)~v$.
\EndFunction
\end{algorithmic}
\end{algorithm}



\begin{remark}
    Note that the discrete Rayleighan $\mathcal{R}^d$ is no longer guaranteed to be convex in $~w$ due to nonconvexity in the free energy $\mathcal{F}$.  On the other hand, its Hessian contains an $\mathcal{O}(\Delta t)$ contribution from the Hessian of the convex dissipation potential $\mathcal{D}$ that dominates the $\mathcal{O}(\Delta t^2)$ contribution from the Hessian of $\mathcal{F}$ when $\Delta t$ is sufficiently small.  Therefore, shrinking the step-size can always restore convexity for regular enough $\mathcal{F}$. 
\end{remark}

It is worth mentioning how the update equation \cref{eq:disc-argmin} can expressed without the need for an explicit optimization, at least for quadratic dissipation potentials $\mathcal{D}$.  Consider the derivative of the discrete Rayleighan $\mathcal{R}^d$ from \eqref{eq:disc-argmin} in the direction $\mathring{~w}$, 
\[ \big\langle \partial_{~w}\mathcal{R}\big(~q^k,~w\big)\,\big|\,\mathring{~w} \big\rangle = \Delta t\,\big\langle d\mathcal{F}\big(~q^k+\Delta t\,\mathcal{P}\big(~q^k\big)~w\big)\,\big|\,\mathcal{P}\big(~q^k\big)\mathring{~w}\big\rangle + \Delta t\,\big\langle \partial_{~w}\mathcal{D}\big(~q^k,~w\big)\,\big|\,\mathring{~w}\big\rangle. \]
Setting this to zero yields an optimality condition for the minimization of $\mathcal{R}^d$,
\[ \mathcal{P}\big(~q^k\big)^\top d\mathcal{F}\big(~q^{k+1}\big) + \flat\big(~q^k\big)~v = \mathcal{P}\big(~q^k\big)^\top d\mathcal{F}\big(~q^k+\Delta t\,\mathcal{P}\big(~q^k\big)~v\big) + \partial_{~v}\mathcal{D}\big(~q^k,~v\big) =  0, \]
where the OVP update equation $~q^{k+1} = ~q^k+\Delta t\,\mathcal{P}\big(~q^k\big)~v$ was used along with the flat operator $\partial_{~v}\mathcal{D} = \flat$ associated to the quadratic dissipation potential. It follows that the minimizer is explicitly given by $~v = -\sharp\big(~q^k\big)\mathcal{P}\big(~q^k\big)^\top d\mathcal{F}\big(~q^{k+1}\big)$, and therefore the OVP update equation can be alternatively expressed as 
\begin{equation}\label{eq:disc-OVP-no-min}
        ~q^{k+1} = ~q^k - \Delta t\big(\mathcal{P}\sharp\mathcal{P}^\top\big)\big(~q^k\big)\, d\mathcal{F}\big(~q^{k+1}\big).
\end{equation}
In particular, the metric term $\mathcal{P}\sharp\mathcal{P}^\intercal$ coming from the phenomenological kinetic coefficients is evaluated at the old state  $~q^k$, while the derivative of the free energy is evaluated at the updated state $~q^{k+1}$.  This ``IMEX'' type method is distinctly different from a forward or backward Euler scheme, where both terms are evaluated at the same state.  However, if the coefficients $\mathcal{P}\sharp\mathcal{P}^\intercal$ happen to be state-independent, this discussion has proved the following result relating the discrete OVP \eqref{eq:disc-argmin} to the backward Euler method.

\begin{theorem}\label{thm:OVP-vs-BE}
    Suppose $\mathcal{D}$ is fiber-wise quadratic.  The discrete OVP \eqref{eq:disc-argmin} recovers backward Euler time integration applied to the gradient flow \eqref{eq:OVP-GF} when $\mathcal{P}\sharp\mathcal{P}^\top$ is independent of the state $~q$.
\end{theorem}

\subsection{Relationship to Other Integration Schemes}\label{subsec:others}

Before discussing how the discrete OVP \Cref{def:discrete_OVP} presented here can be used inside a machine learning workflow to enable dynamics discovery and surrogate modeling, it is worth drawing explicit connections to some related methods in the literature.  First, observe that invertibility of the process mapping $\mathcal{P}$ is sufficient for posing \eqref{eq:disc-argmin} as a minimization over the state $~q \in Q$, as discussed before, since this is equivalent to the existence of a bundle isomorphism $E\cong TQ$.  This is particularly true when the process mapping $\mathcal{P}=I$ is trivial, in which case the discrete Rayleighan $\mathcal{R}^d:Q\times Q \to \mathbb{R}$ reduces to the expression 
\[ \mathcal{R}^d\big(~q^k,~q\big) = \mathcal{F}(~q) + \Delta t\, \mathcal{D}\big(~q^k, D_{\Delta t}\big(~q^k,~q\big)\big). \]
This recovers \eqref{eq:chen-rayleighan} after a constant shift and appropriate rescaling, and is precisely the object considered in \cite{chen2025onsager} with associated discrete OVP  $~q^{k+1} = \argmin_{~q\in Q}\mathcal{R}^d\big(~q^k,~q\big)$.  Moreover, choosing the simplest quadratic dissipation potential $\mathcal{D}(~q,~w) = (1/2)\nn{~w}_{M}^2$ for some norm $\nn{\cdot}_M$ yields Di Giorgi's well known minimizing movement scheme \cite{giorgi1992movimenti,de1993new,laux2020thresholding,lu2026structure}, 
\begin{equation}\label{eq:minimizing-movement}
    ~q^{k+1} = \argmin_{~q\in Q}\mathcal{R}^d\big(~q^k,~q\big), \qquad \mathcal{R}^d\big(~q^k,~q\big) = \mathcal{F}(~q) + \frac{1}{2\Delta t}\big|~q-~q^k\big|^2_{M},
\end{equation}
a popular method for modeling evolutionary processes in low-regularity settings.  Note that \eqref{eq:minimizing-movement} is also known as the proximal point method when $\mathcal{F}$ is convex \cite{rockafellar1976monotone,cai2022developments,tran2025nonlinear}, in which case $\mathrm{prox}_{\Delta t\,\mathcal{F}}\big(~q^k\big) \coloneqq \argmin_{~q}\mathcal{R}^d\big(~q^k,~q\big)$ is called the proximal operator and OVP becomes the fixed point iteration $~q^{k+1} = \mathrm{prox}_{\Delta t\,\mathcal{F}}\big(~q^k\big)$.

Another interesting similarity arises in the context of optimal transport, where  \Cref{def:discrete_OVP} also includes the venerable Jordan-Kinderlehrer-Otto (JKO) scheme \cite{jordan1998variational} for simulating WGFs: considering a probability density $c\in Q$, 
\[ c^{k+1} = \argmin_{c\in Q}\mathcal{R}^d\big(c^k,c\big), \qquad \mathcal{R}^d\big(c^k,c\big) = \mathcal{F}(c) + \frac{1}{2\Delta t}W_2\big(c^k,c\big)^2. \]
This is not surprising given the connection to Di Giorgi, since JKO is a minimizing movement scheme with norm $\nn{\cdot}_M$ equal to the 2-Wasserstein distance \cite{villani2021topics},
\[ W_{2}\left( {c}_{0},{c}_{1} \right)^{2} = \inf\,\left\{ \int_{0}^{1}\left\| ~{w}(t) \right\|_{L^{2}\left( {c}(t) \right)}^{2}\,dt : \dot{{c}} + \nabla\cdot c~w = 0 \right\}. \]
This so-called Benamou-Brenier formula interprets the 2-Wasserstein distance as an infimum over time-dependent vector fields $~w$ that transport the probability density $c_0$ to $c_1$ under the continuity equation.  

\section{Dynamics Discovery with OVP}\label{sec:inverse}

A key advantage of OVP-based discretizations such as \Cref{def:discrete_OVP} is their guarantee of energy-stable dynamics independent of the governing potentials $\mathcal{F},\mathcal{D}$.  In the context of dynamics discovery, this enables physics-constrained predictions through unconstrained inference.  Remarkably, the difficult problem of learning structure-preserving dynamics consistent with a given set of data is lifted (through OVP) to the data-driven discovery of potentials whose dynamics agree with the collected observations.  This allows for a powerful hybrid modeling approach: the free energy $\mathcal{F}$ and dissipation potential $\mathcal{D}$ can include a mix of fixed and learnable terms, which are automatically processed by \Cref{def:discrete_OVP} into a structure-preserving dynamical evolution which couples the physical mechanisms encoded by these potentials. Additionally, the variational character of the system ensures that properties such as energy-stability hold independently of the data available or the quality of the training process.

To describe the benefits of this perspective more precisely, recall that practical systems of interest often include unknown or uncertain terms in their free energy $\mathcal{F}$ representing constitutive relationships or empirical fits to experimental data. For example, the shallow and deep quenching limits in the Cahn-Hilliard equation \cite{wu2021review} arise after simplification of the already approximate logarithmic potential \eqref{eq:log-potential}, based on the qualitative behavior of phase separation.  Even boundary conditions may be uncertain for some mixtures, making it especially difficult to come up with accurate models which reflect the underlying dissipative physics.  The proposed discrete OVP \Cref{def:discrete_OVP} offers a principled way to mitigate these issues: by including machine-learnable components in the free energy $\mathcal{F}$, salient features of the data under consideration can be preserved in a way which is compatible with the underlying thermodynamic structure, leading to a surrogate model with rigorous guarantees on its predictive behavior.

The strategy to accomplish this with the OVP \Cref{def:discrete_OVP} is to design parametric model classes $C_\mathcal{F} = \{\mathcal{F}_{\nu}: Q\to\mathbb{R}\,|\, \nu\in\Theta\}$ and $C_{\mathcal{D}} = \{\mathcal{D}_{\eta}: E\to\mathbb{R} \,|\, \eta\in\Theta\}$, which include learnable parameters $(\nu,\eta) \coloneqq \theta \in \Theta\times\Theta$ lying in some convex set $\Theta\subset\mathbb{R}^p$ and respect the structural constraints in \Cref{def:OVP}.  Along with an appropriate process mapping $\mathcal{P}:E\to TQ$ encoding problem structure, this defines a parametric family of discrete Rayleighans $\mathcal{R}^d_{\theta}:E\to\mathbb{R}$ as in \eqref{eq:disc-rayleighian}, which can be trained implicitly from observable data and are guaranteed to produce discrete energy-stable trajectories $~{q}^{\bullet}:[K] \rightarrow Q$ satisfying $\mathcal{F}_{\theta}(~q^{k+1}) \leq \mathcal{F}_{\theta}(~q^k)$ for all $k\in [K]$.  Importantly, there is no need to directly supervise the functionals $\mathcal{F}_{\nu},\mathcal{D}_{\eta},$ or $\mathcal{R}^d_{\theta}$ during training, as parameter gradients will flow backward from quantity-of-interest-based output metrics through OVP to update $\mathcal{R}_{\theta}$ in the desired way.  For simplicity, the experiments in \Cref{sec:numerics} consider the most straightforward case, where supervision is done using simulated trajectory data.  However, the same procedure is applicable in much more generality, providing a rich avenue for future work.  For example, realistic settings may target observable quantities (e.g., material properties or spectral content) derived from trajectory information.

\begin{remark}
    Note that the process mapping $\mathcal{P}$ could also be parameterized and included as part of the learning problem.  On the other hand, $\mathcal{P}$ is not usually uncertain: for all experiments in this work, it is either trivial or enforces the continuity equation.  
\end{remark}


\subsection{Related Work}
While the present work is the first discretely structure-preserving dynamics discovery method rooted in OVP, it is not the first OVP-based framework for dynamical inference.  The OnsagerNet method \cite{yu2021onsager} defines trainable neural networks based on a ``generalized Onsager's principle'', guaranteeing structure-preservation at the time-continuous (but not time-discrete) level.  Variational Onsager neural networks (VONNs) \cite{huang2022variational} were introduced as another way to satisfy OVP at the continuous level with machine learnable potentials.  After defining differentiable models for the free energy and dissipation potential, VONNs use collocation data to minimize the residual of the OVP optimality conditions (analogous to \cref{eq:onsager_optimality}) during training, so that finite difference time discretization can be applied to the learned network potentials at inference time to generate predictions.  Another work which is perhaps the most similar to our approach is \cite{lu2026structure}.  Here, the authors aim to learn dissipation potentials within di Giorgi's minimizing movement scheme, so that unconditional energy stability is preserved.  However, this continuous-time formulation is not capable of guaranteeing discrete energy stability, and does not consider the preservation of associated invariants such as mass.  Alternative strategies include the deep Onsager-Matchlup method (DOMM) \cite{li2023deep}, which combines the deep Ritz method \cite{weinan2018deep} for unsupervised training with an Onsager-Matchlup variational principle to produce a space-time solution network which is physics-informed. Further, unsupervised operator learning based on OVP was considered in \cite{chang2025unsupervised}, where direct minimization of a continuous Rayleighan was used to generate a flux operator whose associated conservation law was integrated with forward Euler to produce dynamics.  

Considering the approaches mentioned in \Cref{subsec:others} which are similar to OVP, the literature is even more numerous.  A regularization technique for module-wise neural network training, based on minimizing movement approaches, was introduced in \cite{karkar2023module} and shown to improve the performance of greedy network training procedures.  A minimizing movement scheme for the Willmore flow of surfaces was introduced in \cite{rumpf2025hybrid}, which uses a learnable notion of mean curvature to improve simulation speed.  Several more works consider learnable variants of the JKO scheme for sampling: the concurrent efforts \cite{mokrov2021large,alvarez2021optimizing,bunne2022proximal} apply input convex neural networks \cite{amos2017input} to the learnable free energy in JKO to maintain structural rigidity for high-dimensional inference, 
\cite{xu2023normalizing} provides a simplified normalizing flow strategy based on following discrete Wasserstein gradient trajectories with comparable performance to standard normalizing flows \cite{rezende2015variational}, and \cite{hertrich2024importance} combines local Wasserstein gradient steps with nonlocal rejection-resampling to mitigate slow convergence for multimodal distributions.  It is remarkable that OVP provides a unifying perspective of these methods which have quite different motivations and applications.  


\begin{algorithm}[!htb]
\caption{Dynamics Discovery with OVP}
\label{alg:learning}
\begin{algorithmic}[1]
\Require Discrete Onsager data \((Q, E, \mathcal{F}_\theta, \mathcal{D}_\theta, P)\) with trainable potentials \(\mathcal{F}_\theta\) and/or \(\mathcal{D}_\theta\), integrator step-size \(\Delta t > 0\), one-step training pairs \(\{(~q_n, ~q_n^+)\}_{n=1}^N\), observable \(O : Q \to \mathbb{R}^m\), number of training epochs \(N_e > 0\), batch size \(B > 0\), gradient descent step-size \(h > 0\).
\State Form the Rayleighan \(\mathcal{R}_\theta^d(~q,~w) = \mathcal{F}_\theta(~q + \Delta t\,\mathcal{P}(~q)~w) + \Delta t\,\mathcal{D}_\theta(~q,~w)\) as in \eqref{eq:disc-argmin}.
\For{epoch \(e = 1,\dots,N_e\)}
    \State Partition \([N]\) uniformly into \(N_b = \lfloor N/B \rfloor\) mini-batches \(B_1,\dots,B_{N_b}\).
    \For{batch \(b = 1,\dots,N_b\)}
        \State Loss \(\mathcal{L} \gets 0\).
        \For{data index \(n \in B_b\)}
            \State Compute evolved state \(~y_n = \mathrm{Step}(\mathcal{R}_\theta^d)(~q_n)\) with \Cref{alg:OVP}.
            \State Augment loss \(\mathcal{L} \mathrel{+}= | O(~y_n) - O(~q_n^+) |^2\).
        \EndFor
        \State Normalize \(\mathcal{L} \gets \mathcal{L}/|B_b|\) by batch size.
        \State Compute the gradient \(\nabla_\theta\mathcal{L}\) with automatic differentiation and \eqref{eq:no_backprop}. \Comment{Using the converged minimizer $~v$.}
        \State Update parameters \(\theta \gets \mathrm{Adam}(\nabla_\theta \mathcal{L}, h)\).
    \EndFor
\EndFor
\Ensure Trained potentials \(\mathcal{F}_\theta\) and \(\mathcal{D}_\theta\).
\end{algorithmic}
\end{algorithm}

\subsection{Training Procedure}
To illustrate how \Cref{alg:OVP} is used inside a framework for dynamics discovery, consider a dataset $D = \cup_{i=1}^{I} \{(~q_i^k,O(~q_i^k))\}_{k=0}^{K}$ consisting of $I$ trajectories of length $K$ along with an observable $O: Q \to \mathbb{R}^m$ measured along them.\footnote{We consider here only the case where $O$ is measured instantaneously.  The case where $O$ measures across larger numbers of time steps is feasible but left for future work.}  The goal is to learn a discrete Rayleighan functional $\mathcal{R}^d_{\theta}$ which generates dynamics compatible with the observations in $D$.  To that end, we minimize an empirical approximation to the following loss function, 
\begin{align}\label{eq:loss}
    \mathcal{L}(D, \theta) = \mathbb{E}_{i\sim \mathcal{U}([I])}\,\mathbb{E}_{k\sim\mathcal{U}([K-s])}\,\mathbb{E}_{s\sim\mathcal{U}([S])} \sum_{t=0}^{s}\big|O\big(~q_{i}^{k+t}\big) - O\circ \text{ Step}(\mathcal{R}^d_{\theta})^{\circ t}\big(~q_{i}^{k} \big)\big|^2,
\end{align}
where $\mathcal{U}$ denotes the uniform distribution, $S>0$ is an integer denoting the maximum consecutive integration steps, and $\mathrm{Step}(\mathcal{R}^d_{\theta})^{\circ t}$ denotes $t$ successive applications of OVP \Cref{alg:OVP} with Rayleighan $\mathcal{R}^d_{\theta}$.  The resulting minimizer $(\nu^*, \eta^*) = \theta^* = \argmin_{\theta} L(D,\theta)$ then parameterizes potential functions $\mathcal{F}_{\nu^*}$ and $\mathcal{D}_{\eta^*}$ whose OVP dynamics accurately describe the training data and generalize out-of-distribution in a physically meaningful way.  \Cref{alg:learning} provides a description of the training process employed in this work, where $S=1$ and $\mathcal{L}$ is approximated by the mean squared error over a ``flattened'' dataset $D = \{(~q_n,~q_n^+)\}_{n=1}^{N}$ of compatible pairs drawn from the set of all discrete trajectories.  Here, $~q_n = ~q_i^k$ denotes a snapshot drawn from the $i^{\rm th}$ trajectory at the $k^{\rm th}$ time step (for some $1\leq i\leq I$ and $1\leq k\leq K-1$), while $~q_n^+ = ~q_i^{k+1}$ denotes the snapshot along the same trajectory at the next time step.  For convenience, learnable functionals are denoted as functions of $\theta$ (i.e., by $\mathcal{F}_{\theta}$ and $\mathcal{D}_{\theta}$) in the remainder of the work.

\begin{remark}
    When the system state $~q$ is not directly observable, the dynamics discovery \Cref{alg:learning} can be applied by incorporating arbitrary-length rollouts from a known (or otherwise fixed) initial condition $~q_0$ instead of single-step rollouts from intermediate data.
\end{remark}

Observe that minimizing the objective $\mathcal{L}$ in \cref{eq:loss} with a gradient-based optimization algorithm requires differentiating through the variational integrator step $~q^{k+1} = \mathrm{Step}(\mathcal{R}^d_{\theta})\big(~q^k\big)$ computed with \Cref{alg:OVP}.  
In particular, this involves propagating parameter gradients through the minimization $~v = \argmin_{~w\in E_{~q}}\mathcal{R}^d_{\theta}(~q,~w)$ defining the update in the process bundle.  While this can certainly be done by differentiating directly through each iteration of the nonlinear solver used to compute $~v$, this option is expensive and prone to the accumulation of errors.  Instead, we adopt an alternative and more efficient approach, outlined in \cite{blondel2022efficient}, which applies implicit differentiation to the first-order optimality conditions of this minimization.
To see how this works, recall that invertibility of the Hessian $\partial^2_{~w}\mathcal{R}^d_{\theta}(~v)$ at the minimizer guarantees its implicit expression 
\[\partial_{~w}\mathcal{R}^d_{\theta}(~q,~v(\theta)) = 0,\]
after invoking the usual implicit function theorem.  It follows that  the parameter gradient $\partial_{\theta}~v$ can be computed by differentiating this expression in $\theta$,
\begin{align}\label{eq:no_backprop}
\partial_{\theta}~v = -\big(\partial_{~w}^2\mathcal{R}^d_{\theta}\big)^{-1}\partial_{\theta}\partial_{~w}\mathcal{R}^d_{\theta},
\end{align}
where 
the Hessian $\partial_{~w}^2\mathcal{R}^d_{\theta}$ can be inverted with a standard algorithm such as the conjugate gradient method. The experiments in \Cref{sec:numerics} utilize the implementation of \cref{eq:no_backprop} from the TorchOpt Python package \cite{ren2023torchopt} with an additional Tikhonov regularization of $\mathcal{O}(10^{-8})$, and gradient descent with an Adam update rule is used for the nonconvex ``outer'' optimization over parameters.

\section{Spatial Discretization}\label{sec:spacedisc}
To showcase the results of the proposed OVP-based \Cref{alg:learning} for learning dissipative dynamics, consider 2-D finite differences on a regular marker-and-cell (MAC) grid \cite{harlow1965numerical} (also called an Arakawa C-grid \cite{arakawa1977computational}) of $n_x\times n_y$ cells.\footnote{We emphasize that this choice is for example purposes only.  The OVP-based technology developed here is compatible with any desired spatial discretization.} 

As depicted in \Cref{fig:mac_grid}, the scalar field $c\in Q$ is stored as an array of values at cell centers $C\in \mathbb{R}^{n_x\times n_y}$, while the vector field $~w\in E_c$ is stored as a tuple of arrays $~W = (W^x,W^y)$ with $W^x\in\mathbb{R}^{(n_x+1)\times n_y}$ and $W^y\in\mathbb{R}^{n_x\times(n_y+1)}$. Once these choices are made, only discretizations of the functional data $\mathcal{F},\mathcal{D}$ and $\mathcal{P}$ need be provided in order to simulate PDEs governed by OVP (c.f. \Cref{alg:OVP}), since all other discrete differential operators will emerge automatically through the minimization process.  This is a key advantage of the OVP-based machinery presented here, which simplifies the problem of forward simulation provided the resulting optimization problem can be efficiently solved.

\begin{wrapfigure}{r}{0.30\textwidth}
    \centering
    \includegraphics[width=0.25\textwidth]{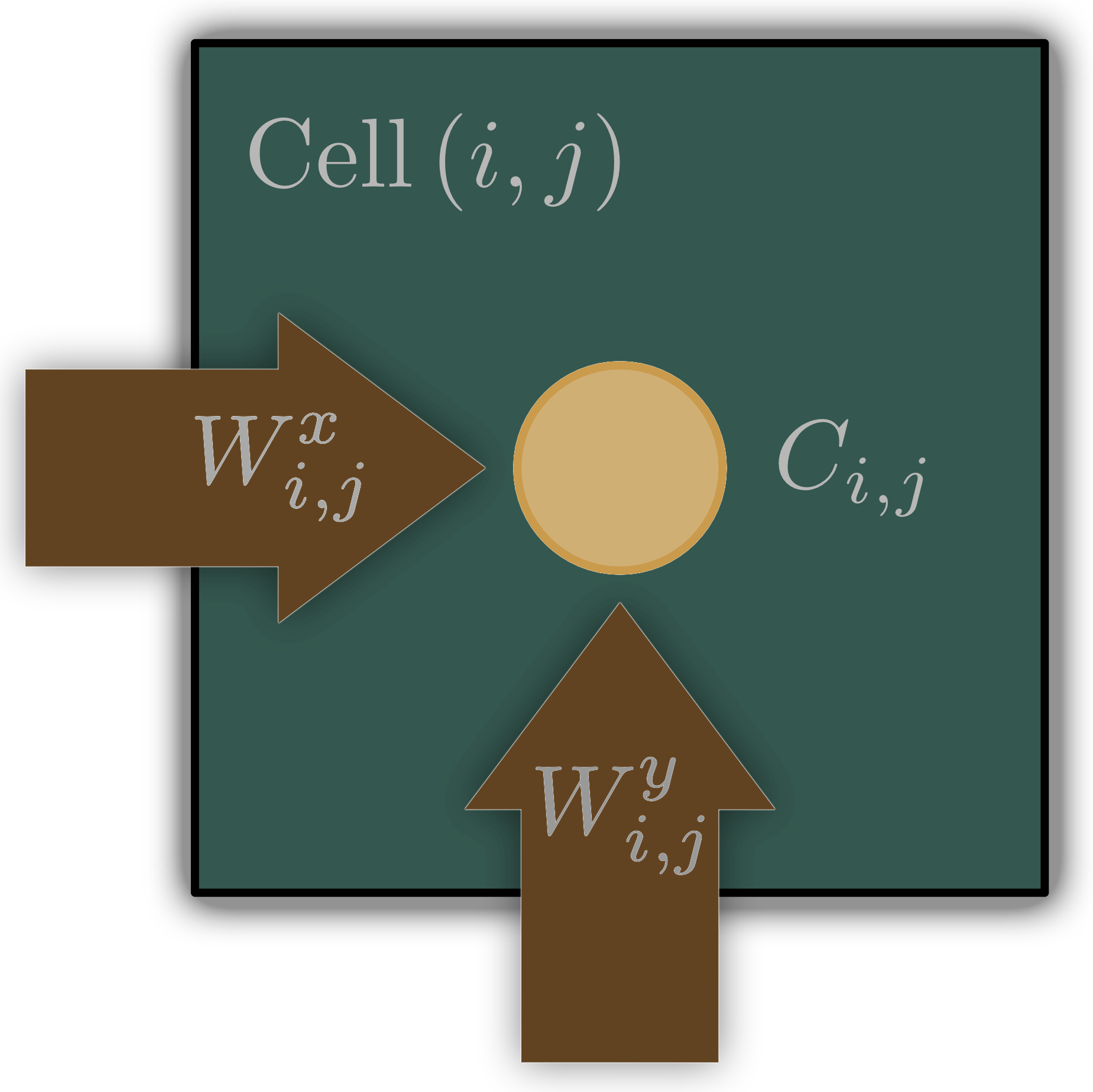}
    \caption{One cell of the MAC grid used for spatial discretization.}
    \label{fig:mac_grid}
\end{wrapfigure}

The PDEs implemented for illustration  follow the Wasserstein and Sobolev gradient flow examples in \Cref{cor:wgf-examples,thm:AC-CH}.
Two discrete process mappings will be considered:  the identity mapping $\mathcal{P}_I(C)W = W$ in the case of Allen-Cahn (where $W \in \mathbb{R}^{n_x\times n_y}$ is cell-centered), and the discrete divergence 
\begin{align*}
    [\mathcal{P}(C)~W]_{i,j} = -\lr{\frac{J^x_{i+1,j}-J^x_{i,j}}{\Delta x} + \frac{J^y_{i,j+1}-J^y_{i,j}}{\Delta y}},
\end{align*}
representing the continuous mapping $\mathcal{P}_{c}~w = -\nabla\cdot c~w$
in terms of the ``discrete fluxes'' at interior edges, 
\begin{align*}
    J^x_{i+1,j} &= \frac{1}{2}\lr{C_{i,j}+C_{i+1,j}}W^x_{i+1,j}, \qquad 1\leq i \leq n_x-1,\\
    J^y_{i,j+1} &= \frac{1}{2}\lr{C_{i,j}+C_{i,j+1}}W^y_{i,j+1}, \qquad 1\leq j\leq n_y-1.
\end{align*}
Boundary values for $J^x,J^y$ are then assigned based on the desired OVP boundary conditions.  Typically, the no-flux condition $~w\cdot~n=0$ on $\partial\mathcal{M}$ is enforced via the assignment $J_{1,j}^x=J_{n_x+1,j}^x=J_{i,1}^y=J^y_{i,n_y+1}=0$.

The central-difference gradient of a grid function, stored at interior edges, is denoted by $\nabla C =(\partial_xC,\partial_yC)\in\mathbb{R}^{(n_x-1)\times n_y}\times\mathbb{R}^{n_x\times (n_y-1)}$ with components 
\begin{align*}
     (\partial_xC_{i+1,j}, \partial_yC_{i,j+1}) = \lr{\frac{C_{i+1,j}-C_{i,j}}{\Delta x}, \frac{C_{i,j+1}-C_{i,j}}{\Delta y}}.
\end{align*}
This is used to define the Dirichlet energy 
\[\mathcal{E}_D(C) = \frac{1}{2}\left[\sum_{i=1}^{n_x-1}\sum_{j=1}^{n_y} (\partial_xC_{i+1,j})^2 + \sum_{i=1}^{n_x}\sum_{j=1}^{n_y-1}(\partial_yC_{i,j+1})^2\right]\Delta x\Delta y.\]
To compute dissipation measures, it is convenient to have an averaged representation of the vector field $~W$ (and also $\nabla C$) at cell centers.  To that end, consider the quantities 
\begin{align*}
    ~W_{i,j} &= \lr{\frac{W^x_{i,j}+W^x_{i+1,j}}{2}, \frac{W^y_{i,j}+W^y_{i,j+1}}{2}}, \qquad |~W_{i,j}|^2_M = \big\langle ~W_{i,j},~M_{i,j}~W_{i,j} \big\rangle,
\end{align*}
where $~M\in\mathbb{R}^{n_x\times n_y\times 2\times 2}$ is a discrete mobility field stored at cell centers\footnote{The experiments in \Cref{sec:numerics} use $~M_{i,j}=I$, but the discretization holds more generally}.
These choices allow definition of the discrete free energies under consideration, 
\begin{equation}\label{eq:discrete_free_energies}
\begin{split}
    \mathcal{F}_{h,{\rm FP}}(C) &= \sum_{i,j} \lr{ C_{i,j}\log C_{i,j} + C_{i,j}U_{i,j} }\Delta x \Delta y, \\
    \mathcal{F}_{h,{\rm GL}}(C) &=  \mathcal{E}_D(C) + \sum_{i,j} F(C_{i,j})\Delta x\Delta y, \\ 
    \mathcal{F}_{h,{\rm PM}}(C) &= \sum_{i,j}C_{i,j}^m\,\Delta x\Delta y,
\end{split}
\end{equation}
along with discrete dissipation potentials corresponding to these free energies,
\begin{equation}\label{eq:discrete_dissipation_potentials}
\begin{split}
    \mathcal{D}_{h,{\rm FP}}(C,~W) &= \frac{1}{2}\sum_{i,j} C_{i,j}\big|~W_{i,j}\big|^2\,\Delta x\Delta y, \\
    \mathcal{D}_{h,{\rm AC}}(C,~W) &= \frac{1}{2}\sum_{i,j} \big|~W_{i,j}\big|^2\,\Delta x\Delta y, \\
    \mathcal{D}_{h,{\rm CH}}(C,~W) &= \frac{1}{2}\sum_{i,j}C_{i,j}^2 \big|~W_{i,j}\big|^2_{M^{-1}}\,\Delta x\Delta y.
\end{split}
\end{equation}
Importantly, observe that the choices of discrete gradient $\nabla C$ and process mapping $\mathcal{P}(C)~W$ are discrete adjoints.  To see this, let $A\in\mathbb{R}^{n_x\times n_y}$ be a grid function and consider the summation-by-parts relationships
\begin{align*}
    -\sum_{i=1}^{n_x} A_{ij}(J^x_{i+1,j}-J^x_{i,j}) &= A_{1,j}J^{x}_{1,j} - A_{n_x,j}J^x_{n_x+1,j} + \sum_{i=2}^{n_x} (A_{i,j}-A_{i-1,j})J^x_{i,j} = \sum_{i=2}^{n_x} \partial_xA_{i,j}J^x_{i,j}\Delta x, \\
    -\sum_{j=1}^{n_y} A_{ij}(J^y_{i,j+1}-J^y_{i,j}) &= A_{i,1}J^{y}_{i,1} - A_{i,n_y}J^y_{i,n_y+1} + \sum_{i=2}^{n_y} (A_{i,j}-A_{i,j-1})J^y_{i,j} = \sum_{i=2}^{n_y} \partial_yA_{i,j}J^y_{i,j}\Delta y,
\end{align*}
where the isolated boundary terms vanish due to the no-flux conditions $J_{1,j}=J_{n_x+1,j}=0$ and $J_{i,1}=J_{i,n_y+1}=0$.  It follows that 
\begin{align*}
    \langle \mathcal{P}(C)~W,A\rangle &= -\sum_{i=1}^{n_x}\sum_{j=1}^{n_y} \lr{\frac{J^x_{i+1,j}-J^x_{i,j}}{\Delta x} + \frac{J^y_{i,j+1}-J^y_{i,j}}{\Delta y}}A_{i,j}\Delta x\Delta y \\
    &= \left( \sum_{i=2}^{n_x}\sum_{j=1}^{n_y} \frac{A_{i,j}-A_{i-1,j}}{\Delta x}J^x_{i,j} + \sum_{i=1}^{n_x}\sum_{j=2}^{n_y}\frac{A_{i,j}-A_{i,j-1}}{\Delta y}J^y_{i,j} \right)\Delta x\Delta y \\
    &= \sum_{i=2}^{n_x}\sum_{j=1}^{n_y}\partial_xA_{i,j} J_{i,j}^x\Delta x\Delta y +  \sum_{i=1}^{n_x}\sum_{j=2}^{n_y} \partial_yA_{i,j}J_{i,j}^y\Delta x\Delta y = \langle ~J,\nabla A\rangle, 
\end{align*}
in terms of the flux $~J = (J^x,J^y)$ defined previously.   All together, this framework provides a second-order and adjoint-consistent spatial discretization for the Allen-Cahn, Cahn-Hilliard, and Fokker-Planck models considered in the next section. 
\section{Numerical Results}\label{sec:numerics}
Using the spatial discretization just described, the OVP-based dynamics discovery \Cref{alg:learning} is now illustrated on several proof-of-concept examples of interest.  Emphasis is placed on demonstration of the core functionality, with more sophisticated applications reserved for future work.


\subsection{Recovery of Unknown Free Energy Density}\label{subsec:polynomial}
As a first demonstration, consider the problem of recovering an unknown free energy density from trajectory data for a system of Cahn-Hilliard type \cref{eq:Cahn-Hilliard}.  
Recall the discrete Ginzburg-Landau free energy $\mathcal{F}_{h,{\rm GL}}$ from \cref{eq:discrete_free_energies}, which includes a bulk free energy density $F(C)$ discretizing the continuous $F(c)$.  The goal is to recover the 
discrete function $F$ given snapshots of a corresponding Cahn-Hilliard model with constant mobility $~M = I$.  To this end, three different bulk free energy densities are considered:
\begin{enumerate}
    \item The quartic function $F(c) = \tfrac{1}{4}(1-c^2)^2$ representing the ``shallow quenching'' limit potential.
    \item The Flory-Huggins logarithmic potential $F(c) = \tfrac{\theta}{2}[(1-c)\log(1-c)+(1+c)\log(1+c)] - \tfrac{\theta_c}{2}c^2$ with $\theta=2$ and $\theta_c=2.7$,
    capturing competition between mixing entropy and molecular interactions.
    \item An asymmetric potential $F(c) = \tfrac{1}{4}c^4 - \tfrac{1}{2}c^2 + \tfrac{1}{3}c^3$ that mimics the case where mixture components have different molecular volumes (see \cite{weber2019physics}).
\end{enumerate}

\figCahnHilliard

Representative evolutions for each of these models, simulated using \Cref{alg:OVP}, are displayed in Figure~\ref{fig:threeCahnHilliard}, which are  referred to as ``Quartic'', ``Flory-Huggins'', and ``Asymmetric'', respectively.
The parameterizations of the learnable density $F_{\theta}$ considered for dynamics discovery include (1) a generic degree-six polynomial $F_{\theta}(C) = a_0+a_1C+...+a_6C^6$ in the cell centers (7 parameters $\theta = \{a_0,...,a_6\}$, mildly over-parameterized), along with (2) a shallow component-wise neural network $F_{\theta}(C_{i,j}) = w_2^\top\sigma(w_1C_{i,j} + b_1)$ with $w_1,w_2,b_1\in\mathbb{R}^{32}$ and $\sigma=\tanh$ activation function (96 parameters $\theta = \{w_1,w_2,b_1\}$, highly over-parameterized). 
Note that the absolute scale of the learnable free energy density $F_{\theta}$ does not affect the downstream dynamics, since the minimizer depends on the functional derivative $F_{\theta}'$ and its spatial derivatives $\Delta F_{\theta}'$ and $\nabla_{~n}F_{\theta}'$.  To reduce indeterminacy, we apply a simple gauge fix: the value $F_{\theta}(0)$ and the average slope $(1/2)(F_{\theta}(1)-F_{\theta}(-1))$ of the learnable surrogate are subtracted at each evaluation.\footnote{In practice, we encountered relatively smooth training both with and without gauge fixing.}  This has the net effect of enforcing $F_{\theta}(0)=0$ and (by the mean value theorem) $F_{\theta}'(\zeta)=0$ for some $\zeta\in[-1,1]$.  The free energies are trained according to \Cref{alg:learning} with a mean squared error loss function between the exact data $C$ and the reconstruction $\tilde{C}$.

 
 


To generate the requisite training data, the OVP integrator \Cref{alg:OVP} is simulated from the data $\mathcal{F}_{h,{\rm GL}}$ and $\mathcal{D}_{h,{\rm CH}}$ \eqref{eq:discrete_dissipation_potentials} with process mapping $\mathcal{P}$ on a regular $100^2$ grid.  For each bulk free energy density, data is collected from simulations involving 50 different initial concentration fields generated via Gaussian blurrings of a single random field, each integrated 20 steps in time with $\alpha=1$ and a step-size of $\Delta t = 10^{-4}$.  Since this is not enough time for the system to equilibrate, these short integration periods create a heterogeneous dataset with relatively good coverage of the level sets $F^{-1}(c)$.  Accordingly, reasonable recovery of the bulk free energy density is expected, along with reasonable trajectory-wise performance on a hold-out set of validation data. 

\figFEloss
\figFEgrid

\Cref{fig:61training,fig:61grid} show the loss curves obtained during training, along with the epoch-wise change in the corresponding learned free energy densities.  Observe that the curves quickly converge to meaningful representations of these densities with extrema in the expected places.  This indicates that the learning process has effectively captured the macroscopic physical behavior reflected in the equilibrium states of the system, despite the fact that the training data does not contain equilibrated solutions.  On the other hand, all models struggle to learn the tails of the free energy densities, and this is most pronounced in the Flory-Huggins case.  Here, both the neural network and polynomial models do not recover the full magnitude of the singularities at $c=\pm 1$, with the neural model exhibiting the greatest deviation.  This is likely due to the relative sparsity of near-singular concentrations in the training data: the large energy cost of these concentrations immediately drives them closer to equilibrium under Cahn-Hilliard evolution.  Regardless, it is clear that \Cref{alg:learning} is able to recover meaningful bulk free energy densities from these concentration data.

\tabThreeCases


A more quantitative view of this experiment is provided by \Cref{tab:threecases}, where the learned models are evaluated on the validation data.  Besides the training and validation losses, the time- and parameter-averaged relative $L^2$ errors in trajectory prediction along with the bulk free energy density relative $L^2$ errors are measured, 
\[RL^2(~q,\tilde{~q}) = \frac{1}{IK}\sum_{i=1}^{I}\sum_{k=1}^K \frac{| ~q_i^k - \tilde{~q}_i^k |_{L^2}}{|~q_i^k|_{L^2}}, \quad R\ell^2(F, F_{\theta}) = \frac{|F-F_{\theta}|_{L^2}}{|F|_{L^2}}, \]
where $|F|_{L^2} = \sqrt{\Delta s\sum |F(s_i)|^2}$ for uniformly spaced points $s_i\in[-1,1]$.
Evidently, both the polynomial and neural network models are able to reasonably predict the dynamics in every case, with relative $L^2$ errors ranging from $4\text{-}9\%$ in trajectory prediction.  Recovery of the bulk free energy density is similarly good in the quartic and asymmetric cases, with errors of $1\text{-}7\%$.  The case of Flory-Huggins is more interesting: recovery is quite good away from the tails where it is somewhat poor, leading to errors of roughly $20\%$.  Since the OVP dynamics will aggressively drive concentrations toward the wells, this is likely due to a heavy bias away from these regions in the training dataset.

Additional information is provided by \Cref{tab:poly}, which records the weights (coefficients) of the polynomial free energies. Observe that the largest coefficients are reasonably well captured in all cases except for Asymmetric, but all polynomial degrees contribute meaningfully to the learned surrogates.  If desired, this could likely be controlled through the use of a sparsity-promoting regularizer during training.  In any case, it is remarkable that the use of OVP guarantees stability and mass conservation under arbitrarily long rollouts in all cases considered, so that essential physics are preserved independently of the learning process.

\tabPoly



\figLearnedEvolutions

\Cref{fig:learned_evolutions} displays a representative evolution of the learned models in the asymmetric case, using an initial condition from the validation set and rolled out for ten times the length of the training window.  Observe that almost all the error in the polynomial and neural network surrogates is concentrated away from the phase field interface.  This indicates that the learned free energies in \Cref{fig:61grid} are capable of producing plausible evolutions that closely match the topology of the corresponding Cahn-Hilliard solutions.

\subsection{Recovery of Unknown Boundary Contributions}


It is frequently the case that field data are presented without explicit knowledge of the underlying boundary conditions that generated them. Even trajectories which are computed with fixed BCs may present with dynamically varying ones if snapshots of the entire domain are not provided. Moreover, realistic phase field models often include boundary terms in their free energy that describe coupling of the bulk concentration to a substrate or another material \cite{wu2021review}.  The next test of \Cref{alg:learning} considers the recovery of a boundary potential enforcing its own BCs.  

First, recall the example from \Cref{subsec:sobolev} that showed how the configuration and process spaces $Q,E$ determine the domains of the process mapping $\mathcal{P}$ and its transpose $\mathcal{P}^\top$ along with the BCs enforced by OVP.  When the process mapping $\mathcal{P}_c~w = -\nabla\cdot c~w$ and the training data varies nontrivially at the boundary, it is convenient to consider the ``default'' process bundle with fibers $E_c = \{~w\,|\,~w\cdot~n = 0\,\,\mathrm{on}\,\partial M\}$, so that the configuration space $Q$ can remain essentially unconstrained.  Additional natural conditions exhibited by the data can then be enforced implicitly via OVP, perhaps via additional terms in $\mathcal{F},\mathcal{D}$.

The present experiment involves a free energy $\mathcal{F}$ that includes a surface term.  For example, consider the Ginzburg-Landau free energy with surface potential, 
\begin{align}\label{eq:GL-surf}
    \tilde{\mathcal{F}}_{\rm GL}(c) = \mathcal{F}_{\rm GL}(c) + \int_{\partial\mathcal{M}} G(c)\,dS,
\end{align}
where $G:Q\to Q$ is a surface free energy density \cite{tsori2001diblock,wu2006guided}.  This leads to a Rayleighan of the form
\begin{align*}
    \mathcal{R}(c,~w) &= \dual{d\mathcal{F}_{\rm GL}(c)}{\mathcal{P}(c)~w} + \dual{G'(c)}{\mathcal{P}(c)~w}_{\partial M} + \mathcal{D}(c,~w) \\
    &= \ip{F'(c)-\alpha\Delta c}{\mathcal{P}(c)~w} + \ip{\alpha\partial_{~n}c + G'(c)}{\mathcal{P}(c)~w}_{\partial \mathcal{M}},
\end{align*}
where $\dual{\cdot}{\cdot}_{\partial \mathcal{M}}$ (resp. $\ip{\cdot}{\cdot}_{\partial \mathcal{M}}$) denotes the relevant duality pairing (resp. inner product) at the boundary, and the process mapping $\mathcal{P}(c)$ at the boundary is understood in the restricted sense.  It can be shown (see \Cref{app:omitted}) that when $\mathcal{P}_c~w = -\nabla\cdot c~w$ is the process mapping enforcing the continuity equation and the process variables $~w\in E_c$ are tangentially constrained, this restriction (or trace) has the formal expression $$(\mathcal{P}_c~w)|_{\partial\mathcal{M}} = -\nabla_{\partial}\cdot c~w + \nabla_{~n}~n\cdot c~w,$$
where $\nabla_{\partial}$ denotes the connection along the boundary $\partial\mathcal{M}$ induced by the connection $\nabla$ on the interior of $\mathcal{M}$.
Accordingly, the minimization of $\mathcal{R}$ inherent in OVP yields the same Cahn-Hilliard equations \eqref{eq:Cahn-Hilliard} in the interior and the same boundary condition $~M\nabla\mu\cdot~n = 0$ on $\partial\mathcal{M}$ inherited from fiber-wise tangentiality, but also the more complex boundary condition 
$$c\nabla_{~n}~n + c\nabla_{\partial}(\alpha\nabla_{~n}c + G'(c))=0 \quad {\rm on}\,\,\partial\mathcal{M}.$$ 
This describes how the curvature of the boundary interacts with the derivative of the surface density, providing an additional natural condition which can be calibrated to data via $G$.

For simplicity, the present experiment considers the case when the process mapping $\mathcal{P}=I$ is trivial, yielding the Allen-Cahn equation \eqref{eq:Allen-Cahn} with dynamic boundary condition 
\[\alpha\partial_{~n}c + G'(c) = 0\quad  on\,\, \partial\mathcal M.\] 
The goal is to investigate the performance of the dynamics discovery \Cref{alg:learning} in recovering an asymmetric polynomial profile $G(c) = c^2-4c^3+3c^4$.  This induces regions of the boundary with contact angles substantially different than $90^\circ$, reflecting the wetting affinity of the substrate.  To accomplish this, training data are prepared on a $100^2$ grid as before and the OVP integrator \Cref{alg:OVP} is simulated using the dissipation potential $\mathcal{D}_{h,{\rm AC}}$ along with the free energy 
\begin{equation}\label{eq:boundary_free_energy}
    \mathcal{F}_h(C) = \mathcal{F}_{h,{\rm GL}}(C) + \sum_{j=1}^{n_y} \bigl(G(C_{1,j}) + G(C_{n_x,j})\bigr)\Delta y + \sum_{i=2}^{n_x-1} \bigl(G(C_{i,1}) + G(C_{i,n_y})\bigr)\Delta x , 
\end{equation}
where the bulk free energy density $F(c) = (1/4)(1-c^2)^2$ in $\mathcal{F}_{h,{\rm GL}}$ is fixed at the quartic potential.  Again, data are collected from simulations involving 50 different initial concentration fields generated via Gaussian blurrings of a single random field, each integrated 20 steps in time with $\alpha=1$ and a step-size of $\Delta t = 10^{-3}$.  Since this is not enough time for the system to equilibrate, training data are relatively balanced and reasonable recovery of the surface potential $G$ is expected.




\figBClossdens

 \Cref{fig:bc_lossdens} displays the results of inferring $G = G_{\theta}$ in \eqref{eq:boundary_free_energy} using the polynomial and shallow neural network parameterizations from the previous experiment with a gauge fix for the constant term.   Notice that both are quite accurate in the $[-0.5, 1]$ regime, where deviations in the value of the potential are relatively mild.
 Conversely, the neural network model has a much more difficult time capturing the left tail near $c=-1$, likely because it is highly overparameterized and this regime is less present in the training data.  Both learned boundary profiles are reasonably different from the ground truth: $16.4\%$ for the polynomial model and $30.9\%$ for the neural network model.  Despite this, both models train reasonably well, with the neural network model producing solutions that are more accurate trajectory-wise but less accurate in recovering the surface free energy density.  

\figBCprofiles

Additional understanding is provided by \Cref{fig:bc_profiles}, which shows the trajectory-wise error incurred in a representative evolution from the validation set when rolled out for ten times the training window. Notice that the phase-field interface no longer intersects the boundary at right angles, illustrating the influence of the potential $G_{\theta}$ on the boundary condition.  Observe also that both learned models incur the majority of their error at interfaces, and the character of this error is substantially different between the models.  For example, the neural model tends to over-predict the concentration almost uniformly, while the polynomial model can over- or under-predict depending on the region.  Remarkably, trajectory-wise errors are small without the extrapolatory rollout: $1.66\%$ for the polynomial case and $0.7\%$ for the neural network case.  
However, this increases substantially during extrapolation, where trajectory-wise errors reach $23.1\%$ using the  polynomial surrogate and $17.7\%$ using the neural network one.  Regardless, the learned boundary potentials have the correct extrema and the monotonic free energy decay guaranteed by OVP is observed in all surrogate models, demonstrating that \Cref{alg:learning} is reasonably effective at dynamics discovery in this setting as well.






\subsection{Recovery of Nonlocal Potentials}


A popular variant of Cahn-Hilliard used to model di-block copolymer physics is given by the Ohta-Kawasaki free energy \cite{ohta1986equilibrium,cao2022globally},
\begin{align*}
    \mathcal{F}_{\text{OK}}(c) &= \int_{\mathcal{M}}\left(F(c) + \frac{\alpha}{2}\left| {\nabla c} \right|^{2} + \frac{\sigma}{2}\left( c - \bar{c} \right)( - \Delta)^{- 1} \left( c - \bar{c} \right)\right)\mathrm{d}V \\
    &= \mathcal{F}_{\rm GL}(c) + \frac{\sigma}{2}|c-\bar{c}|^2_{H^{-1}(\mathcal{M})},
\end{align*}
where $\bar{c} = \frac{1}{{\rm vol}\,\mathcal{M}}\int_\mathcal{M}c\,dV$ denotes the average value of $c$.  The inverse operator $(-\Delta)^{-1}$ is interpreted as the inverse of the (positive definite) Laplacian with homogeneous Neumann boundary conditions that defines the usual $H^{-1}$-norm $|\cdot|_{H^{-1}(\mathcal{M})}$ on the domain.  A similar calculation to that of \Cref{thm:AC-CH} (assuming trivial mobility $~M=I$) shows that the free energy $\mathcal{F}_{\rm OK}$ along with the dissipation $\mathcal{D}$ and process mapping $\mathcal{P}$ from the statement induce (via OVP) the PDE system
\begin{equation}\label{eq:OKCH}
\begin{cases}
    \dot{c} = \Delta\mu(c) - \sigma(c-\bar{c}) & \mathrm{in}\,\,\mathcal{M}, \\
    \mu = -\alpha\Delta c + F'(c) & \mathrm{in}\,\,\mathcal{M}, \\
    \partial_{~n}c = \nabla\mu\cdot~n = 0 & \mathrm{on}\,\,\partial\mathcal{M}.
\end{cases}
\end{equation}
This coincides with the usual Cahn-Hilliard model up to the affine shift $\sigma(c-\bar{c})$ in the interior equation.  Importantly, notice that the equations \eqref{eq:OKCH} are local but arise from a free energy $\mathcal{F}_{\rm OK}$ that is highly nonlocal, requiring concentration information from across the entire domain.   




The system \eqref{eq:OKCH} presents interesting challenge for the dynamics discovery \Cref{alg:learning}: is it possible to learn the inverse Neumann Laplacian $(-\Delta)^{-1}$ from concentration data?  To investigate this, it is useful to recall that $(-\Delta)^{-1}$ is diagonalized in Fourier space. Therefore, it is reasonable to pose the learning problem in terms of an appropriate spectral filter.  Considering the process space $E$ as before with fibers $E_c$ consisting of vector fields satisfying the tangentiality constraint $~w\cdot ~n = 0$ at the boundary $\partial\mathcal{M}$ and discretized as in \Cref{sec:spacedisc}, the relevant discrete Fourier transform is given by the well studied discrete cosine transform (DCT-II) \cite{strang1999discrete}.  Given index ranges $0\leq i,k\leq n_x-1$, $0\leq j,\ell\leq n_y-1$, consider the bases 
\[\phi_k(i) = \alpha_k\cos\Big[\frac{\pi}{n_x}(i+\tfrac12)k\Big], \quad \psi_\ell(j) = \beta_\ell\cos\Big[\frac{\pi}{n_y}(j+\tfrac12)\ell\Big], \]
where the normalization constants 
\begin{align*}
    \alpha_0 &= \sqrt{\frac{1}{n_x}}, \qquad \alpha_k = \sqrt{\frac{2}{n_x}}, \quad k\geq 1, \\ 
    \beta_0 &= \sqrt{\frac{1}{n_y}}, \qquad \beta_{\ell} = \sqrt{\frac{2}{n_y}}, \quad \ell\geq 1,
\end{align*}
are chosen to guarantee discrete orthonormality.  Letting $\Phi_{k,\ell}(i,j) = \phi_k(i)\psi_{\ell}(j)$ denote the tensor product basis, the DCT-II of a grid function $C\in\mathbb{R}^{n_x\times n_y}$ is defined for frequencies $k,\ell$ 
as 
\[\hat{C}_{k,\ell} = \sum_{i=0}^{n_x-1}\sum_{j=0}^{n_y-1}C_{i+1,j+1}\Phi_{k,\ell}(i,j). \]
In this basis, the inverse Neumann Laplacian $(-\Delta)^{-1}$ is a diagonal linear operator of size $n_xn_y\times n_xn_y$ with nonzero entries 
\[ a_{k,\ell} = \frac{1}{\lambda_k + \lambda_{\ell}}, \quad \lambda_k = \frac{2}{\Delta x^2}\bigg(1-\cos\Big(\frac{k\pi}{n_x}\Big)\bigg), \quad \lambda_\ell = \frac{2}{\Delta y^2}\bigg(1-\cos\Big(\frac{\ell\pi}{n_y}\Big)\bigg). \]
It is straightforward to check that $\{\lambda_k\}$ (resp. $\{\lambda_\ell\}$) are the eigenvalues of the three-point discrete Laplacian in the 1D basis of cosines along $x$ (resp. along $y$).  However, observe that $(-\Delta)^{-1}$ is not itself a tensor product of 1D operators: its entries depend on a mixture of information from both spatial dimensions.  

This spectral discretization of the inverse discrete Neumann Laplacian yields an efficient computation of the $H^{-1}$ norm appearing in the Ohta-Kawasaki free energy $\mathcal{F}_{\rm OK}$.  Since the average value $\bar{C}$ is proportional to the $(0,0)$-component of the Fourier representation $\hat{C}$ via $\hat{C}_{0,0} = \sqrt{n_xn_y}\bar{C}$, it follows that
\[\big|{C}-{\bar{C}}\big|^2_{H^{-1}(\mathcal{M})} = \sum_{(k,\ell)\neq(0,0)}\frac{\hat{C}_{k,\ell}^2}{\lambda_k + \lambda_{\ell}}\Delta x\Delta y,\]
meaning that the nonlocal part of the free energy $\mathcal{F}_{\rm OK}$ can be expressed as a quadratic form in frequency space.  Therefore, it is reasonable to pose the problem of learning $(-\Delta)^{-1}$ in terms of learning a constant matrix operator (i.e., a convolution kernel) acting on the Fourier coefficients $\hat{C}$, possibly with some band-limit on the allowed frequencies.  This is analogous to the main idea behind Fourier Neural Operators (FNOs) \cite{li2020fourier}, which propagate information using learnable neural networks in Fourier space.  In the present case, this means learning an approximation to the discrete Ohta-Kawasaki potential,
\begin{align*}
    \mathcal{F}_{h,{\rm OK}}(C) = \mathcal{F}_{h,{\rm GL}}(C) + \frac{\sigma}{2}\big|C-\bar{C}\big|^2_{H^{-1}(\mathcal{M})},
\end{align*}
where the $H^{-1}$ norm is replaced with $\mathrm{vec}(\hat{C})^\intercal~B\,\mathrm{vec}(\hat{C})$ for a learnable matrix $~B$. 
Two experiments will be considered: one which explicitly uses the diagonal structure of the Neumann Laplacian operator, and one which does not. For consistency, both will use the potential strength $\sigma=250$.

\subsubsection{Experiment 1: Using Diagonality}
The first experiment carried out in this setting attempts to learn a frequency-space matrix $~A\in\mathbb{R}^{n_x\times n_y}$ of positive diagonal entries, so that the learnable free energy takes the form
\begin{align*}
    \mathcal{F}_{h,\theta}(C) = \mathcal{F}_{h,{\rm GL}}(C) + \frac{\sigma}{2}\sum_{(k,\ell)\neq (0,0)} \frac{\hat{C}_{k,\ell}^2}{[~A]_{k\ell}}.
\end{align*}
Instead of letting all entries of $~A$ be learnable scalars with positivity constraints, it is assumed that $n_x=n_y$ and these entries take the form
$[~A]_{k\ell} = [{\rm softplus}\,~a]_k + [{\rm softplus}\,~a]_\ell$ in terms of an unconstrained learnable vector $~a\in\mathbb{R}^{n_x}$ and the element-wise activation function ${\rm softplus}(x) = \ln{(1+e^x)}$.

\figOneDeigenvalues
\figOneDevolutions

As in previous examples, the dataset consists of 50 initial conditions comprised of random normal concentration fields blurred by (random uniform) Gaussian filters and simulated for 20 steps each with a step-size of $\Delta t= 10^{-4}$.  The data are split 90/10 for training/validation, and the initial learning rate for the Adam optimizer is 0.1.  The results of this experiment are displayed in \Cref{fig:OK_1D_eigenvalues}, where it is remarkable that the first 20 eigenvalues (minus the first) are learned reasonably well while the remaining are not.  This has a consequence on the rollouts depicted in \Cref{fig:OK_1D_evolutions}, where differences in the ground truth and learned solutions become visibly apparent around $3\times$ the training window. On the other hand, the large-scale behavior of the system, including striping and overall topology, is reasonably well captured even at $10\times$ the training window. 

\subsubsection{Experiment 2: Using Low-Frequency Information}

The next experiment involving learning the influence of the Ohta-Kawasaki potential reduces the amount of assumptions put on the learnable model form.  Rather than assuming the desired operator is diagonal in frequency space, the learnable approximation will rely on only the first $r\ll n_x$ largest frequencies.  In particular, the
learnable free energy will take the form
\begin{align*}
    \mathcal{F}_{h,\theta}(C) = \mathcal{F}_{h,{\rm GL}}(C) + \frac{\sigma}{2}\sum_{(k,\ell) \leq (r,r)}\sum_{(m,n)\leq (r,r)} \hat{C}_{k,\ell}\,[~K]_{k\ell,mn}\,\hat{C}_{m,n},
\end{align*}
for a matricized operator $~K \in\mathbb{R}^{r^2\times r^2}$ acting on the first $r^2$ Fourier coefficients.  This is a common strategy employed in FNO-like settings (e.g., \cite{li2020fourier}) that reduces the complexity of the learning problem while also increasing its robustness to high-frequency noise.  Note that an explicit symmetrization $~K= {\rm softplus}(~M + ~M^\intercal)$ for learnable $~M\in\mathbb{R}^{r^2\times r^2}$ is included to reduce indeterminacy: the entries of $~K$ should be positive, and only the symmetric part of $~K$ contributes to the frequency-space norm.

The experimental task is to learn a low-frequency operator $~K={\rm softplus}(~M + ~M^\intercal)$
from solutions to OVP with free energy $\mathcal{F}_{h,{\rm OK}}$ and dissipation potential $\mathcal{D}_{h,{\rm CH}}$.  Given the previous experiment where eigenvalues of the Laplacian were targeted, $r=20 \ll 100 = n_x$ is chosen as the band limit.  Therefore, it is expected that the spectrum of $~K$ will not necessarily approximate the first $r^2$ eigenvalues of $(-\Delta)^{-1}$, since the entries of $~K$ have to compensate for the high-frequency information that is left out.  The same training and validation data are used in this case, with the same learning rate of 0.1.  \Cref{fig:OK_2D_eigenvalues} compares the spectra of the (inverse) learned operator with the first $20^2$ (smallest) eigenvalues of the Laplacian that have the largest effect on the $H^{-1}$ norm.  Notice that the spectrum of $~K$ is substantially different; it exhibits a similar general trend in magnitude with larger frequencies corresponding to larger values, but it is less smooth and there are large values around the diagonal that do not appear in the ground truth case.  These differences propagate into the rollouts depicted in \Cref{fig:OK_2D_evolutions}, where deviations between the learned model and the ground truth are even more pronounced than in the previous experiment.  Conversely, the use of OVP in the surrogate ensures that all trajectories still remain stable, mass-conserving, and monotonically energy decreasing according to the governing physics.  It is expected that trajectory-wise accuracy could be substantially improved with a more sophisticated model class and learning strategy.





\figTwoDeigenvalues
\figTwoDevolutions

\section{Conclusion and Future Work}\label{sec:conclusion}

Onsager's Variational Principle (OVP) has been discussed for the purposes of variational integration and dynamics discovery.  Through an appropriate discrete formulation, it has been shown that OVP-based timestepping preserves key invariants and exhibits  unconditional energy stability at the fully discrete level, guaranteeing structural properties similar to the continuous variational setting.  This led to an algorithm for dynamics discovery based on learned corrections to functional-level information.  Using observational or simulated data to calibrate the free energy functional in an OVP-based integrator yields surrogate solutions that are stable and appropriately conservative under arbitrarily long rollouts, regardless of what is learned during training.  Examples from phase-field modeling have demonstrated that this procedure is effective in recovering the effects of unknown bulk free energy densities, boundary potentials, and nonlocal integral kernels from simulated data.  

Future efforts will focus on multiphysics coupling and the incorporation of more sophisticated learning strategies. In addition to the bulk-boundary interactions considered here, OVP can also be used to encode more complicated physical mechanisms, such as fluid-structure interactions and material contact, through additional terms in the governing potentials.  It is interesting to consider the design of surrogate models for such phenomena via the proposed strategy of OVP integration with learnable functionals, potentially using model classes like Gaussian processes that include a notion of uncertainty.  Additionally, it remains to be seen how well this OVP-based approach is able to learn from ``real-world'' observational data that contains noise and other defects.  Since the variational integrator presented here is guaranteed to generalize appropriately (though perhaps not accurately) out of distribution, we expect that this strategy may help separate useful signal from sparse and noisy data in resource-limited settings.

\section*{Acknowledgments}

The authors thank Oliver Gross for significant help in the design and creation of Figure 1, leading to greatly improved clarity and visual appeal.
Support for this work was received through the U.S. Department of Energy, Office of Science, Office of Advanced Scientific Computing Research, Mathematical Multifaceted Integrated Capability Centers (MMICCS) program, under Field Work Proposal 22025291 and the Multifaceted Mathematics for Predictive Digital Twins (M2dt) project.  
Sandia National Laboratories is a multimission laboratory managed and operated by National Technology \& Engineering Solutions of Sandia, LLC, a wholly owned subsidiary of Honeywell International Inc., for the U.S.~Department of Energy's National Nuclear Security Administration under contract DE-NA0003525.
This paper describes objective technical results and analysis. Any subjective views or opinions that might be expressed in the paper do not necessarily represent the views of the U.S.~Department of Energy or the United States Government.
\appendix

\begin{appendices}

\section{Omitted computations}\label{app:omitted}
Here we record the computations proving some results used in the body.

\begin{proof}[Alternative proof of \cref{thm:wasserstein-grad-flow}]\label{pf:no-isomorphism}
      First, observe that the dual fibers $E_c^* = \{~\alpha \in \Omega^n(\mathcal{M})\otimes \Omega^1(\mathcal{M})\}$ contain $n$-form valued differential $1$-forms on the $n$-manifold $\mathcal{M}$.  It is convenient to write $~\alpha = ~w^\flat \otimes dV$ in terms of the flat operator\footnote{Note that this is distinct from the flat operator $\flat:E\to E^*$ used in OVP.} with respect to the (potentially non-Euclidean) dot product $\cdot$ and the volume form $dV$ on $\mathcal{M}$, so that the evaluation pairing on $E_c^*\otimes E_c$ is given by 
    \begin{equation*}
        \dual{~\alpha}{\mathring{~w}} = \int_M ~w^\flat(\mathring{~w})\,dV = \int_M ~w\cdot\mathring{~w}\,dV
    \end{equation*}
    for $~\alpha\in E_c^*$ and $\mathring{~w}\in E_c$.  With this, the fiber derivative $\partial_{~w}\mathcal{D}$ of $\mathcal{D}$ can be computed at the point $c\in Q$ by considering 
    \begin{equation*}
        \dual{\partial_{~w}\mathcal{D}_c(~w)}{\mathring{~w}} = \int_M c~w\cdot\mathring{~w}\,dV = \dual{c~w^\flat\otimes dV}{\mathring{~w}},
    \end{equation*}
    establishing that the fiber-wise flat operator $\flat_c = \partial_{~w}\mathcal{D}_c: E_c\to E_c^*$ is given by $~w\mapsto c~w^\flat\otimes dV$.  Therefore, its inverse $\sharp_c:E_c^*\to E_c$ is $~w^\flat\otimes dV = ~\alpha\mapsto c^{-1}~w$.  It remains to calculate the transpose $\mathcal{P}^\top:T^*Q\to E^*$.  For any density $\beta = f\,dV\in = T_c^*Q$, it follows that 
    \begin{equation*}
        \dual{\beta}{\mathcal{P}_c\mathring{~w}} = -\int_M f\nabla\cdot c\mathring{~w}\, dV = \int_M c\nabla f\cdot \mathring{~w}\, dV = \dual{c\,df\otimes dV}{\mathring{~w}} = \dual{\mathcal{P}^\top_c\beta}{\mathring{~w}},
    \end{equation*}
    since $df = (\nabla f)^\flat \in\Omega^1(\mathcal{M})$ and $\mathring{~w}\cdot~n = 0$ on $\partial M$.  This demonstrates that $\mathcal{P}^\top$ is the mapping (suppressing tensor products) $f\,dV \mapsto c\,df\, dV$.  Finally, it follows that $\mathcal{P}\sharp\mathcal{P}^\intercal|_c$ maps $f(c)\,dV \mapsto -\nabla \cdot c\nabla f(c)$.  The conclusion follows upon observation that $\mathcal{P}$ is surjective and $d\mathcal{F} = \mathrm{grad}_{L^2}\,\mathcal{F}\otimes dV$.
\end{proof}

\begin{theorem}\label{thm:bndry_decomp}
    Suppose $(M^n,\nabla) \subset (N^m, \bar{\nabla})$ is an inclusion of Riemannian manifolds with Levi-Civita connections $\nabla$ and $\bar{\nabla}$. Denoting both the metric on $N$ and its pullback on $M$ by $\ip{\cdot}{\cdot}$, it follows that the divergence of a vector field $~w$ on $N$ decomposes as 
    \begin{equation*}
        \bar{\nabla}\cdot~w = \nabla\cdot~w_{\parallel} - \ip{~H}{~w_{\perp}} + ~P^\perp\lr{\bar{\nabla}\cdot}~w,
    \end{equation*}
    in terms of the orthogonal projections $~w_{\parallel},~w_{\perp}$ of this field on the tangent resp. normal bundles to $M$ in $N$, the mean curvature vector field $~H = \mathrm{tr}\,\mathbb{II}$ which is the trace of the second fundamental form, and the orthogonal projection of the divergence $~P^\perp\lr{\bar{\nabla}\cdot}~v$ onto the normal bundle $(TM)^\perp$.  Here, $\mathbb{II}:TM\times TM\to TM^\perp$ is defined as $\mathbb{II}(~v,~w) = \lr{\bar{\nabla}_{~v}~w}^\perp$ and $~P^\perp\lr{\bar{\nabla}\cdot}~v$ is the trace of the quadratic form $\ip{\cdot}{\bar{\nabla}_{(\cdot)}~v}: TM^\perp\times TM^\perp \to \mathbb{R}$.
\end{theorem}
\begin{proof}
    Recall the Gauss-Codazzi equations for the decomposition of the connection $\bar{\nabla}$ relative to $M$ in $N$, which state that, for vector fields $~v,~w_{\parallel}\in \Gamma\lr{TM}$ in the tangent bundle to $M$, the ambient connection decomposes as 
    \[ \bar{\nabla}_{~v}~w_{\parallel} = \lr{\bar{\nabla}_{~v}~w_{\parallel}}^\top + \lr{\bar{\nabla}_{~v}~w_{\parallel}}^\perp = \nabla_{~v}~w_{\parallel} + \mathbb{II}\lr{~v,~w_{\parallel}}, \]
    where $\top$ and $\perp$ denote orthogonal projections to $TM$ and $TM^\perp$, relative to $TN$.  Observe that, for $~w_{\perp}\in\Gamma\lr{TM^\perp}$, there is a complementary splitting in terms of the induced connection $\nabla^\perp$ in $TM^\perp$,
    \[\bar{\nabla}_{~v}~w_{\perp} = \lr{\bar{\nabla}_{~v}~w_{\perp}}^\top + \lr{\bar{\nabla}_{~v}~w_{\perp}}^\perp = \lr{\bar{\nabla}_{~v}~w_{\perp}}^\top+ \nabla^\perp_{~v}~w_{\perp}, \]
    which is related to the second fundamental form via the Weingarten equation: for any $~u\in\Gamma\lr{TM}$,
    \[ \ip{\bar{\nabla}_{~v}~w_{\perp}}{~u} = -\ip{\mathbb{II}(~v,~u)}{~w_{\perp}}. \]
    With this, consider a local orthonormal basis $\{~e_{I}\}_{I=1}^{m} = \{~e_i\}_{i=1}^{m} \cup \{~e_{a}\}_{a=1}^{m-n}$ for $TN$ and its decomposition into local orthonormal bases for $TM$ resp. $\lr{TM}^\perp$.  It follows that the ambient divergence of a vector $~w=~w_{\parallel} + ~w_{\perp}$ decomposes as (Einstein summation assumed)
    \begin{align*}
    \bar{\nabla}\cdot~w &= \ip{~e_i}{\bar{\nabla}_{~e_i}~w} + \ip{~e_a}{\bar{\nabla}_{~e_a}~w} = \ip{~e_i}{\bar{\nabla}_{~e_i}~w_{\parallel}} + \ip{~e_i}{\bar{\nabla}_{~e_i}~w_{\perp}} + ~P^\perp\lr{\bar{\nabla}\cdot}~w \\
    &= \nabla\cdot~w_{\parallel} - \ip{\mathbb{II}(~e_i,~e_i)}{~w_{\perp}} + ~P^\perp\lr{\bar{\nabla}\cdot}~w = \nabla\cdot~w_{\parallel} - \ip{~H}{~w_{\perp}} + ~P^\perp\lr{\bar{\nabla}\cdot}~w,
    \end{align*}
    as claimed.
\end{proof}


\begin{corollary}
    Let $~w\in\Gamma(T\mathcal{M})$ be a vector field on the manifold $\mathcal{M}$ with boundary $\partial{M}$ and affine connection $\nabla$.  Writing $~w = ~w_{\parallel}+w_{\perp}~n$ in terms of the outward-pointing normal $~n$ to $\partial M$, the divergence $\nabla\cdot~w$ on $(M,\nabla)$ restricts to $(\partial M,\nabla_{\partial})$ in the following way:
    \[\lr{\nabla\cdot~w}|_{\partial M} = \nabla_{\partial}\cdot~w_{\parallel} - \ip{\nabla_{~n}~n}{~w_{\parallel}} + \lr{\partial_{~n} - \ip{~H}{~n}}w_{\perp}.\]
\end{corollary}
\begin{proof}
    Observe that 
    \begin{align*} 
    ~P^\perp(\nabla\cdot)~w &= \ip{~n}{\nabla_{~n}~w} = \ip{~n}{\nabla_{~n}~w_{\parallel}} + \ip{~n}{\nabla_{~n}\lr{w_{\perp}~n}} = -\ip{\nabla_{~n}~n}{~w_{\parallel}} + \partial_{~n}w_{\perp},
    \end{align*}
    since $\ip{~n}{~n} = 1$ and therefore $\ip{\nabla_{~n}~n}{~n}=0$.  The result now follows from \Cref{thm:bndry_decomp}.
\end{proof}

\begin{theorem}\label{thm:sobolev}
    Let $L^2(\mathcal{M},T\mathcal{M})$ denote the space of square-integrable vector fields on a manifold $\mathcal{M}$, let $\mathring{H}^1(\mathcal{M}) = \{u\in H^1(\mathcal{M}): \int_{\mathcal{M}}u\,dV = 0\}$ denote the Sobolev space of $H^1(\mathcal{M})$ functions with zero mean, and let $H^{-1}=(\mathring{H}^1)^*$
    denote its dual space of formal divergences.
    Then, the free energy functionals $\mathcal{F}:\mathring{H}^1(\mathcal{M})\to\mathbb{R}$ and $\mathcal{G}:H^{-1}(\mathcal{M})\to\mathbb{R}$ have the Sobolev gradients
    \[
    \mathrm{grad}_{\mathring{H}^1}\,\mathcal{F} = (-\Delta)^{-1}d\mathcal{F},
    \qquad
    \mathrm{grad}_{H^{-1}}\,\mathcal{G} = (-\Delta)\, d\mathcal{G},
    \]
    in terms of the Neumann Laplacian $-\Delta: \mathring{H}^1(\mathcal{M})\to H^{-1}(\mathcal{M})$, its inverse $(-\Delta)^{-1}$, and the functional derivatives $d\mathcal{F}\in H^{-1}(\mathcal{M})$, $d\mathcal{G}\in \mathring{H}^1(\mathcal{M})$.
\end{theorem}
\begin{proof}
    Since functions in $\mathring{H}^1(\mathcal{M})$ have zero mean, the Poincar\'e inequality implies that
    \[
    \ip{u}{v}_{\mathring{H}^1} := \int_{\mathcal{M}} \nabla u\cdot \nabla v\,dV = \langle(-\Delta)u\,|\,v\rangle,
    \]
    defines an inner product on $\mathring{H}^1(\mathcal{M})$ in terms of the Neumann Laplacian
    $
    -\Delta : \mathring{H}^1(\mathcal{M}) \to H^{-1}(\mathcal{M})
    $ and the duality pairing $\langle\cdot\,|\,\cdot\rangle : H^{-1}\times\mathring{H}^1\to\mathbb{R}$. 
    By Riesz representation, this identifies $-\Delta$ with the flat map
    \[
    \flat : \mathring{H}^1(\mathcal{M}) \to H^{-1}(\mathcal{M}),
    \qquad
    u \mapsto u^\flat = (-\Delta)u,
    \]
    whose inverse is given by $
    \sharp = (-\Delta)^{-1}: H^{-1}(\mathcal{M}) \to \mathring{H}^1(\mathcal{M}).
    $
    
    Now, let $\mathcal{F}:\mathring{H}^1(\mathcal{M})\to\mathbb{R}$. 
    For every variation $\mathring{u}\in \mathring{H}^1(\mathcal{M})$, it follows that
    \[
    \dual{d\mathcal{F}(u)}{\mathring{u}}
    =
    \ip{\mathrm{grad}_{\mathring{H}^1}\mathcal{F}(u)}{\mathring{u}}_{\mathring{H}^1}
    =
    \dual{(-\Delta)\,\mathrm{grad}_{\mathring{H}^1}\mathcal{F}(u)}{\mathring{u}}.
    \]
    Since this holds for all $\mathring{u}$, we conclude
    \[
    (-\Delta)\,\mathrm{grad}_{\mathring{H}^1}\mathcal{F}(u) = d\mathcal{F}(u) \qquad \Longrightarrow \qquad    \mathrm{grad}_{\mathring{H}^1}\mathcal{F}(u) = (-\Delta)^{-1}d\mathcal{F}(u).
    \]
    Similarly, for $\mathcal{G}:H^{-1}(\mathcal{M})\to\mathbb{R}$, define the dual inner product in the usual way:
    \[
    \ip{f}{g}_{H^{-1}}
    :=
    \ip{f^\sharp}{g^\sharp}_{\mathring{H}^1}
    \qquad
    \forall f,g\in H^{-1}(\mathcal{M}).
    \]
    Then, for any variation $\mathring{f}\in H^{-1}(\mathcal{M})$,
    \[
    \dual{\mathring{f}}{d\mathcal{G}(f)}
    =
    \ip{\mathring{f}}{\mathrm{grad}_{H^{-1}}\mathcal{G}(f)}_{H^{-1}}
    =
    \ip{\mathring{f}^\sharp}{(\mathrm{grad}_{H^{-1}}\mathcal{G}(f))^\sharp}_{\mathring{H}^1}
    =
    \dual{\mathring{f}}{(\mathrm{grad}_{H^{-1}}\mathcal{G}(f))^\sharp}.
    \]
    Since this holds for all $\mathring{f}$,
    \[
    d\mathcal{G}(f) = (\mathrm{grad}_{H^{-1}}\mathcal{G}(f))^\sharp
    = (-\Delta)^{-1}\mathrm{grad}_{H^{-1}}\mathcal{G}(f) \qquad \Longrightarrow \qquad     \mathrm{grad}_{H^{-1}}\mathcal{G}(f) = (-\Delta)\,d\mathcal{G}(f),
    \]
    establishing the claim.
\end{proof}

\end{appendices}

\bibliographystyle{unsrt} 
\bibliography{references}

@incollection{arroyo2017onsager,
  title={Onsager’s variational principle in soft matter: introduction and application to the dynamics of adsorption of proteins onto fluid membranes},
  author={Arroyo, Marino and Walani, Nikhil and Torres-S{\'a}nchez, Alejandro and Kaurin, Dimitri},
  booktitle={The role of mechanics in the study of lipid bilayers},
  pages={287--332},
  year={2017},
  publisher={Springer}
}

@article{doi2011onsager,
  doi = {10.1088/0953-8984/23/28/284118},
  url = {https://doi.org/10.1088/0953-8984/23/28/284118},
  year = {2011},
  month = {jun},
  publisher = {},
  volume = {23},
  number = {28},
  pages = {284118},
  author = {Doi, Masao},
  title = {Onsager’s variational principle in soft matter},
  journal = {Journal of Physics: Condensed Matter},
}

@article{padhan2025cahn,
  title={The {C}ahn--{H}illiard--{N}avier--{S}tokes framework for multiphase fluid flows: laminar, turbulent and active},
  author={Padhan, Nadia Bihari and Pandit, Rahul},
  journal={Journal of Fluid Mechanics},
  volume={1010},
  pages={P1},
  year={2025},
  publisher={Cambridge University Press}
}

@article{giorgini2020weak,
  title = {Weak and strong solutions to the nonhomogeneous incompressible {N}avier-{S}tokes-{C}ahn-{H}illiard system},
  journal = {Journal de Mathématiques Pures et Appliquées},
  volume = {144},
  pages = {194-249},
  year = {2020},
  issn = {0021-7824},
  doi = {https://doi.org/10.1016/j.matpur.2020.08.009},
  url = {https://www.sciencedirect.com/science/article/pii/S0021782420301495},
  author = {Andrea Giorgini and Roger Temam},
}

@misc{callen1998thermodynamics,
  title={Thermodynamics and an Introduction to Thermostatistics},
  author={Callen, Herbert B and Scott, HL},
  year={1998},
  publisher={American Association of Physics Teachers}
}

@article{otto2001geometry,
  title={The geometry of dissipative evolution equations: the porous medium equation},
  author={Otto, Felix},
  year={2001},
  publisher={Taylor \& Francis}
}

@article{wu2021review,
  title={A review on the {C}ahn-{H}illiard equation: classical results and recent advances in dynamic boundary conditions},
  author={Wu, Hao},
  journal={arXiv preprint arXiv:2112.13812},
  year={2021}
}

@article{chen2025onsager,
  title={The {O}nsager principle and structure preserving numerical schemes},
  author={Chen, Huangxin and Liu, Hailiang and Xu, Xianmin},
  journal={Journal of Computational Physics},
  volume={523},
  pages={113679},
  year={2025},
  publisher={Elsevier}
}

@article{huang2022variational,
  title={Variational {O}nsager Neural Networks ({VONN}s): A thermodynamics-based variational learning strategy for non-equilibrium PDEs},
  author={Huang, Shenglin and He, Zequn and Reina, Celia},
  journal={Journal of the Mechanics and Physics of Solids},
  volume={163},
  pages={104856},
  year={2022},
  publisher={Elsevier}
}

@article{cao2022globally,
  title={A Globally Convergent Modified {N}ewton Method for the Direct Minimization of the {O}hta--{K}awasaki Energy with Application to the Directed Self-Assembly of Diblock Copolymers},
  author={Cao, Lianghao and Ghattas, Omar and Oden, J Tinsley},
  journal={SIAM Journal on Scientific Computing},
  volume={44},
  number={1},
  pages={B51--B79},
  year={2022},
  publisher={SIAM}
}

@article{luo2023optimal,
  title={Optimal design of chemoepitaxial guideposts for the directed self-assembly of block copolymer systems using an inexact {N}ewton algorithm},
  author={Luo, Dingcheng and Cao, Lianghao and Chen, Peng and Ghattas, Omar and Oden, J Tinsley},
  journal={Journal of Computational Physics},
  volume={485},
  pages={112101},
  year={2023},
  publisher={Elsevier}
}

@article{ohta1986equilibrium,
  title={Equilibrium morphology of block copolymer melts},
  author={Ohta, Takao and Kawasaki, Kyozi},
  journal={Macromolecules},
  volume={19},
  number={10},
  pages={2621--2632},
  year={1986},
  publisher={ACS Publications}
}

@book{friedli2018statistical,
  title={Statistical mechanics of lattice systems: a concrete mathematical introduction},
  author={Friedli, Sacha and Velenik, Yvan},
  year={2018},
  publisher={Cambridge University Press}
}

@book{villani2021topics,
  title={Topics in optimal transportation},
  author={Villani, C{\'e}dric},
  volume={58},
  year={2021},
  publisher={American Mathematical Soc.}
}

@article{jordan1998variational,
  title={The variational formulation of the {F}okker--{P}lanck equation},
  author={Jordan, Richard and Kinderlehrer, David and Otto, Felix},
  journal={SIAM journal on mathematical analysis},
  volume={29},
  number={1},
  pages={1--17},
  year={1998},
  publisher={SIAM}
}

@article{blondel2022efficient,
  title={Efficient and modular implicit differentiation},
  author={Blondel, Mathieu and Berthet, Quentin and Cuturi, Marco and Frostig, Roy and Hoyer, Stephan and Llinares-L{\'o}pez, Felipe and Pedregosa, Fabian and Vert, Jean-Philippe},
  journal={Advances in neural information processing systems},
  volume={35},
  pages={5230--5242},
  year={2022}
}

@article{ren2023torchopt,
  title={Torchopt: An efficient library for differentiable optimization},
  author={Ren, Jie and Feng, Xidong and Liu, Bo and Pan, Xuehai and Fu, Yao and Mai, Luo and Yang, Yaodong},
  journal={Journal of Machine Learning Research},
  volume={24},
  number={367},
  pages={1--14},
  year={2023}
}

@article{wu2006guided,
    author = {Wu, Xiang-Fa and Dzenis, Yuris A.},
    title = {Guided self-assembly of diblock copolymer thin films on chemically patterned substrates},
    journal = {The Journal of Chemical Physics},
    volume = {125},
    number = {17},
    pages = {174707},
    year = {2006},
    month = {11},
    issn = {0021-9606},
    doi = {10.1063/1.2363982},
    url = {https://doi.org/10.1063/1.2363982}
}

@article{tsori2001diblock,
    doi = {10.1209/epl/i2001-00211-3},
    url = {https://doi.org/10.1209/epl/i2001-00211-3},
    year = {2001},
    month = {mar},
    publisher = {},
    volume = {53},
    number = {6},
    pages = {722},
    author = {Y. Tsori and D. Andelman},
    title = {Diblock copolymer ordering induced by patterned surfaces},
    journal = {Europhysics Letters}
}

@article{li2020fourier,
  title={Fourier neural operator for parametric partial differential equations},
  author={Li, Zongyi and Kovachki, Nikola and Azizzadenesheli, Kamyar and Liu, Burigede and Bhattacharya, Kaushik and Stuart, Andrew and Anandkumar, Anima},
  journal={arXiv preprint arXiv:2010.08895},
  year={2020}
}

@article{weber2019physics,
    doi = {10.1088/1361-6633/ab052b},
    url = {https://doi.org/10.1088/1361-6633/ab052b},
    year = {2019},
    month = {apr},
    publisher = {IOP Publishing},
    volume = {82},
    number = {6},
    pages = {064601},
    author = {Weber, Christoph A and Zwicker, David and Jülicher, Frank and Lee, Chiu Fan},
    title = {Physics of active emulsions},
    journal = {Reports on Progress in Physics}
}

@article{onsager1931reciprocal,
  title={Reciprocal relations in irreversible processes. I.},
  author={Onsager, Lars},
  journal={Physical review},
  volume={37},
  number={4},
  pages={405},
  year={1931},
  publisher={APS}
}

@article{onsager1931reciprocal2,
  title={Reciprocal relations in irreversible processes. II.},
  author={Onsager, Lars},
  journal={Physical review},
  volume={38},
  number={12},
  pages={2265},
  year={1931},
  publisher={APS}
}

@article{de1993new,
  title={New problems on minimizing movements},
  author={De Giorgi, Ennio},
  journal={Ennio de Giorgi: selected papers},
  pages={699--713},
  year={1993}
}

@article{laux2020thresholding,
  title={The thresholding scheme for mean curvature flow and de Giorgi’s ideas for minimizing movements},
  author={Laux, Tim and Otto, Felix},
  journal={The role of metrics in the theory of partial differential equations},
  volume={85},
  pages={63--94},
  year={2020}
}

@misc{tran2025nonlinear,
  title={Nonlinear Splitting for Gradient-Based Unconstrained and Adjoint Optimization}, 
  author={Brian K. Tran and Ben S. Southworth and David B. Cavender and Sam Olivier and Syed A. Shah and Tommaso Buvoli},
  year={2025},
  eprint={2508.20280},
  archivePrefix={arXiv},
  primaryClass={math.OC},
  url={https://arxiv.org/abs/2508.20280}, 
}

@article{rockafellar1976monotone,
  title={Monotone operators and the proximal point algorithm},
  author={Rockafellar, R Tyrrell},
  journal={SIAM journal on control and optimization},
  volume={14},
  number={5},
  pages={877--898},
  year={1976},
  publisher={SIAM}
}

@article{cai2022developments,
  title={The developments of proximal point algorithms},
  author={Cai, Xing-Ju and Guo, Ke and Jiang, Fan and Wang, Kai and Wu, Zhong-Ming and Han, De-Ren},
  journal={Journal of the Operations Research Society of China},
  volume={10},
  number={2},
  pages={197--239},
  year={2022},
  publisher={Springer}
}

@article{rayleigh1871some,
  author = {Strutt, J. W.},
  title = {Some General Theorems relating to Vibrations},
  journal = {Proceedings of the London Mathematical Society},
  volume = {s1-4},
  number = {1},
  pages = {357-368},
  doi = {https://doi.org/10.1112/plms/s1-4.1.357},
  url = {https://londmathsoc.onlinelibrary.wiley.com/doi/abs/10.1112/plms/s1-4.1.357},
  eprint = {https://londmathsoc.onlinelibrary.wiley.com/doi/pdf/10.1112/plms/s1-4.1.357},
  year = {1871}
}

@article{chen2002phase,
  title={Phase-field models for microstructure evolution},
  author={Chen, Long-Qing},
  journal={Annual review of materials research},
  volume={32},
  number={1},
  pages={113--140},
  year={2002},
  publisher={Annual Reviews 4139 El Camino Way, PO Box 10139, Palo Alto, CA 94303-0139, USA}
}

@article{wu2020phase,
  title={Phase-field modeling of fracture},
  author={Wu, Jian-Ying and Nguyen, Vinh Phu and Nguyen, Chi Thanh and Sutula, Danas and Sinaie, Sina and Bordas, St{\'e}phane PA},
  journal={Advances in applied mechanics},
  volume={53},
  pages={1--183},
  year={2020},
  publisher={Elsevier}
}

@article{shi2013self,
  title={Self-assembly of diblock copolymers under confinement},
  author={Shi, An-Chang and Li, Baohui},
  journal={Soft Matter},
  volume={9},
  number={5},
  pages={1398--1413},
  year={2013},
  publisher={Royal Society of Chemistry}
}

@article{noshay2013block,
  title={Block copolymers: overview and critical survey},
  author={Noshay, Allen and McGrath, James E},
  year={2013},
  publisher={Elsevier}
}

@article{zhu2025stokes,
  author = {Zhu, Cuncheng and Saintillan, David and Chern, Albert},
  journal = {Journal of Fluid Mechanics},
  pages = {R1},
  title = {Stokes flow of an evolving fluid film with arbitrary shape and topology},
  volume = {1003},
  year = {2025}
}

@article{marsden2001discrete,
  title={Discrete mechanics and variational integrators},
  author={Marsden, Jerrold E and West, Matthew},
  journal={Acta numerica},
  volume={10},
  pages={357--514},
  year={2001},
  publisher={Cambridge University Press}
}

@article{yu2021onsager,
  title = {Onsager{N}et: Learning stable and interpretable dynamics using a generalized {O}nsager principle},
  author = {Yu, Haijun and Tian, Xinyuan and E, Weinan and Li, Qianxiao},
  journal = {Phys. Rev. Fluids},
  volume = {6},
  issue = {11},
  pages = {114402},
  numpages = {32},
  year = {2021},
  month = {Nov},
  publisher = {American Physical Society},
  doi = {10.1103/PhysRevFluids.6.114402},
  url = {https://link.aps.org/doi/10.1103/PhysRevFluids.6.114402}
}

@misc{chang2025unsupervised,
      title={Unsupervised operator learning approach for dissipative equations via {O}nsager principle}, 
      author={Zhipeng Chang and Zhenye Wen and Xiaofei Zhao},
      year={2025},
      eprint={2508.07440},
      archivePrefix={arXiv},
      primaryClass={cs.LG},
      url={https://arxiv.org/abs/2508.07440}, 
}

@article{li2023deep,
  title={Deep learning-based computational method for soft matter dynamics: Deep Onsager-Machlup method},
  author={Li, Zhihao and Zou, Boyi and Wang, Haiqin and Su, Jian and Wang, Dong and Xu, Xinpeng},
  journal={arXiv preprint arXiv:2308.14513},
  year={2023}
}

@article{weinan2018deep,
  author = {E, Weinan and Yu, Bing},
  journal = {Communications in Mathematics and Statistics},
  number = {1},
  pages = {1--12},
  title = {The Deep {R}itz Method: A Deep Learning-Based Numerical Algorithm for Solving Variational Problems},
  volume = {6},
  year = {2018}
}

@article{karkar2023module,
  title={Module-wise training of neural networks via the minimizing movement scheme},
  author={Karkar, Skander and Ayed, Ibrahim and de B{\'e}zenac, Emmanuel and Gallinari, Patrick},
  journal={Advances in Neural Information Processing Systems},
  volume={36},
  pages={53126--53145},
  year={2023}
}

@article{rumpf2025hybrid,
  title={A hybrid minimizing movement and neural network approach to {W}illmore flow},
  author={Rumpf, Martin and Sassen, Josua and Smoch, Christoph},
  journal={arXiv preprint arXiv:2502.14656},
  year={2025}
}

@article{alvarez2021optimizing,
  title={Optimizing functionals on the space of probabilities with input convex neural networks},
  author={Alvarez-Melis, David and Schiff, Yair and Mroueh, Youssef},
  journal={arXiv preprint arXiv:2106.00774},
  year={2021}
}

@article{xu2023normalizing,
  title={Normalizing flow neural networks by {JKO} scheme},
  author={Xu, Chen and Cheng, Xiuyuan and Xie, Yao},
  journal={Advances in Neural Information Processing Systems},
  volume={36},
  pages={47379--47405},
  year={2023}
}

@article{hertrich2024importance,
  title={Importance corrected neural {JKO} sampling},
  author={Hertrich, Johannes and Gruhlke, Robert},
  journal={arXiv preprint arXiv:2407.20444},
  year={2024}
}

@inproceedings{bunne2022proximal,
  title={Proximal optimal transport modeling of population dynamics},
  author={Bunne, Charlotte and Papaxanthos, Laetitia and Krause, Andreas and Cuturi, Marco},
  booktitle={International Conference on Artificial Intelligence and Statistics},
  pages={6511--6528},
  year={2022},
  organization={PMLR}
}

@article{mokrov2021large,
  title={Large-scale {W}asserstein gradient flows},
  author={Mokrov, Petr and Korotin, Alexander and Li, Lingxiao and Genevay, Aude and Solomon, Justin M and Burnaev, Evgeny},
  journal={Advances in Neural Information Processing Systems},
  volume={34},
  pages={15243--15256},
  year={2021}
}

@article{harlow1965numerical,
  title={Numerical calculation of time-dependent viscous incompressible flow of fluid with free surface},
  author={Harlow, Francis H and Welch, J Eddie and others},
  journal={Physics of fluids},
  volume={8},
  number={12},
  pages={2182},
  year={1965}
}

@article{arakawa1977computational,
  title={Computational design of the basic dynamical processes of the {UCLA} general circulation model},
  author={Arakawa, Akio and Lamb, Vivian R},
  year={1977}
}

@inproceedings{amos2017input,
  title={Input convex neural networks},
  author={Amos, Brandon and Xu, Lei and Kolter, J Zico},
  booktitle={International conference on machine learning},
  pages={146--155},
  year={2017},
  organization={PMLR}
}

@inproceedings{rezende2015variational,
  title={Variational inference with normalizing flows},
  author={Rezende, Danilo and Mohamed, Shakir},
  booktitle={International conference on machine learning},
  pages={1530--1538},
  year={2015},
  organization={PMLR}
}

@book{marsden1994mathematical,
  title={Mathematical foundations of elasticity},
  author={Marsden, Jerrold E and Hughes, Thomas JR},
  year={1994},
  publisher={Courier Corporation}
}

@article{wang2021onsager,
  title={Onsager's variational principle in active soft matter},
  author={Wang, Haiqin and Qian, Tiezheng and Xu, Xinpeng},
  journal={Soft Matter},
  volume={17},
  number={13},
  pages={3634--3653},
  year={2021},
  publisher={Royal Society of Chemistry}
}

@article{liu2024variational,
  title={A variational discretization method for mean curvature flows by the {O}nsager principle},
  author={Liu, Yihe and Xu, Xianmin},
  journal={arXiv preprint arXiv:2404.11935},
  year={2024}
}

@article{zhou2018dynamics,
  title={Dynamics of viscoelastic filaments based on {O}nsager principle},
  author={Zhou, Jiajia and Doi, Masao},
  journal={Physical Review Fluids},
  volume={3},
  number={8},
  pages={084004},
  year={2018},
  publisher={APS}
}

@article{xu2017hydrodynamic,
  title={Hydrodynamic boundary conditions derived from {O}nsager's variational principle},
  author={Xu, Xinpeng and Qian, Tiezheng},
  journal={Procedia IUTAM},
  volume={20},
  pages={144--151},
  year={2017},
  publisher={Elsevier}
}

@article{du2024thermodynamic,
  title={Thermodynamic variational principle, its connections to the phenomenological laws and its applications to the derivation of microstructural models},
  author={Du, Qiang},
  journal={Materials Genome Engineering Advances},
  volume={2},
  number={2},
  pages={e51},
  year={2024},
  publisher={Wiley Online Library}
}

@article{lin2023onsager,
  title={Onsager’s variational principle for nonreciprocal systems with odd elasticity},
  author={Lin, Li-Shing and Yasuda, Kento and Ishimoto, Kenta and Hosaka, Yuto and Komura, Shigeyuki},
  journal={Journal of the Physical Society of Japan},
  volume={92},
  number={3},
  pages={033001},
  year={2023},
  publisher={The Physical Society of Japan}
}

@article{kou2020energy,
  title={Energy stable and mass conservative numerical method for a generalized hydrodynamic phase-field model with different densities},
  author={Kou, Jisheng and Wang, Xiuhua and Zeng, Meilan and Cai, Jianchao},
  journal={Physics of Fluids},
  volume={32},
  number={11},
  year={2020},
  publisher={AIP Publishing}
}

@article{xiong2025thermodynamically,
  title={A thermodynamically consistent phase-field lattice {B}oltzmann method for two-phase electrohydrodynamic flows},
  author={Xiong, Fang and Wang, Lei and Huang, Jiangxu and Luo, Kang},
  journal={Journal of Scientific Computing},
  volume={103},
  number={1},
  pages={1--32},
  year={2025},
  publisher={Springer}
}

@article{xiao2025moving,
  title = {A moving mesh method for porous medium equation by the {O}nsager variational principle},
  journal = {Journal of Computational Physics},
  volume = {536},
  pages = {114061},
  year = {2025},
  issn = {0021-9991},
  doi = {https://doi.org/10.1016/j.jcp.2025.114061},
  url = {https://www.sciencedirect.com/science/article/pii/S0021999125003444},
  author = {Si Xiao and Xianmin Xu}
}

@article{strang1999discrete,
  author = {Strang, Gilbert},
  title = {The Discrete Cosine Transform},
  journal = {SIAM Review},
  volume = {41},
  number = {1},
  pages = {135-147},
  year = {1999},
  doi = {10.1137/S0036144598336745},
  URL = {https://doi.org/10.1137/S0036144598336745},
  eprint = {https://doi.org/10.1137/S0036144598336745}
}

@article{stuart1989nonlinear,
  title={Nonlinear instability in dissipative finite difference schemes},
  author={Stuart, Andrew},
  journal={SIAM review},
  volume={31},
  number={2},
  pages={191--220},
  year={1989},
  publisher={SIAM}
}

@article{furihata2001stable,
  title={A stable and conservative finite difference scheme for the {C}ahn-{H}illiard equation},
  author={Furihata, Daisuke},
  journal={Numerische Mathematik},
  volume={87},
  number={4},
  pages={675--699},
  year={2001},
  publisher={Springer}
}

@article{tang2019energy,
  title={On energy dissipation theory and numerical stability for time-fractional phase-field equations},
  author={Tang, Tao and Yu, Haijun and Zhou, Tao},
  journal={SIAM Journal on Scientific Computing},
  volume={41},
  number={6},
  pages={A3757--A3778},
  year={2019},
  publisher={SIAM}
}

@misc{lu2026structure,
      title={Structure-Aware Variational Learning of a Class of Generalized Diffusions}, 
      author={Yubin Lu and Xiaofan Li and Chun Liu and Qi Tang and Yiwei Wang},
      year={2026},
      eprint={2604.20188},
      archivePrefix={arXiv},
      primaryClass={cs.LG},
      url={https://arxiv.org/abs/2604.20188}, 
}

@article{mielke2016generalization,
  author  = {Mielke, Alexander and Renger, Dora R.\ M. and Peletier, Mark A.},
  title   = {A generalization of {O}nsager’s reciprocity relations to gradient flows with nonlinear mobility},
  journal = {Journal of Non‐Equilibrium Thermodynamics},
  volume  = {41},
  number  = {2},
  pages   = {141--149},
  year    = {2016},
  doi     = {10.1515/jnet-2015-0073},
  url     = {https://doi.org/10.1515/jnet-2015-0073},
}

@inproceedings{gruber2025efficiently,
  author = {Anthony Gruber and Kookjin Lee and Haksoo Lim and Noseong Park and Nathaniel Trask},
  booktitle = {The Thirteenth International Conference on Learning Representations},
  title = {Efficiently Parameterized Neural Metriplectic Systems},
  year = {2025}
}

@article{celledoni2021structure,
  title={Structure-preserving deep learning},
  author={Celledoni, Elena and Ehrhardt, Matthias J and Etmann, Christian and McLachlan, Robert I and Owren, Brynjulf and Schonlieb, C-B and Sherry, Ferdia},
  journal={European journal of applied mathematics},
  volume={32},
  number={5},
  pages={888--936},
  year={2021},
  publisher={Cambridge University Press}
}

@article{loya2025structure,
  title={Structure-Preserving Neural Ordinary Differential Equations for Stiff Systems},
  author={Loya, Allen Alvarez and Serino, Daniel A and Burby, JW and Tang, Qi},
  journal={arXiv preprint arXiv:2503.01775},
  year={2025}
}

@article{greydanus2019hamiltonian,
  title={Hamiltonian neural networks},
  author={Greydanus, Samuel and Dzamba, Misko and Yosinski, Jason},
  journal={Advances in neural information processing systems},
  volume={32},
  year={2019}
}

@article{hu2024energetic,
  author = {Hu, Ziqing and Liu, Chun and Wang, Yiwei and Xu, Zhiliang},
  title = {Energetic Variational Neural Network Discretizations of Gradient Flows},
  journal = {SIAM Journal on Scientific Computing},
  volume = {46},
  number = {4},
  pages = {A2528-A2556},
  year = {2024},
  doi = {10.1137/22M1529427},
  URL = {https://doi.org/10.1137/22M1529427},
  eprint = {https://doi.org/10.1137/22M1529427}
}

@inproceedings{giorgi1992movimenti,
  author    = {E. D. Giorgi},
  title     = {Movimenti minimizzanti},
  booktitle = {Proceedings of the Conference on Aspetti e problemi della Matematica oggi},
  year      = {1992},
  address   = {Lecce, Italy},
}

\end{document}